\documentclass{article}
\usepackage{amsrefs}
\usepackage{enumitem}
\usepackage{amsmath}
\usepackage{tikz}
\usepackage{amssymb}
\usepackage{amsthm}
\usepackage{color,soul}
\usepackage{soul}
\usepackage{xcolor}
\usepackage{graphicx}
\usepackage{hyperref, multicol}

\usepackage{tikz-cd}
\usepackage{faktor}
\usepackage{ulem}

\usepackage{hyperref}
\newcounter{essaypart} % We keep track of a variable "coursepart" which records which section of the course we're in
\newtheorem{thm}{Theorem}[essaypart] % We define a \begin{thm}...\end{thm} which looks like "Theorem" and is numbered according to coursepart
\newtheorem{prop}[thm]{Proposition} % We define a \begin{prop}...\end{prop} which looks like "Proposition" and shares its numbering with Theorems, above

\newtheorem{corol}[thm]{Corollary}
\newtheorem{defn}[thm]{Definition}
\newtheorem{rmk}[thm]{Remark}
\newtheorem{notn}[thm]{Notation}
\newtheorem{question}[thm]{Question}
\newtheorem*{thm name}{Theorem}
\newtheorem{ex}[thm]{Example}
\newtheorem{innercustomgeneric}{\bf{\customgenericname}}
\providecommand{\customgenericname}{}
\newcommand{\newcustomtheorem}[2]{%
  \newenvironment{#1}[1]
  {%
   \renewcommand\customgenericname{#2}%
   \renewcommand\theinnercustomgeneric{##1}%
   \innercustomgeneric
  }
  {\endinnercustomgeneric}
}

\newcustomtheorem{customthm}{Theorem}

\DeclareMathOperator{\lcm}{lcm}
\newcommand{\Par}{\operatorname{Par}}
\newcommand{\BM}{\operatorname{BM}}
\newcommand{\Hilb}{\operatorname{Hilb}}
\newcommand{\inv}{\operatorname{inv}}

\newcommand{\core}{\operatorname{core}}

\newcommand{\South}{\operatorname{South}}
\newcommand{\East}{\operatorname{East}}
\newcommand{\So}{\operatorname{So}}
\newcommand{\Ea}{\operatorname{Ea}}
\newcommand{\ch}{\operatorname{ch}}
\newcommand{\wind}{\operatorname{wind}}
\newcommand{\middxc}{\operatorname{mid}_{x,c}}

\newcommand{\critplusxc}{\operatorname{crit}_{x,c}^+}

\newcommand{\critminusxc}{\operatorname{crit}_{x,c}^-}
\newcommand{\critpmxc}{\operatorname{crit}_{x,c}^{\pm}}

\newcommand{\divcparts}{_{\square}^{c*}}
\newcommand{\Eout}{{\operatorname{E}_{\operatorname{out}}}}
\newcommand{\Ein}{{\operatorname{E}_{\operatorname{in}}}}
\newcommand{\Sout}{{\operatorname{S}_{\operatorname{out}}}}
\newcommand{\Sin}{{\operatorname{S}_{\operatorname{in}}}}
\newcommand{\Einplus}{\operatorname{E}^+_{\operatorname{in}}}
\newcommand{\Sinplus}{\operatorname{S}^+_{\operatorname{in}}}
\newcommand{\Eoutplus}{\operatorname{E}^+_{\operatorname{out}}}
\newcommand{\Soutplus}{\operatorname{S}^+_{\operatorname{out}}}

\usepackage{graphicx}
\title{A combinatorial proof of Buryak-Feigin-Nakajima}

\author{Eve Vidalis\footnote{School of Mathematics and Statistics, University of Sheffield: epound1@sheffield.ac.uk}}
\date{August 2022}
\begin{document}

\maketitle
\stepcounter{essaypart}
% ABSTRACT
% E-JC papers must include an abstract. The abstract should consist of a
% succinct statement of background followed by a listing of the principal
% new results that are to be found in the paper. The abstract should be
% informative, clear, and as complete as possible. Phrases like
% "we investigate..." or "we study..." should be kept to a minimum in
% favor of "we prove that..."  or "we show that...".  Do not include equation
% numbers, unexpanded citations (such as "[23]"), or any other references
% to things in the paper that are not defined in the abstract. The abstract
% may be distributed without the rest of the  paper so it must be entirely
% self-contained.  Try to include all words and phrases that someone
% might search for when looking for your paper.

\begin{abstract}
     Buryak, Feigin and Nakajima computed a generating function for a family of partition statistics by using the geometry of the $\mathbb{Z}/c\mathbb{Z}$ fixed point sets in the Hilbert scheme of points on $\mathbb{C}^2$.  Loehr and Warrington had already shown how a similar observation by Haiman using the geometry of the Hilbert scheme of points on $\mathbb{C}^2$ could be made purely combinatorial. We extend Loehr and Warrington's techniques to also account for cores and quotients. In particular, we construct a multigraph $M_{r,s,c}$ that is a direct refinement of Loehr and Warrington's multigraphs $M_{r,s}$, retains the relevant partition data, and is preserved by an involution $I_{r,s,c}$ which we use to prove the equidistribution of a family of partition statistics. As a consequence, we obtain a purely combinatorial proof of Buryak, Feigin, and Nakajima's result.
    
    More precisely, we define a family of partition statistics 
    $\left\{h_{x,c}^+, x\in (0,\infty]\right\}$  and give a combinatorial proof that for all $x$ and all positive integers $c$,
    \begin{equation*}
        \sum q^{|\lambda|}t^{h_{x,c}^+(\lambda)}=q^{|\mu|}\prod_{i\geq 1}\frac{1}{(1-q^{ic})^{c-1}}\prod_{j\geq 1}\frac{1}{1-q^{jc}t},
    \end{equation*}
    where the sum ranges over all partitions $\lambda$ with $c$-core $\mu$.
    
\end{abstract}

\section{Introduction}
A \textit{partition} $\lambda$ of a positive integer $n$ is a non-increasing sequence of positive integers $\lambda_1\geq \lambda_2\geq\ldots\geq\lambda_l$ such that $\lambda_1+\cdots+\lambda_l=n.$ We write $|\lambda|=n.$  We represent partitions as \textit{Young diagrams}, informally by drawing $\lambda_i$ unit squares in a row, left to right, starting with a square with bottom left corner $(0,i-1).$ 
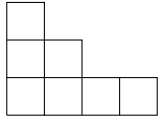
\begin{figure}[ht]
\begin{center}
\begin{tikzpicture}\begin{scope}[yscale=1,xscale=-1,rotate=90,scale=0.5]
\draw (0,0)--(0,4)--(1,4)--(1,0)--(3,0)--(3,1)--(0,1);
\draw (0,0) --(1,0);
\draw (0,2) --(1,2);
\draw (0,3) --(1,3);
\draw (2,0) --(2,1);
\draw (1,2)--(2,2)--(2,1);
\end{scope}\end{tikzpicture}
\caption{The Young diagram for the partition $(4,2,1)$ of $7$.}
\end{center}
\end{figure}

For a square $\square$ in a Young diagram, $a(\square)$ is the number of squares to the right of $\square$ in the same row, and $l(\square)$ is the number of squares above $\square$ in the same column. For example, the square with bottom left corner $(1,0)$ in Figure 1 has $a(\square)=2$, $l(\square)=1.$ We also define $h(\square)=a(\square)+l(\square)+1$ and let
$h_{r,s}(\lambda)$ count the number of squares in the Young diagram of $\lambda$ such that $ (r+s)\mid h(\square)$ and  $rl(\square)=s(a(\square)+1).$

Buryak, Feigin, and Nakajima gave a geometric proof of the following \cite[Corollary 1.3]{BFN} 
\begin{equation}\label{BFN cor}
    \sum_{\lambda \in \Par} q^{h_{r,s}(\lambda)}t^{|\lambda|}=\prod_{\substack{i\geq 1\\ r+s\nmid i}}\frac{1}{1-q^{i}}\prod_{i\geq 1}\frac{1}{1-q^{i(r+s)}t}
\end{equation}
where $\Par$ denotes the set of all partitions. One result of this paper is a purely combinatorial proof of the same result.

We now explain the geometric significance of generating function~\eqref{BFN cor}. The \textit{Hilbert Scheme of n points on $\mathbb{C}^2$}, $\Hilb_n(\mathbb{C}^2)$, parametrises the ideals $I\subset\mathbb{C}[x,y]$ such that $\dim_{\mathbb{C}}\left(\mathbb{C}[x,y]/I\right)=n.$ $\Hilb_n(\mathbb{C}^2)$ admits a torus action by lifting the $\left(\mathbb{C}^*\right)^2$ action on $\mathbb{C}^2$ given by\begin{equation}\label{act on plane}(t_1,t_2)\cdot(x,y)=(t_1x,t_2y)\end{equation} to the action on ideals $I\subset\mathbb{C}[x,y]$ given by
\begin{equation}\label{act on ideals}
(t_1,t_2)\cdot I=\{p(t_1^{-1}x,t_2^{-1}y): p(x,y)\in I\}.
\end{equation}
Let \begin{equation}\Gamma_m=\left\langle\left(e^{\frac{2\pi i}{m}},e^{\frac{-2\pi i}{m}}\right) \right\rangle\end{equation} be a finite subgroup of $\mathbb{C}^2$ of order $m$ and let $T_{r,s}$ be the one-parameter subtorus of $\mathbb{C}^2$ given by \begin{equation}T_{r,s}=\{(t^r,t^s) : t\in\mathbb{C}^*\}.\end{equation}

Let $H_*^{\BM}(X;\mathbb{Q})$ denote the Borel-Moore homology of $X$ with rational coefficients and let \begin{equation}
    P_q^{\BM}(X)=\sum_{i\geq 0} \dim H_i^{\BM}(X;\mathbb{Q})q^{\frac{i}{2}}.
\end{equation}

Buryak, Feigin and Nakajima \cite[Theorem 1.2]{BFN} proved that, if $r,s$ are non-negative integers with $r+s\geq 1,$

\begin{equation}\label{BFN thm}\sum_{n\geq 0}P_q^{\BM}\left(\Hilb_n(\mathbb{C}^2)^{\Gamma_{r+s}\times T_{r,s}}\right)t^n=\prod_{\substack{i\geq 1\\ r+s\nmid i}}\frac{1}{1-q^{i}}\prod_{i\geq 1}\frac{1}{1-q^{i(r+s)}t},\end{equation}

where $\Hilb_n(\mathbb{C}^2)^{T_{r,s}\times \Gamma_{r+s}}$ is the fixed point locus of 
$\Hilb_n(\mathbb{C}^2)$ under the action of $T_{r,s}\times \Gamma_{r+s}.$ The proof is split into two results. One \cite[Lemma 3.1]{BFN} shows that the left hand side of \eqref{BFN thm} is dependent only on $r+s$. The other \cite[Lemma 3.2]{BFN} computes the left hand side of \eqref{BFN thm} in the case $s=0.$ Broadly speaking, Buryak, Feigin, and Nakajima compute the dimension of the Białynicki-Birula cells when the ``slope'' of the acting one parameter torus is very steep, and prove that the slope itself does not affect the eigenspace.

Finally, using the methods of \cite{BB}, a cell decomposition of $\Hilb_n(\mathbb{C}^2)^{T_{r,s}\times \Gamma_{r+s}}$ shows that the left hand side of \eqref{BFN thm} in the Grothendieck ring of varieties is given by 
\begin{equation}
    \sum_{\lambda \in \Par} q^{h_{r,s}(\lambda)}t^{|\lambda|}.
\end{equation}

In \cite{LW}, Loehr and Warrington gave a bijective proof that a partition statistic $h_{x}^+$ is independent of the parameter $x$. In a similar vein to the above, Haiman observed that $h_{x}^+$ accounts for the distribution of the dimension of the Białynicki-Birula cells associated to the action of $(\mathbb{C}^*)^2$ on $\Hilb_n(\mathbb{C}^2),$ i.e. the case when $\Gamma_m$ is the trivial group.

We are interested in
\begin{question}\label{main q}
Is there a bijection proving \eqref{BFN cor}?
\end{question}
To answer this question, we also ask the following.
\begin{question}\label{coresmash}
Can we use Loehr and Warrington's methods to produce a related bijection that preserves the core of a partition?
\end{question}

We provide an affirmative answer to Question~\ref{coresmash}, and use the bijection we produce to provide a partial answer to Question~\ref{main q}. In particular, we define a partition statistic $h_{x,c}^+$ where $x\in[0,\infty)$ and $c$ is a positive integer, and $h_{x,c}^+(\lambda)$ counts the number of squares $\square\in\lambda$ such that both
\begin{itemize} 
    \item the hook length $h(\square)$ is divisible by $c$, and 
    \item  if $a(\square)$ and $l(\square)$ denote the size of the arm and leg of $\square$ respectively,  \begin{equation}
    \frac{a(\square)}{l(\square)+1} \leq x < \frac{a(\square)+1}{l(\square)}.
    \end{equation}
    \end{itemize}
In the case $c=1,$ we recover Loehr and Warrington's statistic $h_x^+$. We then exhibit a bijection proving a refinement (Theorem ~\ref{main theorem}) of \cite[Lemma 3.1]{BFN}. The key ingredient is a bijection at rational slope showing that $h_{x,c}^+$ is equidistributed over partitions with a fixed $c$-core with the statistic $h_{x,c}^-,$ counting boxes $\square$ in the Young diagram such that both \begin{itemize} 
    \item the hook length $h(\square)$ is divisible by $c$, and 
    \item  if $a(\square)$ and $l(\square)$ denote the size of the arm and leg of $\square$ respectively,  \begin{equation}
    \frac{a(\square)}{l(\square)+1} < x \leq \frac{a(\square)+1}{l(\square)}.
    \end{equation}
    \end{itemize}
\begin{customthm}{\ref{h+ is h-}} For all positive rational numbers $x$ and all integers $n\geq 0$, 
$$\sum t^{h_{x,c}^+(\lambda)}=\sum t^{h_{x,c}^-(\lambda)}$$ where both sums range over partitions $\lambda$ of $n$ with a fixed $c$-core $\mu$.
\end{customthm}
To do so, we adapt Loehr and Warrington's construction of a bijection $I_{r,s}$ \cite{LW} to give a new bijection $I_{r,s,c}$ which preserves the $c$-core of a partition and ``picks out'' whether or not $c$ divides the hook length of a cell contributing to a partition statistic.  In the case $c=1$, $I_{r,s,c}$ specialises to $I_{r,s}$. To construct $I_{r,s,c}$, we refine Loehr and Warrington's multigraph $M_{r,s}$ to a multigraph $M_{r,s,c}$ which also sees the $c$-core of a partition. In order to do so, we recast the $c$-abacus construction first introduced in \cite{James} in terms of complete circuits of multigraphs and define an appropriate notion of homomorphism, taking $M_{r,s,c}$ to be the product of the $c$-abacus and $M_{r,s}$ with respect to these homomorphisms. 

We then give a combinatorial proof of a result (Theorem~\ref{basic partition step}), computing the distribution of $h_{0,c}^+.$ This result in particular implies \cite[Lemma 3.2]{BFN}. Whilst our proof is combinatorial, it is not bijective, as we use a multi-counting argument. The map we define was previously defined by Walsh and Waarnar \cite[\S6]{Walsh}.

\begin{customthm}{\ref{basic partition step}}\label{basecaseintro} For all $x$ in $[0,\infty),$
$$\sum q^{|\lambda|}t^{h_{0,c}^+(\lambda)}=q^{|\mu|}\prod_{i\geq 1}\frac{1}{(1-q^{ic})^{c-1}}\prod_{j\geq 1}\frac{1}{1-q^{jc}t}$$
where the sum ranges over all partitions $\lambda$ with $c$-core $\mu$, henceforth denoted $\Par^c_{\mu}.$
\end{customthm}

Finally, our main theorem (Theorem~\ref{main theorem}) uses both Theorem~\ref{basic partition step} and the bijection $I_{r,s,c}$ to compute the following distribution, and we explain how \eqref{BFN cor} follows.

\begin{customthm}{\ref{main theorem}} For all $x$ in $[0,\infty),$
$$\sum q^{|\lambda|}t^{h_{x,c}^+(\lambda)}=q^{|\mu|}\prod_{i\geq 1}\frac{1}{(1-q^{ic})^{c-1}}\prod_{j\geq1}\frac{1}{1-q^{jc}t}$$ where the sum is taken over all partitions $\lambda\in\Par^c_{\mu}$.
\end{customthm}

\subsection{Organisation of the paper}
Section 2 recalls some definitions from partition combinatorics. In particular, we recall the abacus construction (the standard reference for this is \cite[\S2.7]{JK}) and recall some basic generating functions. The section builds up to proving Theorem~\ref{basic partition step}, which uses a bijection introduced in \cite{Walsh} to compute the distribution of $h_{0,c}^+$ over $\Par^c_{\mu},$ the set of partitions with $c$-core $\mu$.
%\begin{equation}\label{base case}\sum_{\lambda\in\Par_{\mu}^c} q^{|\lambda|}t^{\lambda\divcparts}=q^{|\mu|}\prod_{i\geq 1}\frac{1}{(1-q^{ic})^{c-1}}\prod_{j\geq1}\frac{1}{1-q^{jc}t}.\end{equation} Our main theorem has two parameters: the real number $x$ and positive integer $c$. Theorem~\ref{basic partition step} is the $x=0$ case.

Section 3 defines the main partition statistics of interest, $\middxc,$ $\critminusxc,$ $\critplusxc,$ $h^+_{x,c}$ and $h^{-}_{x,c}$ where $h^{\pm}_{x,c}=\middxc+\critpmxc$.
Then, we introduce our main theorem, Theorem~\ref{main theorem}.
%which states that for all $x\in[0,\infty)$,
%\begin{equation}\sum_{\lambda\in\Par^c_{\mu}}q^{|\lambda|}t^{h_{x,c}^+(\lambda)}=q^{|\mu|}\prod_{i\geq 1}\frac{1}{(1-q^{ic})^{c-1}}\prod_{j\geq1}\frac{1}{1-q^{jc}t}.\end{equation} 
In view of Theorem~\ref{basic partition step}, it remains to prove that the left hand side is independent of $x$. An argument analogous to that in ~\cite{LW} is then used to show that the independence of the left hand side from $x$ is implied by a symmetry property when $x$ is rational, \begin{equation}
    \sum_{\lambda\in\Par^c_\mu}q^{|\lambda|}w^{h^+_{x,c}(\lambda)}y^{h_{x,c}^-(\lambda)}=\sum_{\lambda\in\Par^c_\mu}q^{|\lambda|}w^{h^{-}_{x,c}(\lambda)}y^{h_{x,c}^+(\lambda)}.
\end{equation}
We use this to give a set of criteria that constitute a sufficient condition for a bijection to prove Theorem~\ref{main theorem} in Proposition~\ref{bijection properties}.
Finally, the section concludes with a proof that the main result of \cite{BFN} is a consequence of Theorem~\ref{main theorem}.

Section 4 defines the multigraph $M_{r,s,c}(\lambda)$ corresponding to a rational $x=\frac{r}{s}$ and positive integer $c$, defines an ordering $<_{r,s,c}$ on partitions and multigraphs, and a special set of partitions $\lambda_{r,s,k}.$ It then goes on to outline the structure of our proofs that $M_{r,s,c}$ remembers partition data. Our proof is structured somewhat differently to Loehr and Warrington's proofs that $M_{r,s}$ remembers partition data in \cite{LW}. In particular, we do not prove formulae in terms of $M_{r,s,c}$ for any partition statistics except for $\critplusxc +\critminusxc$. Instead, the section works towards providing an inductive framework to prove that $M_{r,s,c}$ remembers partition data by studying how taking successor at the level of partitions and multigraphs are related, culminating in Proposition~\ref{acc points are useful}. One result of this section (Proposition~\ref{accumulation point}) is that the map $\lambda\mapsto M_{r,s,c}(\lambda)$ is injective at the $\lambda_{r,s,k},$ so the map does not lose any data at all at these points, allowing the $\lambda_{r,s,k}$ form a family of base cases. Having outlined the key principles behind the proofs, we then defer the technical checks to Section 6. 

Section 5 defines involutions $I_{r,s,c}:\Par^c_{\mu}\to \Par^c_{\mu}$ that preserve multigraphs $M_{r,s,c}(\lambda)$.

Section 6 studies how each statistic of interest in Proposition~\ref{bijection properties} changes when taking successor with respect to the ordering $<_{r,s,c}$, in particular using Proposition~\ref{acc points are useful} to prove that the map $\lambda\mapsto M_{r,s,c}(\lambda)$ remembers the statistics $\middxc(\lambda)$ and $\critplusxc+\critminusxc$. It also proves that $I_{r,s,c}$ exchanges the statistics $\critplusxc$ and $\critminusxc$. Together with the results of Section 4, this completes a combinatorial proof of Theorem~\ref{main theorem}.

\section{Background: partitions, cores, quotients}
%%\stepcounter{essaypart}
In this section, we recall first definitions in partition combinatorics, including the abacus construction, cores and quotients. The standard reference for the abacus construction is \cite[\S 2.7]{JK}, the abacus was first introduced in \cite{James}, cores in \cite{Nakayama} and quotients in \cite{Littlewood}. We take a nonstandard view of the $c$-core, and describe it as an equivalence class of complete circuits of a directed multigraph $M_c$. The language we use to describe the abacus is also nonstandard, but the construction is equivalent. We take this approach so that we have descriptions of Loehr and Warrington's construction in \cite{LW} and the $c$-core in terms of directed multigraphs, which allows us to formulate a simultaneous refinement of the two in Section 4.  Once we have recalled this theory, we will recall a few standard generating functions and define the map $G_c$ previously defined in \cite{Walsh} and use these to give a combinatorial proof of Theorem~\ref{basic partition step}, which forms our base case.
%\stepcounter{essaypart}
\begin{defn}[Partition,Young diagram] A \textit{partition} of an integer $n\geq 0$ is a sequence of non-increasing positive integers $\lambda_1\geq\lambda_2\geq\ldots\geq \lambda_t$ with sum $n$. The \textit{size} of $\lambda$, denoted $|\lambda|,$ is $n$ and the \textit{length} of $\lambda$ is the number of summands, written $l(\lambda)=t.$ The \textit{Young diagram} of $\lambda$ consists of $t$ rows of $1\times 1$ boxes $\square$ in $\mathbb{R}^2$, with $\lambda_i$ boxes in the $i$th row for each $1\leq i\leq t$. The bottom left corner of the diagram sits at $(0,0)$.
\end{defn}

\begin{ex} The partition $\mu=(12,12,10,8,7,4,1,1,1)$ of $56$ has the diagram given in Figure~\ref{first partition}.
\end{ex}

Informally, the boundary of a partition $\lambda$ is the bi-infinite path traversing the $y$-axis from $+\infty$ until it hits a box of the partition, then follows the edge of the Young diagram until it hits the $x$-axis, before traversing the $x$-axis to $+\infty$. We split the boundary up into unit steps between lattice points, and view it as a directed multigraph where edges are additionally assigned a label indicating if they are south or east.

\begin{defn}[SE directed multigraph]
A \textit{SE directed multigraph} $M=(V,E,s,t,d)$ consists of a vertex set $V$, an edge set $E$, and three maps $s:E\to V,$ $t:E\to V$ and $d:E\to\{\South,\East\},$ called source, target, and direction respectively. We say the edge $e$ \textit{departs from} the vertex $v$ if $s(e)=v$ and we say that $e$ \textit{arrives at} the vertex $w$ if $t(e)=w.$ We call $e$ a \textit{south edge} if $d(e)=\South$ and an \textit{east edge} if $d(e)=\East$. We sometimes abbreviate $\South$ to $S$ and $\East$ to $E$ in contexts where there is no danger of confusion with the edge set.
\end{defn}

\begin{defn}[Boundary graph] The \textit{boundary graph} $b(\lambda)$ of a partition $\lambda$ is an SE directed multigraph. The edge set is defined as follows. For natural numbers $x,y$ there is a south edge $e$ with $s(e)=(x,y+1),$ $t(e)=(x,y)$ if either
\begin{itemize}
\item $x=0$ and $y\geq l(\lambda)$, or 
\item $x>0$ and $\lambda_{y+1}=x.$
\end{itemize}
There is an east edge $e$ with $s(e)=(x,y)$ and $t(e)=(x+1,y)$ if either
\begin{itemize}
\item $y=0$ and $x\geq \lambda_1$, or 
\item $y>0$ and $\lambda_{y+1}\leq x<\lambda_y .$
\end{itemize}
The vertex set $V(b(\lambda))$ is the union of sources and targets of the edges. 
\end{defn}

\begin{ex}
Let $\mu=(12,12,10,8,7,4,1,1,1)$. The boundary graph of $\mu$ is given in Figure~\ref{first partition}, the south edges being the downward arrows and the east edges being the rightward arrows.

\begin{figure}[ht]
\begin{center}
\begin{tikzpicture}\begin{scope}[yscale=1,xscale=-1,rotate=90,scale=0.4]
\draw (0,0)--(0,12)--(2,12)--(2,0)--(0,0);
\draw (1,0)--(1,12);
\draw (3,0)--(3,10);
\draw (4,0)--(4,8);
\draw (5,0)--(5,7);
\draw (6,0)--(6,4);
\draw (7,0)--(7,1);
\draw (8,0)--(8,1);
\draw (9,0)--(9,1);
      \foreach \x in {1,...,11} {      
        \draw (0,\x)--(2,\x); }
              \foreach \x in {0,...,10} {      
        \draw (2,\x)--(3,\x); }
              \foreach \x in {0,...,8} {      
        \draw (3,\x)--(4,\x); }
              \foreach \x in {0,...,7} {      
        \draw (4,\x)--(5,\x); }
              \foreach \x in {0,...,4} {      
        \draw (5,\x)--(6,\x); }
              \foreach \x in {0,...,1} {      
        \draw (6,\x)--(9,\x); }
         \foreach \x in {10,...,15} {      
        \draw[red, ->] (25-\x,0)--(24.5-\x,0); 
        \draw[red] (24.5-\x,0)--(24-\x,0);}
                \foreach \x in {12,...,14} {      
        \draw[red, ->] (0,\x)--(0,\x+0.5); 
        \draw[red] (0,\x+0.5)--(0,\x+1);}
        
       \draw[red,->] (9,0)--(9,0.5);
       \draw[red,->] (9,1)--(8.5,1);
       \draw[red,->] (8,1)--(7.5,1);       
       \draw[red,->] (7,1)--(6.5,1);
       \draw[red,->] (6,1)--(6,1.5);
       \draw[red,->] (6,2)--(6,2.5);       
       \draw[red,->] (6,3)--(6,3.5);
       \draw[red,->] (6,4)--(5.5,4);
       \draw[red,->] (5,4)--(5,4.5);       
       \draw[red,->] (5,5)--(5,5.5);
       \draw[red,->] (5,6)--(5,6.5);
       \draw[red,->] (5,7)--(4.5,7);
       \draw[red,->] (4,7)--(4,7.5);
       \draw[red,->] (4,8)--(3.5,8);       
       \draw[red,->] (3,8)--(3,8.5);
       \draw[red,->] (3,9)--(3,9.5);
       \draw[red,->] (3,10)--(2.5,10);
       \draw[red,->] (2,10)--(2,10.5);
       \draw[red,->] (2,11)--(2,11.5);
       \draw[red,->] (2,12)--(1.5,12);       
       \draw[red,->] (1,12)--(0.5,12);   
       
       \draw[red] (9,0.5)--(9,1);
       \draw[red] (8.5,1)--(8,1);
       \draw[red] (7.5,1)--(7,1);       
       \draw[red] (6.5,1)--(6,1);
       \draw[red] (6,1.5)--(6,2);
       \draw[red] (6,2.5)--(6,3);       
       \draw[red] (6,3.5)--(6,4);
       \draw[red] (5.5,4)--(5,4);
       \draw[red] (5,4.5)--(5,5);       
       \draw[red] (5,5.5)--(5,6);
       \draw[red] (5,6.5)--(5,7);
       \draw[red] (4.5,7)--(4,7);
       \draw[red] (4,7.5)--(4,8);
       \draw[red] (3.5,8)--(3,8);       
       \draw[red] (3,8.5)--(3,9);
       \draw[red] (3,9.5)--(3,10);
       \draw[red] (2.5,10)--(2,10);
       \draw[red] (2,10.5)--(2,11);
       \draw[red] (2,11.5)--(2,12);
       \draw[red] (1.5,12)--(1,12);       
       \draw[red] (0.5,12)--(0,12);   

\end{scope}\end{tikzpicture}
\end{center}
\caption{The Young diagram and boundary graph of $(12,12,10,8,7,4,1,1,1)$.}\label{first partition}
\end{figure}
\end{ex}

Note that for any edge $e$ in the boundary graph, the value of $x-y$ at the target of $e$ is one greater than at the source, because taking a unit step south or east increases the value of $x-y$ by 1.

So, the value of $x-y$ at the target of an edge indexes an Eulerian tour, or complete circuit, of $b(\lambda).$ For clarity, we recall the definition of a complete circuit.

\begin{defn}[Complete circuit]
Given a directed multigraph $M$, a \textit{complete circuit} of $M$ is an ordering of $E(M)$ such that if $e_i$ and $e_{i+1}$ are consecutive with respect to the ordering, then there is a vertex $v\in V(M)$ such that $t(e_i)=s(e_{i+1})$.
\end{defn}

\begin{defn}[Boundary tour, boundary sequence, index] If an edge $e\in E(b(\lambda))$ has target $(x,y),$ we say the \textit{index} of $e$ is $i(e)=x-y.$
The \textit{boundary tour} is the complete circuit of $b(\lambda)$ where the edges are ordered by index. We write the edges in this ordering as $(\ldots, e_{-2},e_{-1},e_0,e_1,e_2,\ldots).$ We say an edge $e_j$ \textit{occurs before} the edge $e_k$ if $j<k$. The \textit{boundary sequence} is the bi-infinite sequence $(d_i)_{i\in\mathbb{Z}}$ where $d_i=d(e_i).$ We write $S$ and $E$ in place of South and East respectively in the boundary sequence.
%The boundary sequence of $\lambda$ \textit{aligned at $m$} is the boundary sequence of $\lambda$ with a bar placed between $d_m$ and $d_{m+1}.$
\end{defn}

\begin{ex}
The partition $\mu=(12,12,10,8,7,4,1,1,1)$ has boundary sequence $$\ldots SSSSSSESSSE_0EESEEESESEESEESSEEE\ldots$$
where $d_0$ is indicated with a 0 suffix. \end{ex}

\subsection{Anatomy of a Young Diagram}
\stepcounter{essaypart}
Next, we recall some standard partition statistics and how they relate to the boundary sequence, define rimhooks, and connect to cores. We also introduce the notion of an SE directed multigraph homomorphism.

\begin{defn}[Hand, foot, arm, leg, inversion, hook length] A box $\square\in\lambda$ can be specified by giving the row and column of the Young diagram that the box sits in. In particular, each box in the Young diagram corresponds to a pair of edges: one south, at extreme right of the row $\square$ lies in, called the \textit{hand of $\square$}, and another east, at the top of the column $\square$ lies in, called the \textit{foot of $\square$}, where the foot necessarily occurs before the hand. Conversely, given an east edge $s_1$ departing from $(x_1,y_1)$ and arriving at $(x_1+1,y_1)$ and a south edge $s_2$ departing from $(x_2,y_2)$ and arriving at $(x_2,y_2-1)$ such that $x_1-y_1<x_2-y_2$, there is a unique box $\square$ in the Young diagram with bottom left corner $(x_1,y_2-1)$ such that $s_1$ and $s_2$ are respectively the foot and hand of $\square$. We call such a pair of south and east edges an \textit{inversion}. Hence, we may identify a box in the Young diagram with its hand and foot in the boundary sequence. 

The \textit{arm} of $\square$ consists of the boxes that lie strictly to the right of $\square$ in the same row, and the $\textit{leg}$ of $\square$ consists of the boxes that lie strictly above $\square$ in the same column. We denote the number of boxes in the arm of $\square$ by $a(\square)$ and the number of boxes in the leg of $\square$ by $l(\square)$. The \textit{hook length} of $\square$ is defined to be $ h(\square)=a(\square)+l(\square)+1$. 
\end{defn}

\begin{ex} The boxes in the arm and leg of the shaded box $\square$ in Figure~\ref{armlegdiagram} are labelled with the corresponding body part. The hand of $\square$ is the red arrow, and the foot is the blue arrow, and $a(\square)=5$ and $l(\square)=1$, so $ h(\square)=7$.
\begin{figure}[ht]
\begin{center}
\begin{tikzpicture}
\begin{scope}[yscale=1,xscale=-1,rotate=90,scale=0.4]
\draw (0,0)--(0,12)--(2,12)--(2,0)--(0,0);
\draw (1,0)--(1,12);
\draw (3,0)--(3,10);
\draw (4,0)--(4,8);
\draw (5,0)--(5,7);
\draw (6,0)--(6,4);
\draw (7,0)--(7,1);
\draw (8,0)--(8,1);
\draw (9,0)--(9,1);
      \foreach \x in {1,...,11} {      
        \draw (0,\x)--(2,\x); }
              \foreach \x in {0,...,10} {      
        \draw (2,\x)--(3,\x); }
              \foreach \x in {0,...,8} {      
        \draw (3,\x)--(4,\x); }
              \foreach \x in {0,...,7} {      
        \draw (4,\x)--(5,\x); }
              \foreach \x in {0,...,4} {      
        \draw (5,\x)--(6,\x); }
              \foreach \x in {0,...,1} {      
        \draw (6,\x)--(9,\x); }

       \draw[blue,->] (6,1)--(6,1.5);
       \draw[red,->] (5,7)--(4.5,7);
       \draw[blue] (6,1.5)--(6,2);
       \draw[red] (4.5,7)--(4,7);
       
       \draw[fill=gray] (4,1)--(4,2)--(5,2)--(5,1)--(4,1);
       \node at (4.5,2.5) {\includegraphics[scale=0.035]{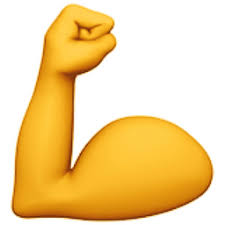}};
       \node at (4.5,3.5) {\includegraphics[scale=0.035]{"arm"}};;
       \node at (4.5,4.5) {\includegraphics[scale=0.035]{"arm"}};;       
       \node at (4.5,5.5) {\includegraphics[scale=0.035]{"arm"}};;
       \node at (4.5,6.5) {\includegraphics[scale=0.035]{"arm"}};;
       \node at (5.5,1.5) {\includegraphics[scale=0.035]{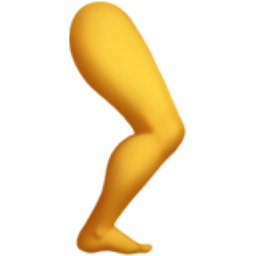}};;
              \draw[blue,->] (6,1)--(6,1.5);
       \draw[red,->] (5,7)--(4.5,7);
       \draw[blue] (6,1.5)--(6,2);
       \draw[red] (4.5,7)--(4,7);
       \end{scope}\end{tikzpicture}
       \end{center}
       \caption{the arm and leg of $\square$.}\label{armlegdiagram}
       \end{figure}
\end{ex}

\begin{prop}\label{c-hooks are c-inversions} Let $\lambda$ be a partition. A box in the Young diagram of $\lambda$ with hook length $c$ corresponds to an inversion $(d_i,d_j)$ in the boundary sequence of $\lambda$ where $j=i+c.$ 
\end{prop}
\begin{proof} Let $h$ and $f$ be the hand and foot of $\square$ in the boundary respectively.
Consider the map from the arm of $\square$ to the boundary sending each box to its foot. The foot of any box in the arm of $\square$ is an east edge that occurs after $f$ and occurs before $h$. Conversely, each east edge that occurs after $f$ and occurs before $h$ is the foot of a box in the arm of $\square$. So, $a(\square)$ counts east edges that occur after $f$ and before $h$.

Analogously, $l(\square)$ counts south edges that occur after $f$ and before $h$. Thus, $a(\square)+l(\square)$ counts the total number of edges that occur after $f$ and before $h$. There are $h(\square)-1$ such edges.  
\end{proof}

We now turn our attention to cores and rimhooks, first introduced by Nakayama \cite{Nakayama}.

\begin{defn}[Rimhook] A \textit{rimhook} $R$ of length $c$ is a connected set of $c$ boxes in $\lambda$ such that removing $R$ gives the Young diagram of a partition, and $R$ does not contain a $2\times 2$ box.
\end{defn}

\begin{corol}\label{rimhookishook} Rimhooks of length $c$ are in bijection with boxes of hook length $c$. 
\end{corol}
\begin{proof} 
Let $R$ be a rimhook of length $c$ in the diagram of a partition $\lambda$. Then, by the definition of a rimhook, for every box $\square\in R$ there is an edge in the boundary graph of $\lambda$ arriving at the top right corner of $\square.$ Let $e_i,e_{i+1},\ldots,e_{i+c-1}$ be the the set of all such edges (since $R$ is connected these edges are consecutive in the boundary tour), and let $e_{i+c}$ be the next edge in the boundary tour. 

Since $R$ is removable, $d(e_i)=E$. 

We now check that $d(e_{i+c})=S$. Since $e_{i+c-1}$ arrives at the top right corner of the south-eastern-most square $\square$ in $R$, $e_{i+c}$ departs from the top right corner of $\square$. If $e_{i+c}$ were an east edge, there would be another box to the right of $\square$ in the same row, contradicting that $R$ is removable. Therefore, by Proposition~\ref{c-hooks are c-inversions}, $e_i$ and $e_{i+c}$ are the foot and hand respectively of a box of hook length $c$. 

Conversely, if $\square$ is a box of hook length $c$, with foot $e_i$ and hand $e_{i+c}$ then taking the boxes with top right corners the targets of $e_i,e_{i+1},\ldots, e_{i+c-1}$ gives a rimhook of length $c$.\end{proof}

\begin{defn} A $c$-core of a partition $\lambda$ is a partition obtained by iteratively removing rimhooks of length $c$ from $\lambda$ until a partition with no rimhooks of length $c$ is obtained. A partition $\mu$ is called a $c$-core if $\mu$ has no rimhooks of length $c$.
\end{defn}

Applying Corollary~\ref{rimhookishook} to $c$-cores gives the following.

\begin{corol} A partition $\lambda$ is a $c$-core if and only if $\lambda$ has no boxes of hook length $c$.
\end{corol}

Our aim for now will be to redefine the $c$-core in the language we wish to use later, and then use it to see that the result of iteratively removing rimhooks of length $c$ is independent of the order in which rimhooks are removed. In order to do so, we need the notion of an SE directed multigraph homomorphism. Informally, these consist of two maps, one between edges, and another between vertices. We require that these maps preserve the direction (S or E) of the edges, and that they be compatible with the source and target maps.

\begin{defn}[SE directed multigraph homomorphism] Let $M_1=(V_1,E_1,s_1,t_1,d_1),$ $M_2=(V_2,E_2,s_2,t_2,d_2)$ be SE directed multigraphs. A homomorphism of SE directed multigraphs $\varphi:M_1\to M_2$ is a pair of maps $\varphi_V:V_1\to V_2$ and $\varphi_E:E_1\to E_2$ such that for all edges $e\in E_1,$ 
\begin{align}
    s_2(\varphi_E(e))&=\varphi_V(s_1(e))\\
    t_2(\varphi_E(e))&=\varphi_V(t_1(e))\\
    d_2(\varphi_E(e))&=d_1(e).
\end{align}
\end{defn}

In other words, $\varphi$ is a quiver homomorphism that preserves direction ($S$ or $E$).

\begin{ex}
Let $M_1$ be the boundary graph of $\mu=(12,12,10,8,7,4,1,1,1)$ and let $\varphi_V$ be the map taking each vertex $(x,y)$ to $[x-y],$ the class of $x-y$ modulo 2. This map induces the homomorphism $q_2$ illustrated in Figure~\ref{fig:2 hom}, with east edges coloured red and south edges coloured blue.

For ease of reading, we draw edges in the image of $q_2$ from left to right in order of index as $\ldots,q_2(e_{-2}), q_2(e_{-1}),q_2(e_0),q_2(e_1),q_2(e_{2}), \ldots$.
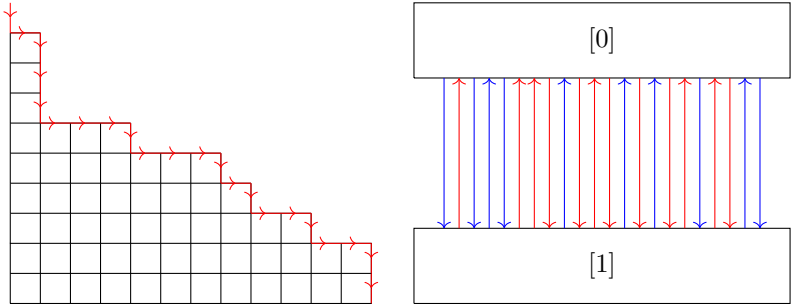
\begin{figure}[ht]
\begin{center}
\begin{tikzpicture}\begin{scope}[yscale=1,xscale=-1,rotate=90,scale=0.4]
\draw (0,0)--(0,12)--(2,12)--(2,0)--(0,0);
\draw (1,0)--(1,12);
\draw (3,0)--(3,10);
\draw (4,0)--(4,8);
\draw (5,0)--(5,7);
\draw (6,0)--(6,4);
\draw (7,0)--(7,1);
\draw (8,0)--(8,1);
\draw (9,0)--(9,1);
      \foreach \x in {1,...,11} {      
        \draw (0,\x)--(2,\x); }
              \foreach \x in {0,...,10} {      
        \draw (2,\x)--(3,\x); }
              \foreach \x in {0,...,8} {      
        \draw (3,\x)--(4,\x); }
              \foreach \x in {0,...,7} {      
        \draw (4,\x)--(5,\x); }
              \foreach \x in {0,...,4} {      
        \draw (5,\x)--(6,\x); }
              \foreach \x in {0,...,1} {      
        \draw (6,\x)--(9,\x); }

       \draw[red,->] (10,0)--(9.5,0);
       \draw[red,->] (9,0)--(9,0.5);
       \draw[red,->] (9,1)--(8.5,1);
       \draw[red,->] (8,1)--(7.5,1);       
       \draw[red,->] (7,1)--(6.5,1);
       \draw[red,->] (6,1)--(6,1.5);
       \draw[red,->] (6,2)--(6,2.5);       
       \draw[red,->] (6,3)--(6,3.5);
       \draw[red,->] (6,4)--(5.5,4);
       \draw[red,->] (5,4)--(5,4.5);       
       \draw[red,->] (5,5)--(5,5.5);
       \draw[red,->] (5,6)--(5,6.5);
       \draw[red,->] (5,7)--(4.5,7);
       \draw[red,->] (4,7)--(4,7.5);
       \draw[red,->] (4,8)--(3.5,8);       
       \draw[red,->] (3,8)--(3,8.5);
       \draw[red,->] (3,9)--(3,9.5);
       \draw[red,->] (3,10)--(2.5,10);
       \draw[red,->] (2,10)--(2,10.5);
       \draw[red,->] (2,11)--(2,11.5);
       \draw[red,->] (2,12)--(1.5,12);       
       \draw[red,->] (1,12)--(0.5,12);   
       
       \draw[red] (9.5,0)--(9,0);
       \draw[red] (9,0.5)--(9,1);
       \draw[red] (8.5,1)--(8,1);
       \draw[red] (7.5,1)--(7,1);       
       \draw[red] (6.5,1)--(6,1);
       \draw[red] (6,1.5)--(6,2);
       \draw[red] (6,2.5)--(6,3);       
       \draw[red] (6,3.5)--(6,4);
       \draw[red] (5.5,4)--(5,4);
       \draw[red] (5,4.5)--(5,5);       
       \draw[red] (5,5.5)--(5,6);
       \draw[red] (5,6.5)--(5,7);
       \draw[red] (4.5,7)--(4,7);
       \draw[red] (4,7.5)--(4,8);
       \draw[red] (3.5,8)--(3,8);       
       \draw[red] (3,8.5)--(3,9);
       \draw[red] (3,9.5)--(3,10);
       \draw[red] (2.5,10)--(2,10);
       \draw[red] (2,10.5)--(2,11);
       \draw[red] (2,11.5)--(2,12);
       \draw[red] (1.5,12)--(1,12);       
       \draw[red] (0.5,12)--(0,12);   

\end{scope}\end{tikzpicture}
\quad\begin{tikzpicture}
\node at (0,3) {$[0]$};
\node at (0,0) {$[1]$};
\draw (-2.5,2.5)--(2.5,2.5)--(2.5,3.5)--(-2.5,3.5)--(-2.5,2.5);
\draw (-2.5,-0.5)--(2.5,-0.5)--(2.5,0.5)--(-2.5,0.5)--(-2.5,-0.5);

\draw[blue,->] (-2.1,2.5)--(-2.1,0.5);
\draw[red,->] (-1.9,0.5)--(-1.9,2.5);
\draw[blue,->] (-1.7,2.5)--(-1.7,0.5);
\draw[blue,->] (-1.5,0.5)--(-1.5,2.5);
\draw[blue,->] (-1.3,2.5)--(-1.3,0.5);
\draw[red,->] (-1.1,0.5)--(-1.1,2.5);
\draw[red,->] (-0.9,0.5)--(-0.9,2.5);
\draw[red,->] (-0.7,2.5)--(-0.7,0.5);
\draw[blue,->] (-0.5,0.5)--(-0.5,2.5);
\draw[red,->] (-0.3,2.5)--(-0.3,0.5);
\draw[red,->] (-0.1,0.5)--(-0.1,2.5);
\draw[red,->] (0.1,2.5)--(0.1,0.5);
\draw[blue,->] (0.3,0.5)--(0.3,2.5);
\draw[red,->] (0.5,2.5)--(0.5,0.5);
\draw[blue,->] (0.7,0.5)--(0.7,2.5);
\draw[red,->] (0.9,2.5)--(0.9,0.5);
\draw[red,->] (1.1,0.5)--(1.1,2.5);
\draw[blue,->] (1.3,2.5)--(1.3,0.5);
\draw[red,->] (1.5,0.5)--(1.5,2.5);
\draw[red,->] (1.7,2.5)--(1.7,0.5);
\draw[blue,->] (1.9,0.5)--(1.9,2.5);
\draw[blue,->] (2.1,2.5)--(2.1,0.5);
\end{tikzpicture}
\end{center}
\caption{A portion of $M_1$ and the corresponding edges in $q_2(M_1)$.}
\label{fig:2 hom}
\end{figure}

\end{ex}

We will always work with SE directed multigraph homomorphisms where the edge map $\varphi_E$ is bijective, so from now on we assume $\varphi_E$ is bijective for any homomorphism $\varphi.$ In particular, this assumption allows us to push complete circuits through homomorphisms.

\begin{prop}
Let $\varphi:M_1\to M_2$ be an SE directed multigraph homomorphism. Let $(e_i)_{i\in I}$ be a complete circuit of $M_1$. Then $(\varphi_E(e_i))_{i\in I}$ is a complete circuit of $M_2$.
\end{prop}
\begin{proof}
Since $\varphi_E$ is bijective, we need only check that $s_2(\varphi_E(e_{i+1})=t_2(\varphi_E(e_{i})$ for each $i\in I$. By definition,\begin{align}
s_2(\varphi_E(e_{i+1})&=\varphi_V(s_1(e_{i+1}))\\
&=\varphi_V(t_1(e_i))\\
&=t_2(\varphi_E(e_i).
\end{align}
\end{proof}

We have seen already that rimhooks of length $c$ correspond to boxes of hook length $c$ which in turn correspond to inversions in the boundary sequence where, if the first term has index $i$, the second has index $i+c$. Intuitively enough, then, the useful homomorphism that captures all of this information is the following.

\begin{defn}[$c$-abacus tour]
Let $(z,w)\sim_c(x,y)$ if $z-w\equiv x-y\pmod{c}$. Then, $q_c:b(\lambda)\to M_c$ is the SE directed multigraph homomorphism induced by imposing the relation $\sim_c$ on the vertices of $b(\lambda)$. The complete circuit $(q_c(e_{i}))_{i\in\mathbb{Z}}$ of $M_c$ is called the $c$-abacus tour associated to $\lambda$. 
\end{defn}

Proposition~\ref{c-hooks are c-inversions} tells us that the number of boxes with hook length divisible by $c$ can be read off from the $c$-abacus tour by looking at edges that correspond to a hand and foot arriving at the same vertex $(v,[i]).$ So, it is sometimes useful to group the edges in a complete circuit by target. This leads us to arrival words.

\begin{defn}[Arrival words, departure words]
Let $M=(V,E,s,t,d)$ be a directed SE multigraph and let $(e_i)_{i\in I}$ be a complete circuit of $M$. For $v\in V$ $I_v\subset I$ be the subset of indices such that $t(e_i)=v.$ The \textit{arrival word at $v$}, written $v_a$, is the sequence of directions $d(e_i)_{i\in I_v}.$ The departure word at $v$ is defined analogously, replacing the target map with the source map.
\end{defn}

\begin{notn} Given a sequence $\left(d_i\right)_{i\in I}$ of $S$s and $E$s, we write $\inv(d_i)$ for the number of inversions. 
\end{notn}

\begin{prop}\label{c-hook count}
Let $\lambda$ be a partition with boundary tour $(e_i)_{i\in\mathbb{Z}}$ and let $M_c$ have vertex set. Then, taking arrival words with respect to the complete circuit $(q_c(e_i))_{i\in\mathbb{Z}}$
\begin{equation}
    |\{\square\in\lambda: c\mid h(\square)\}|=\sum_{i=0}^{c-1}\inv([i]_a).
\end{equation}
\end{prop}
\begin{proof}
Apply Proposition~\ref{c-hooks are c-inversions}.
\end{proof}

\subsection{Alignment and charge}

So far, we have associated to every partition a boundary sequence, a bi-infinite sequence of $S$s and $E$s such that if we travel far enough to the left in the sequence every entry is an $S$, and if we travel far enough to the right, every entry is an $E$. We will now study these sequences in general, and identify which of them arise as boundary sequences of a partition. Then, we will define an equivalence relation on partitions, which we shall show is equivalent to having the same $c$-core. We will use this to show that partitions have a unique $c$-core, to define the $c$-quotients originally studied by \cite{Littlewood}, and to give a bijection between partitions of fixed $c$-core and $c$-tuples of partitions.

\begin{defn}[Charge]\label{kcharge}
Let $D=\left(d_i\right)_{i\in\mathbb{Z}}$ be a bi-infinite sequence with $d_i\in\{S,E\},$ for each $i$ such that for some $M\in\mathbb{N},$ $\forall m\geq M,$ $d_{-m}=S$ and $d_{m}=E$. Fix an integer $k$. Let $e_k$ be the number of $E$s in $\left(d_i\right)_{i\in\mathbb{Z}}$ with index at most $k$, \begin{equation}
    e_k=\left|\left\{d_j: d_j=E\text{ and }j\leq k\right\}\right|.
\end{equation}
Similarly, let $s_k$ be the number of $S$s with index greater than $k$,
\begin{equation}
s_k=\left|\left\{d_j: d_j=S\text{ and }j> k\right\}\right|.
\end{equation}
Then, the $k$-charge of $D$, written $\ch_k\left(D\right)$ is $e_k-s_k-k.$
\end{defn}

\begin{prop}
If $k$ and $l$ are integers, and $D$ is as in Definition~\ref{kcharge}, then $\ch_k(D)=\ch_l(D).$
\end{prop}

\begin{proof}
We check that $\ch_{k+1}(D)=\ch_k(D).$ The proposition then follows by repeated application of the equality. Suppose $d_{k+1}=E.$ Then, $e_{k+1}=e_k+1$ and $s_{k+1}=s_k$. So, \begin{align}\ch_{k+1}(D)&=e_{k+1}-s_{k+1}-(k+1)\\
&=e_k+1-s_k-(k+1)\\
&=e_k-s_k-k\\
&=\ch_k(D).
\end{align}
Similarly, if $d_{k+1}=S,$ then $e_{k+1}=e_k$ and $s_{k+1}=s_k-1$,
 so $\ch_{k+1}(D)=\ch_k(D)$. Therefore, $\ch_k(D)$ is independent of $k$.
 \end{proof}

So, in place of $\ch_k(D),$ we may simply write $\ch(D)$.

\begin{prop}
A sequence $D$ as in Definition~\ref{kcharge} is the boundary sequence of a partition if and only if $\ch(D)=0.$
\end{prop}
\begin{proof}
Suppose $D$ is the boundary sequence of a partition. Let $(x_1,y_1)$ be the point on the line $x-y=k$ on the boundary of a partition $\lambda.$ Since $x_1$ counts the number of south edges with index greater than $k$, and $y_1$ counts the number of east edges with index at most $k$, $\ch(D)=x_1-y_1-k=0.$

If $\ch(D)=0,$ then we may reconstruct $\lambda$ from $D$ by placing a point at $(e_k,s_k)$, and drawing the partition boundary in two halves: one as an infinite path departing from $(e_k,s_k)$ taking unit steps with orientations given by $(d_i)_{i>k}$ and the other as an infinite path arriving at $(e_k,s_k)$ taking unit steps with orientations given by $(d_i)_{i\leq k}.$
\end{proof}

\begin{defn}[The relation $\sim_c$]
Let $\lambda$ and $\mu$ be partitions and let the arrival words taken from the $c$-abacus tours of $\lambda$ and $\mu$ be $[0]^{\lambda}_a,\ldots,[c-1]^{\lambda}_a$, and $[0]^{\mu}_a,\ldots,[c-1]^{\mu}_a$, respectively. Define the relation $\lambda\sim_c\mu$ if, for all $i$ with $0\leq i\leq c-1,$ \begin{equation}
    \ch([i]^{\lambda}_a)=\ch([i]^{\mu}_a).
\end{equation}
\end{defn}

\begin{ex}\label{charges of mu}
Let $\mu=(12,12,10,8,7,4,1,1,1)$ and refer to Figure~\ref{fig:2 hom}. When $c=2,$ $[0]^{\mu}_a$ is given by

\begin{center}
\begin{tikzpicture}
\node at (-0.25,0.4) {$\cdots \text{  }\text{  }S\text{  }\text{  }S\text{  }\text{  }S\text{  }\text{  }E\text{  }\text{  }S\text{  }\text{  }E\text{  }\text{  }E\text{  }\text{  }E\text{  }\mid E\text{  }\text{  }E\text{  }\text{  }E\text{  }\text{  }S\text{  }\text{  }E\text{  }\text{  }S\text{  }\text{  }E\text{  }\text{  }E\cdots$.};
\end{tikzpicture}
\end{center}

where the bar separates terms corresponding to edges of negative or zero index from those of positive index.

So, $\ch([0]^{\mu}_a)=4-2=2.$ Analogously, $[1]^{\mu}_a$ is 

\begin{center}
\begin{tikzpicture}
\node at (0,0) {$\cdots\text{  }S\text{  }\text{  }S\text{  }\text{  }S\text{  }\text{  }S\text{  }\text{  }S\text{  }\text{  }S\text{  }\text{  }E\text{  }\text{  }S\text{  } \mid E\text{  }\text{  }S\text{  }\text{  }S\text{  }\text{  }E\text{  }\text{  }E\text{  }\text{  }S\text{  }\text{  }E\text{  }\text{  }E\text{  }\cdots$.};
\end{tikzpicture}
\end{center}

So, $\ch([1]^{\mu}_a)=1-3=-2.$

\end{ex}

\begin{prop}\label{removing rimhooks}
If $\lambda$ is a partition containing a rimhook $R$ of length $c$ and $\lambda'$ is the partition obtained from $\lambda$ by removing $R$, then $\lambda\sim_c\lambda'$.
\end{prop}
\begin{proof} Let the boundary tours of $\lambda$ and $\lambda'$ be $(e_i)_{i\in\mathbb{Z}}$ and $(e_i')_{i\in\mathbb{Z}}$. First, we analyse how the boundary sequences $(d(e_i))$ and $(d(e_i'))$ differ.
Let $R$ have south-eastern most box $\square_2$ and north-western most box $\square_1$. Let $e_j$ be the east edge traversing the top edge of $\square_1$, so that $e_{j+c}$ is the south edge traversing the right of $\square_2.$ 

\begin{center}
\begin{tikzpicture}\begin{scope}[yscale=1,xscale=-1,rotate=90,scale=0.5]
\draw[yellow, fill=yellow] (3,6)--(6,6)--(6,4)--(8,4)--(8,2)--(7,2)--(7,3)--(5,3)--(5,5)--(3,5)--(3,6);
\draw[dashed] (3,6)--(3,5)--(5,5)--(5,3)--(7,3)--(7,2)--(8,2);
\draw (0,6)--(6,6)--(6,4)--(8,4)--(8,2)--(9,2)--(9,1)--(10,1)--(10,0);
\draw[->] (3.6,6)--(3.5,6);
\draw[->] (3,5.4)--(3,5.5);
\draw[->] (7.6,2)--(7.5,2);
\draw[->] (8,2.4)--(8,2.5);
\node[right] at (3.5,6) {$e_{j+c}$};
\node[below] at (3,5.2) {$e'_{j+c}$};
\node[above] at (8,2.4) {$e_{j}$};
\node[left] at (7.5,2) {$e'_{j}$};
\end{scope}\end{tikzpicture}\quad
\begin{tikzpicture}\begin{scope}[yscale=1,xscale=-1,rotate=90,scale=0.5]
\draw[yellow, fill=yellow] (3,6)--(6,6)--(6,4)--(8,4)--(8,2)--(7,2)--(7,3)--(5,3)--(5,5)--(3,5)--(3,6);
\draw (0,6)--(3,6)--(3,5)--(5,5)--(5,3)--(7,3)--(7,2)--(8,2)--(9,2)--(9,1)--(10,1)--(10,0);
\end{scope}\end{tikzpicture}
\end{center} 

Since we remove $\square_1$ and $\square_2$, $d(e_j)=E,$ $d(e_j')=S$, and $d(e_{j+c})=S$ and $d(e_{j+c}')=E$. Let $j=qc+r$ for $0\leq r\leq c-1$. Since the rimhook does not contain a $2\times 2$ box and is connected, the portion of the boundary of $\lambda'$ between the lines $x-y=j+1$ and $x-y=c+j-1$ is a translate of the original partition boundary by $(-1,-1)$, so $d(e_i)=d(e_i')$ for all $i\not\in\{j,j+c\}.$ So, for all $0\leq s\leq c-1$ with $s\not=r$, $[s]^{\lambda}_a=[s]^{\lambda'}_a,$ and the arrival word
\begin{equation}([r]^{\lambda'}_a)_i=\left\{
\begin{array}{cc}
([r]^{\lambda}_a)_{q}    & i=q+1 \\
 ([r]^{\lambda}_a)_{q+1}    & i=q  \\
 ([r]^{\lambda}_a)_i & \text{otherwise.}
\end{array}\right.
\end{equation}
So, \begin{align}\ch([r]^{\lambda'}_a)&=\ch_{q-1}([r]^{\lambda'}_a)\\
&=\ch_{q-1}([r]^{\lambda}_a)\\
&=\ch([r]^{\lambda}_a).
\end{align}

\end{proof}

\begin{corol}
The $c$-core of $\lambda$ is unique, and $\lambda\sim_c\mu$ if and only if $\lambda$ and $\mu$ have the same $c$-core.
\end{corol}
\begin{proof}
If $\lambda$ has $c$-core $\nu$, then $\nu$ is obtained from $\lambda$ by iteratively removing rimhooks of length $c$ from $R$, so by Proposition~\ref{removing rimhooks}, $\lambda\sim_c \nu.$ Every partition has at least one $c$-core, so it remains to check that if $\mu$ and $\nu$ are both $c$-cores with $\mu\sim_c\nu$ then $\mu=\nu.$ By Propositions~\ref{c-hooks are c-inversions} and~\ref{c-hook count}, if $\mu$ and $\nu$ are both $c$-cores then for each $i$, the arrival words $[i]^\mu_a$ and $[i]^\nu_a$ do not contain any inversions. So, both consist of a string of $S$s up to some index, and a string of $E$s thereafter. Since $\mu\sim_c\nu$, the charge of both $[i]^\mu_a$ and $[i]^\nu_a$ must be the same, and therefore $[i]^\mu_a=[i]^\nu_a$.
\end{proof}

The important consequence for us will be the following.
\begin{corol}\label{keypreservecore}
Let $\lambda$ and $\mu$ be partitions. Then $\lambda$ and $\mu$ have the same $c$-core if there is a value of $m$ with $c\mid m$ such that for each $[i],$ both of the following hold.

\begin{itemize}
    \item the arrival words in the $c$-abacus tour of $\lambda$ and $\mu$ agree after the entry with index $m$;
    \item the portion of $[i]^\lambda_a$ with index at most $m$ is a permutation of the portion of $[i]^\mu_a$ with index at most $m$.
\end{itemize}
\end{corol}

\begin{ex}
We will calculate the 2-core $\lambda$ of  $\mu=(12,12,10,8,7,4,1,1,1)$. By Corollary~\ref{keypreservecore} and the calculation in Example~\ref{charges of mu}, $\lambda$ is the unique $2$-core with $\ch([0]^{\lambda}_a)=2$ and $\ch([1]^{\lambda}_a)=-2.$ 

So, placing a bar in the bi-infinite string with no inversions to separate edges with positive index from those with negative or 0 index, the arrival words in $M_2(\lambda)$ at 0 and 1 respectively, are
\begin{center}
\begin{tikzpicture}\node at (0,0) {$\cdots\text{  }S\text{  }\text{  }S\text{  }\text{  }S\text{  }\text{  }S\text{  }\text{  }S\text{  }\text{  }S\text{  }\text{  }S\text{  }\text{  }S\text{  }\text{  }\mid S\text{  }\text{  }S\text{  }\text{  }E\text{  }\text{  }E\text{  }\text{  }E\text{  }\text{  }E\text{  }\text{  }E\text{  }\text{  }E\text{  }\cdots$};

\node at (0.2,-0.4) {$\cdots \text{  }\text{  }S\text{  }\text{  }S\text{  }\text{  }S\text{  }\text{  }S\text{  }\text{  }S\text{  }\text{  }S\text{  }\text{  }E\text{  }\text{  }E\text{  } \text{  }\mid E\text{  }\text{  }E\text{  }\text{  }E\text{  }\text{  }E\text{  }\text{  }E\text{  }\text{  }E\text{  }\text{  }E\text{  }\text{  }E\cdots$.};
\end{tikzpicture}
\end{center}

So, the 2-core is (3,2,1).

\begin{center}
\begin{tikzpicture}\begin{scope}[yscale=1,xscale=-1,rotate=90,scale=0.5]
\draw (0,3)--(1,3)--(1,2)--(2,2)--(2,1)--(3,1)--(3,0);
\draw (0,2)--(1,2)--(1,0);
\draw (0,1)--(2,1)--(2,0);
\draw (0,3)--(0,0)--(3,0);
\end{scope}\end{tikzpicture}
\end{center}

\end{ex}

\begin{prop}\label{corelattice} There is an bijective map $f$ from $c$-core partitions to a $\mathbb{Z}$-module of length $c-1.$
\begin{proof}
 Consider the $c$-abacus of a $c$-core partition. The charges $(\ch([0]),\ldots,\ch([c-1]))$ specify the $c$-core. A $c$-tuple of integers $(a_0,\ldots,a_{c-1})$ represents the charges of a partition if and only if $\sum_{i=0}^{c-1}a_i=0.$ So, sending a $c$-core to the $c$-tuple of charges gives a bijective map with the $\mathbb{Z}$-module $M = \langle e_1,\ldots, e_c: \sum_{i=0}^{c-1}e_i=0\rangle$.
\end{proof}

\end{prop}
Fix a positive integer $c$, a $c$-core $\mu$, and a non-negative integer $n$. Let $\Par^c_\mu(n)$ denote the set of partitions of $E$ with $c$-core $\mu$. Let $\Par^c_\mu$ denote the set of all partitions with $c$-core $\mu$, and let $\Par$ denote the set of all partitions.

\begin{defn}[Quotient] The $c$-quotient of $\lambda$ is the $c$-tuple of partitions given by $(q_1(\lambda),q_2(\lambda),\ldots,q_c(\lambda))$, where $q_i(\lambda)$ is the partition with boundary sequence $[i]_a$, with the index shifted so that the charge is 0.
\end{defn}

\begin{defn}[Quotient map]
The quotient map $\phi:\Par^c_{\mu}\to(\Par)^c$ sends $\lambda$ to $(q_1(\lambda),\ldots,q_c(\lambda).$
\end{defn}
\begin{prop}\label{coresandquotients}
For $\lambda\in\Par^c_{\mu},$ \begin{equation} |\lambda|=|\mu|+c\sum_{i=1}^c|q_i(\lambda)|.\end{equation}
\end{prop}
\begin{proof}
By Proposition~\ref{c-hook count}, the number of boxes with hook length divisible by $c$ are given by $\sum_{i=1}^c\inv[i]_a$. Starting from the $c$-abacus tour of $\mu$, we can obtain the $c$-abacus tour of $\lambda$ by adding these inversions one at a time. Adding each inversion corresponds to adding a rimhook of length $c$ to the diagram, so contributes $c$ to $|\lambda|$.
\end{proof}

\subsection{The map \texorpdfstring{$G_c$}{Gc}}
Now we set about proving Theorem~\ref{basic partition step}. We first recall three standard generating functions.
\begin{prop}
 \begin{equation}\label{partitionswithparts} \sum_{\lambda\in\Par}q^{|\lambda|}t^{l(\lambda)}=\prod_{m\geq 1}\frac{1}{1-q^mt}
    \end{equation}
 \begin{equation}\label{justpartitions} \sum_{\lambda\in\Par}q^{|\lambda|}=\prod_{m\geq 1}\frac{1}{1-q^m}
    \end{equation}
 \begin{equation}\label{coresgen} \sum_{\lambda\in\Par^c_{\mu}}q^{|\lambda|}=q^{|\mu|}\prod_{m\geq 1}\frac{1}{(1-q^{mc})^c}
    \end{equation}

\end{prop}
\begin{proof}
We may rewrite the right hand side of~\eqref{partitionswithparts} as $$\prod_{m\geq 1} 1+q^mt+q^{2m}t^2+q^{3m}t^3+\ldots,$$
so that picking a term $q^{km}t^k$ for each $m$ corresponds to declaring that $\lambda$ contains $k$ parts of size $m$, contributing $|km|$ to $\lambda$ and $k$ to $l(\lambda)$, giving the left hand side. Setting $t=1$ in~\eqref{partitionswithparts} gives~\eqref{justpartitions}.  

For~\eqref{coresgen}, Proposition~\ref{coresandquotients} tells us that the map $\phi$ gives a bijection between $\lambda\in\Par^c_\mu$ and $c$-tuples of partitions $(q_1,\ldots,q_c)$ where $|\lambda|=|\mu|+c\sum_{i=1}^c|q_i(\lambda)|.$ The right hand side of~\eqref{coresgen} corresponds to all choices of $c$-tuples $q_1,\ldots,q_c\in \Par,$ and the weighting by $c$ corresponds to each box in $q_i$ corresponding to $c$ boxes in $\lambda$.
\end{proof}

Next, we define a partition statistic $\lambda\divcparts$ that arises as a special case of one of the statistics that we study.

 For a positive integer $d$, let $m_d(\lambda)$ denote the number of parts of $\lambda$ of size $d$, and for fixed $c$ let $\lambda\divcparts$ denote the weighted sum \begin{equation}\lambda\divcparts=\sum_{d=1}^{\infty} \left\lfloor\frac{m_d(\lambda)}{c}\right\rfloor.\end{equation} In words, $\lambda\divcparts$ counts the number of rectangles, of any width, of positive height divisible by $c$ in the diagram of $\lambda$ such that the whole right edge of the rectangle, and at least the rightmost step of the top edge, lies on the boundary of $\lambda$.

\begin{ex}
Let $c=3$. The partition $\lambda=(7,7,4,4,4,4,4,4,4,3,2,2,2,1)$ has $m_7(\lambda)=2$, $m_4(\lambda)=7$, $m_3(\lambda)=1$, $m_2(\lambda)=3$ and $m_1(\lambda)=1$. So, the only nonzero contributions to $\lambda_{\square}^{3*}$ are when $d=2$ and $d=4$, and \begin{equation*}\lambda_{\square}^{3*}=\left\lfloor\frac{m_2(\lambda)}{3}\right\rfloor+\left\lfloor\frac{m_4(\lambda)}{3}\right\rfloor=1+2=3.
    \end{equation*}
    
    \end{ex}

    We now define the map $G_c,$ previously defined in \cite{Walsh}. 
    \begin{defn}[The map $G_c$]
    The map $G_c:\Par\to \Par\times K_c$, where $K_c=\{\lambda\in\Par: \lambda_{\square}^{c*}=0\}$ is the set of partitions with no parts repeated $c$ or more times, maps a partition $\lambda$ to $(\xi,\nu)$ where for each $d\in \mathbb{N}$,
    \begin{equation}m_d(\xi)=\left\lfloor\frac{m_d(\lambda)}{c}\right\rfloor,\end{equation} and
    \begin{equation}m_d(\nu)=m_d(\lambda)-c\left\lfloor\frac{m_d(\lambda)}{c}\right\rfloor.\end{equation}
    
    We write $(G_c)^{-1}$ for the inverse map $(G_c)^{-1}:K_c\times \Par\to \Par$ where, for each $d\in\mathbb{N}$ \begin{equation}
        m_d\left(\left(G_c)^{-1}(\xi,\nu)\right)\right)= m_d(\xi)c+m_d(\nu). 
    \end{equation}
    \end{defn}

    \begin{ex} As shown in Figure~\ref{exampleGc}, the partition $\lambda=(7,7,4,4,4,4,4,4,4,3,2,2,2,1)$ has $G_3(\lambda)=((4,4,2),(7,7,3,4,1)).$
    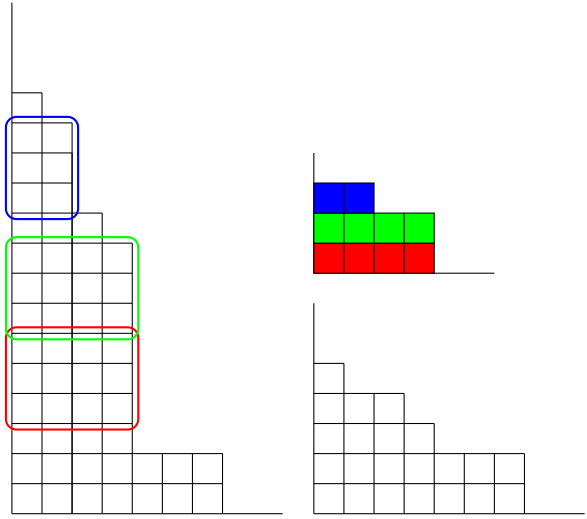
\begin{figure}\label{exampleGc}[ht]
\begin{center}
\begin{tikzpicture}\begin{scope}[yscale=1,xscale=-1,rotate=90,scale=0.4]
\draw (0,9)--(0,0);
\draw (1,7)--(1,0);
\draw (2,7)--(2,0);
\draw (3,4)--(3,0);
\draw (4,4)--(4,0);
\draw (5,4)--(5,0);
\draw (6,4)--(6,0);
\draw (7,4)--(7,0);
\draw (8,4)--(8,0);
\draw (9,4)--(9,0);
\draw (10,3)--(10,0);
\draw (11,2)--(11,0);
\draw (12,2)--(12,0);
\draw (13,2)--(13,0);
\draw (14,1)--(14,0);
\draw (2,7)--(0,7);
\draw (2,6)--(0,6);
\draw (2,5)--(0,5);
\draw (9,4)--(0,4);
\draw (10,3)--(0,3);
\draw (12,2)--(0,2);
\draw (13,2)--(0,2);
\draw (14,1)--(0,1);
\draw (17,0)--(0,0);
\draw[white] (0,9.5)--(0,10);
\draw[rounded corners,red,thick] (2.8, -0.2) rectangle (6.2, 4.2) {};
\draw[rounded corners,green,thick] (5.8, -0.2) rectangle (9.2, 4.2) {};
\draw[rounded corners,blue,thick] (9.8, -0.2) rectangle (13.2, 2.2) {};
\end{scope}\end{tikzpicture}\begin{tikzpicture}\begin{scope}[yscale=1,xscale=-1,rotate=90,scale=0.4]
\draw (8,6)--(8,0)--(12,0);
\draw[fill=red] (8,0) rectangle (9, 4) {};
\draw[fill=green] (9,0) rectangle (10,4) {};
\draw[fill=blue] (10,0) rectangle (11, 2) {};
\draw(8,3)--(10,3); 
\draw(8,2)--(10,2); 
\draw(8,1)--(11,1); 

\draw (0,9)--(0,0);
\draw (1,7)--(1,0);
\draw (2,7)--(2,0);
\draw (3,4)--(3,0);
\draw (4,3)--(4,0);
\draw (5,1)--(5,0);
\draw (2,7)--(0,7);
\draw (2,6)--(0,6);
\draw (2,5)--(0,5);
\draw (3,4)--(0,4);
\draw (4,3)--(0,3);
\draw (4,2)--(0,2);
\draw (5,1)--(0,1);
\draw (7,0)--(0,0);
\end{scope}\end{tikzpicture}
\end{center}
\caption{The partition $\lambda$ has $G_3(\lambda)=((4,4,2),(7,7,4,3,1))$.}\label{G3 figure}
\end{figure}
\end{ex}

The next proposition establishes that the $c$-core of a partition $\lambda$ is also the $c$-core of the second argument of $G_c(\lambda),$ so we may restrict $G_c$ to $\Par^c_\mu$ in a way that interacts sensibly with cores.

\begin{prop}\label{Glaisherrefinestocores}
If $\lambda\in\Par^c_{\mu}$ and $G_c(\lambda)=(\xi,\nu)$, then $\nu\in\Par^c_{\mu}$.
\end{prop}
\begin{proof}
Suppose the proposition is false for some $\lambda$ of minimal possible size. Then, we must have $\lambda\not=\nu$, so $\lambda$ must have some part of some size $d$ repeated at least $c$ times.
The rightmost column of the rectangle of width $d$ and height $c$ which has all right edges and the rightmost top edge in the boundary of $\lambda$ is a rimhook of size $c$. Let $\lambda'$ be the partition formed by deleting this rimhook. Then, $\lambda'$ has $c$-core $\mu$ and $G_c(\lambda')=(\xi',\nu)$ for some $\xi'.$ So, since $\lambda'$ is smaller than $\lambda$, $\nu\in\Par^c_{\mu}$.
\end{proof}

Therefore, $G_c$ restricts to a bijection $G_c|_{\Par^c_{\mu}}:\Par^c_{\mu}\to\Par\times (K_c\cap \Par^c_{\mu})$. This allows us to use $G_c$ to prove the following.

%\begin{notn} Write $G_c^\mu$ for the restriction of $G_c$ from $\Par$ to $\Par^c_\mu.$ Then, by Proposition~\ref{Glaisherrefinestocores}, $G^c_{\mu}$ is a bijection $\Par^c_\mu\to \Par\times K\cap\Par^c_\mu.$
%\end{notn}

\begin{prop}
For a positive integer $c$ and a $c$-core $\mu$, the following product formula holds. \begin{equation}\label{GenFunNots}\sum_{\lambda\in K_c\cap\Par^c_\mu}q^{|\lambda|}=q^{|\mu|}\prod_{m\geq 1}\frac{1}{(1-q^{mc})^{c-1}}.\end{equation}
\end{prop}
\begin{proof}
Let $\lambda\in\Par_c^\mu.$ Then $G^c_\mu$ bijectively maps $\lambda$ to a pair of partitions $(\xi,\nu)$ with $|\lambda|=|\xi|+c|\nu|,$ because each part of $\nu$ corresponds to $c$ parts of $\lambda$ of the same size. So,
\begin{equation}\label{Glaisher}
    \sum_{\lambda\in\Par^c_{\mu}}q^{|\lambda|}=\sum_{\xi\in K_c\cap\Par^c_\mu}q^{|\xi|}\times \sum_{\nu\in\Par}q^{c|\nu|}.
    \end{equation}
    
    Substituting~\eqref{justpartitions} and applying Proposition~\ref{coresandquotients} to~\eqref{Glaisher} gives 
    
    \begin{equation}
   q^{|\mu|}\prod_{m\geq 1}\frac{1}{(1-q^{mc})^c}=\sum_{\xi\in K_c\cap\Par^c_\mu}q^{|\xi|}\times \prod_{m\geq 1}\frac{1}{(1-q^{mc})},
    \end{equation}
    
   which rearranges to give~\eqref{GenFunNots}.\end{proof}
   
   We are now in a position to prove the following identity, which forms the base case for Proposition~\ref{4to3}. 
\begin{thm}\label{basic partition step}
For a fixed positive integer $c$,
\begin{equation}\label{base case}\sum_{\lambda\in\Par_{\mu}^c} q^{|\lambda|}t^{\lambda\divcparts}=q^{|\mu|}\prod_{i\geq 1}\frac{1}{(1-q^{ic})^{c-1}}\prod_{j\geq1}\frac{1}{1-q^{jc}t}.\end{equation}
\end{thm}
\begin{proof} 
Let $\lambda\in\Par_c^\mu.$ Then $G^c_\mu$ bijectively maps $\lambda$ to a pair of partitions $(\xi,\nu)$ with $|\lambda|=|\xi|+c|\nu|,$ where each part of $\nu$ of size $d$ corresponds to a $d\times c$ rectangle in $\lambda$ contributing to $\lambda\divcparts$. So,

\begin{equation}\label{decomposition1}\sum_{\lambda\in\Par_{\mu}^c} q^{|\lambda|}t^{\lambda\divcparts}=\sum_{\xi\in K_c\cap\Par^c_\mu}q^{|\xi|}\times \sum_{\nu\in\Par}q^{c|\nu|}t^{l(\nu)}.\end{equation}

Substituting~\eqref{GenFunNots} and~\eqref{partitionswithparts} into~\eqref{decomposition1} gives~\eqref{base case}.
\end{proof}

\section{Further partition statistics}
\stepcounter{essaypart}
%\stepcounter{essaypart}
In this section we define the main partition statistics of interest, $h_{x,c}^+$ and $h_{x,c}^{-},$ where $x$ is a real parameter and $c$ is a positive integer. The main aim of this paper is to compute the distribution of the statistics $h_{x,c}^+$ and $h_{x,c}^{-}$ over $\Par^c_{\mu}$, given in Theorem~\ref{main theorem}. The previous section computed the distribution of $\lambda\divcparts$ over $\Par^c_{\mu}$, giving the right hand side in Theorem~\ref{main theorem}. In this section, we connect to $\lambda\divcparts$ by observing that $\lambda\divcparts=h^+_{0,c},$ and then sketch a framework for piecing together a family of involutions $I_{r,s,c}$ defined on $\Par^c_{\mu}$ to prove that the distribution $h_{x,c}^{\pm}$ over $\Par^c_{\mu}$ is independent of both $x$ and the sign. The rest of the paper will then construct the component bijections $I_{r,s,c}.$

In order to reduce the proof of Theorem~\ref{main theorem} to the construction of appropriate bijections $I_{r,s,c},$ we first prove that Theorem~\ref{main theorem} is implied by Theorem~\ref{h+ is h-}, which states that the $h_{x,c}^+$ and $h_{x,c}^-$ have the same distribution over $\Par^c_{\mu}$. Then, we introduce three other statistics $\middxc$, $\critminusxc$ and $\critplusxc$ and decompose $h_{x,c}^+$ and $h_{x,c}^-$ in terms of these other statistics. Finally, we outline sufficient conditions for the bijections $I_{r,s,c}$ to prove Theorem~\ref{h+ is h-} in terms of these three statistics.

We conclude the section by explaining how the main result of \cite{BFN} follows from Theorem~\ref{main theorem}.

\begin{defn}
For a partition $\lambda$, $x\in [0,\infty]$ and a fixed $c\in\mathbb{N}$, 
\begin{equation} h_{x,c}^+(\lambda)=\left|\left\{\square \in \lambda \,:  c\mid  h(\square) \text{ and } \frac{a(\square)}{l(\square)+1}\leq x<\frac{a(\square)+1}{l(\square)}\right\}\right|,\end{equation} and
\begin{equation}h_{x,c}^-(\lambda)=\left|\left\{\square \in \lambda \,:  c\mid  h(\square) \text{ and } \frac{a(\square)}{l(\square)+1}< x\leq\frac{a(\square)+1}{l(\square)}\right\}\right|.\end{equation}
We interpret a fraction with denominator $0$ as $+\infty$.
\end{defn}

Note that a box $\square$ contributes to $h_{0,c}^+$ if and only if $a(\square)=0$ and $c\mid (l(\square)+1).$ That is, $\square$ is the rightmost box in its row, and there is some $m$ such that the row containing $\square$ and exactly $mc-1$ rows to the above all have the same height. The number of such boxes is exactly $\lambda\divcparts$.

Similarly, $h_{\infty,c}^-(\lambda)=\bar{\lambda}\divcparts,$ where $\bar{\lambda}$ is the partition conjugate to $\lambda$. 

We are now in a position to state our main result.

\begin{thm}\label{main theorem}
For all $x\in[0,\infty)$ we have
\begin{equation}\sum_{\lambda\in\Par^c_{\mu}}q^{|\lambda|}t^{h_{x,c}^+(\lambda)}=q^{|\mu|}\prod_{i\geq 1}\frac{1}{(1-q^{ic})^{c-1}}\prod_{j\geq1}\frac{1}{1-q^{jc}t},\end{equation} and for all $x\in (0,\infty],$\begin{equation}\sum_{\lambda\in\Par^c_{\mu}}q^{|\lambda|}t^{h_{x,c}^-(\lambda)}=q^{|\mu|}\prod_{i\geq 1}\frac{1}{(1-q^{ic})^{c-1}}\prod_{j\geq1}\frac{1}{1-q^{jc}t}.\end{equation}\end{thm}

Proposition~\ref{4to3} shows that Theorem~\ref{main theorem} is a consequence of the following result.

\begin{thm}\label{h+ is h-}
For all positive rational numbers $x$ and all integers $n\geq0$ we have \begin{equation}\sum_{\lambda\in\Par_\mu^c(n)}t^{h_{x,c}^+(\lambda)}=\sum_{\lambda\in\Par_\mu^c(n)}t^{h_{x,c}^-(\lambda)}.\end{equation}
\end{thm}

\subsection{Reducing to Theorem~\ref{h+ is h-}}
\begin{prop}\label{4to3}
Theorem~\ref{h+ is h-} implies Theorem~\ref{main theorem}.
\end{prop}
\begin{proof}
For $x\in[0,\infty)$, $c\in\mathbb{N}$ and $\delta\in\{+,-\}$ define $$H_{x,c}^\delta(n)=\sum_{\lambda\in\Par^c_{\mu}(n)}t^{h_{x,c}^\delta(\lambda)}.$$ Suppose $H_{x,c}^\delta(n)$ is independent of both $x$ and $\delta$. Then $$H_{x,c}^\delta(n)=H_0^+(n)=\sum t^{h_{0,c}^+(\lambda)}=\sum t^{\lambda\divcparts}.$$ Theorem~\ref{main theorem} then follows immediately by multiplying by $q^n$, adding over all $n\geq 0$, and applying Theorem~\ref{basic partition step}. So, it suffices to prove that Theorem~\ref{h+ is h-} implies that $H_{x,c}^\delta(n)$ is independent of $x$ and $\delta$.

For an integer $n$, we call a positive rational number $r$ a \textit{critical rational for} $n$ if there is a partition $\mu\in\Par(n)$ and a box $\square\in d(\mu)$ such that $ h(\square)$ is divisible by $c$, and $\frac{a(\square)}{l(\square)+1}=r$ or $\frac{a(\square)+1}{l(\square)}=r.$ By convention, $0$ and $+\infty$ are regarded as critical rationals for all $n$.

We denote the set of all critical rationals for $n$ by $C(n).$ Since there are finitely many partitions of $n$ each containing finitely many boxes in their diagrams, $C(n)$ is finite for all $n$. For a fixed $n$, write $C(n)=\{0=r_0<r_1<\cdots<r_{k-1}<r_k=+\infty\}.$ Define open intervals $I_j=(r_{j-1},r_j)$ for each $1\leq j\leq k$. Then $[0,\infty]$ decomposes into a disjoint union $$[0,\infty]=I_1\cup I_2\cup\cdots\cup I_k\cup C(n).$$

Let $x,x'$ be two elements of the same interval $I_j$ and let $\delta,\delta'\in\{+,-\}.$ Suppose $\lambda$ is any partition of $n$. Since there are no critical rationals between $x$ and $x'$, $\square\in d(\lambda)$ contributes to $h_{x,c}^\delta(\lambda)$ if and only if it contributes to $h_{x',c}^{\delta'}(\lambda).$ So, $t^{h_{x,c}^\delta(\lambda)}=t^{h_{x',c}^{\delta'}(\lambda)}$. Adding over all $\lambda$, we see that if $x,x'\in I_j$, \begin{align} H_{x,c}^\delta(n)=H_{x',c}^{\delta'}(n).\end{align}

Similarly, for all $x\in I_j$, \begin{align}\label{two}H_{r_{j-1},c}^+(n)=H_{x,c}^\delta(n)=H_{r_j,c}^-(n).\end{align}

On the other hand, Theorem~\ref{h+ is h-} implies that \begin{align}\label{three}H_{r_j,c}^+(n)=H_{r_j,c}^-(n).\end{align}

Therefore, for $\delta,\delta'\in\{+,-\}$ and $y\geq y'$ by applying a chain of these equalities starting with $H_{y,c}^\delta(n),$ one can reduce $y$ to a critical rational and change $\delta$ to a $+$ using~\eqref{two}, or using~\eqref{three} if $y$ is already a critical rational. Then one may iteratively apply~\eqref{three} and~\eqref{two} to change $\delta$ to a $-$, and then reduce $y$ to the next lowest critical rational and change $\delta$ back to a $+$, until an equality $H_{y,c}^{\delta}(n)=H_{r_j,c}^-(n)$ is obtained for $r_{j-1}\leq y'\leq r_j$. Then, applying~\eqref{two} again with $x=y'$ (and~\eqref{three} to flip the sign of $\delta$ if $y=r_{j-1}$ and $\delta'=-$), one obtains $H_{y,c}^\delta(n)=H_{y',c}^{\delta'}(n).$ \end{proof}

\subsection{Reducing to a symmetry property}

In the case $x$ is rational, where $h_{x,c}^+$ and $h_{x,c}^-$ may differ, it is useful to separate the boxes that contribute to both statistics from those that contribute to just one. In order to do this, we define the following statistics. 

\begin{defn} For $x=\frac{r}{s}$ a rational number, we have
\begin{equation}\critplusxc(\lambda)=\left|\left\{\square\in\lambda: c\mid  h(\square)\text{ and } \frac{a(\square)}{l(\square+1)}=x\right\}\right|,\end{equation}
\begin{equation}\critminusxc(\lambda)=\left|\left\{\square\in\lambda: c\mid  h(\square)\text{ and } \frac{a(\square)+1}{l(\square)}=x\right\}\right|,\end{equation}
\begin{equation}\middxc(\lambda)=\left|\left\{\square\in\lambda: c\mid  h(\square)\text{ and }-s<sa(\square)-rl(\square)<r
\right\}\right|.\end{equation}
\end{defn}

The next proposition shows that a bijection satisfying some constraints on its behaviour with respect to these statistics will give a bijective proof of Theorem~\ref{h+ is h-}.

\begin{prop}\label{bijection properties}
Let $r,s,c$ be positive integers with $(r,s)=1$ and let $x=\frac{r}{s}$. Suppose there exists a bijection $I_{r,s,c}:\Par^c_{\mu}\to\Par^c_\mu$ such that
\begin{enumerate}
    \item $|\lambda|=|I_{r,s,c}(\lambda)|$,
    \item $\middxc(\lambda)=\middxc(I_{r,s,c}(\lambda))$,
    \item $\critplusxc(\lambda)+\critminusxc(\lambda)=\critplusxc(I_{r,s,c}(\lambda))+\critminusxc(I_{r,s,c}(\lambda))$,
    \item $\critplusxc(\lambda)=\critminusxc(I_{r,s,c}(\lambda)).$
\end{enumerate}
Then, Theorem~\ref{h+ is h-} is true.
\end{prop}
\begin{proof}
Assume that $I_{r,s,c}$ exists. Then, property 3 and 4 together imply that 
\begin{equation}
    \critminusxc(\lambda)= \critplusxc(I_{r,s,c}(\lambda))
\end{equation}
so $I_{r,s,c}$ exchanges $\critplusxc$ and $\critminusxc$ whilst preserving $|\lambda|$ and $\middxc.$

Note that a box $\square$ contributes to $\middxc$ if and only if $-s<sa(\square)-rl(\square)<r$ and the $c\mid h(\square)$. Adding $s+rl(\square)$, and dividing by $sl(\square)$, the left inequality is equivalent to
%\begin{equation}\label{midequivalent}
 %   rl(\square)<s(a(\square)+1)
%\end{equation}
%Equivalently,
\begin{equation}\label{leftrearrange}
    \frac{a(\square)+1}{l(\square)}>x.
\end{equation}
Similar manipulation of the right inequality together with \eqref{leftrearrange} shows that $\square$ contributes to $\middxc$ if and only if 
\begin{equation}\label{midequiv}
    \frac{a(\square)}{l(\square)+1}<x<\frac{a(\square)+1}{l(\square)}.
\end{equation}
So, comparing the definitions of $\critminusxc,$ $\critplusxc,$ $h_{x,c}^{+},$ $h_{x,c}^{-}$ and ~\eqref{midequiv},
\begin{equation}h^+_{x,c}(\lambda)=\middxc(\lambda)+\critplusxc(\lambda)\end{equation}
and
\begin{equation}h^-_{x,c}(\lambda)=\middxc(\lambda)+\critminusxc(\lambda).\end{equation}

So, $I_{r,s,c}$ exchanges $h_{x,c}^+(\lambda)$ and $h_{x,c}^{-}(\lambda)$ whilst preserving $|\lambda|,$ and hence proves Theorem~\ref{h+ is h-}.\end{proof}

\subsection{Connecting to Buryak-Feigin-Nakajima}
When $c$ is divisible by $r+s$, Theorem~\ref{main theorem} implies the following product formula. In the case $r+s=c$, this is the main combinatorial result of \cite{BFN}.
\begin{corol}
Let $r$ and $s$ be coprime integers, let $x=\frac{r}{s}$ and let $r+s\mid c$. Then
\begin{equation}
    \sum_{\lambda\in\Par}q^{|\lambda|}t^{\critplusxc(\lambda)}=\prod_{\substack{i\geq 1\\ c\nmid i}}\frac{1}{1-q^i}\prod_{i\geq 1}\frac{1}{1-q^{ic}t}.
\end{equation}
\end{corol}
\begin{proof} First we show that under the assumption that $r+s\mid c$, then for any partition $\lambda$, $\middxc(\lambda)=0.$
Suppose $\square$ were to contribute to $\middxc(\lambda),$ then $\square$ would have to satisfy \begin{equation}
    -s<sa(\square)-rl(\square)<r.
\end{equation} Adding $rl+sl+s,$
\begin{equation}
    (r+s)l(\square)<s\left(a(\square)+l(\square)+1\right)<(r+s)\left(l(\square)+1\right)
\end{equation} 
However, the upper and lower bound are consecutive multiples of $r+s$, and therefore $s(a(\square)+l(\square)+1)$ cannot be a multiple of $r+s$, so by assumption cannot be a multiple of $c$. So, $c\nmid h(\square)$ so $\square$ cannot contribute to $\middxc(\lambda).$

So in this case $h_{x,c}^+(\lambda)=\critplusxc(\lambda)$ and Theorem~\ref{main theorem} becomes 
\begin{equation}
\sum_{\lambda\in\Par^c_{\mu}}q^{|\lambda|}t^{\critplusxc(\lambda)}=q^{|\mu|}\prod_{i\geq 1}\frac{1}{(1-q^{ic})^{c-1}}\frac{1}{1-q^{ic}t}.
\end{equation}
Summing both sides over all $c$-cores $\mu$ and applying Proposition~\ref{coresandquotients}, 
\begin{align}
\sum_{\lambda\in\Par}q^{|\lambda|}t^{\critplusxc(\lambda)}&=\prod_{i\geq 1}\frac{(1-q^{ic})^c}{1-q^{i}}\frac{1}{(1-q^{ic})^{c-1}}\frac{1}{1-q^{ic}t}\\
&=\prod_{i\geq 1}\frac{(1-q^{ic})}{1-q^{i}}\frac{1}{1-q^{ic}t}\\
&=\prod_{\substack{i\geq 1\\ c\nmid i}}\frac{1}{1-q^{i}}\prod_{i\geq 1}\frac{1}{1-q^{ic}t}.\end{align}
\end{proof}

\section{The multigraph \texorpdfstring{$M_{r,s,c}$}{Mrsc}}
\stepcounter{essaypart}
For the rest of the paper, $x=\frac{r}{s}$ is a rational number with $r$ coprime to $s$. In this section we take our first key step in the construction of the involution $I_{r,s,c}$. First, Proposition~\ref{motivate Mrsc} relates the statistics $\middxc,$ $\critminusxc$ and $\critplusxc$ to the boundary graph. We use this relationship to define a map from the boundary graph to a multigraph $M_{r,s,c}$ that picks out the information relevant to $\middxc$ and $\critplusxc$, much as the $c$-abacus tour does for the $c$-core. The rest of the section then outlines the method for proving that $M_{r,s,c}$ retains partition data, including the proofs that $M_{r,s,c}$ retains the $c$-core and the area. The proof it retains the area is a particularly easy example using the same methodology as used in the more technical proofs in Section 6, which check that $M_{r,s,c}$ retains $\middxc(\lambda)$ and $\critplusxc(\lambda)+\critminusxc(\lambda)$. 

 When we define $I_{r,s,c}$ as a bijection on partitions, we build into the definition that $I_{r,s,c}$ preserves $M_{r,s,c}(\lambda)$ for any partition $\lambda.$ So, together with these results it is immediate that $I_{r,s,c}$ does map $\Par^c_\mu$ to $\Par^c_\mu$ and satisfies hypothesis 1 in Proposition~\ref{bijection properties}. 

\begin{prop}\label{motivate Mrsc}
Let $\lambda$ be a partition and let $\square\in\lambda.$ Let $e_i$ be the foot of $\square$, departing from $(x_1-1,y_1)$ and arriving at $(x_1,y_1)$ and let $e_j$ be the hand of $\lambda$, departing from $(x_2,y_2+1)$ and arriving at $(x_2,y_2)$.  Let $t=r(y_1-y_2)+s(x_1-x_2)$. Then
\begin{enumerate}
\item $\square$ contributes to $\critplusxc$ if and only if $t=0$ and $x_1-y_1\equiv x_2-y_2\pmod{c}$;
\item $\square$ contributes to $\middxc$ if and only if  $0<t<r+s$ and $x_1-y_1\equiv x_2-y_2\pmod{c}$;
\item $\square$ contributes to $\critminusxc$ if and only if $t=r+s$ and $x_1-y_1\equiv x_2-y_2\pmod{c}.$
\end{enumerate}
\end{prop}
\begin{proof} By the definition of index, $x_1-y_1=i$ and $x_2-y_2=j$. Let $\square \in \lambda$ have bottom left corner $(x_\square,y_\square)$.
By Proposition~\ref{c-hooks are c-inversions}, $\square$ has hook length divisible by $c$ if and only if $j\equiv i \pmod{c}$, i.e. $x_1-y_1\equiv x_2-y_2\pmod{c}.$ So, assume that $\square$ does have hook length divisible by $c$.

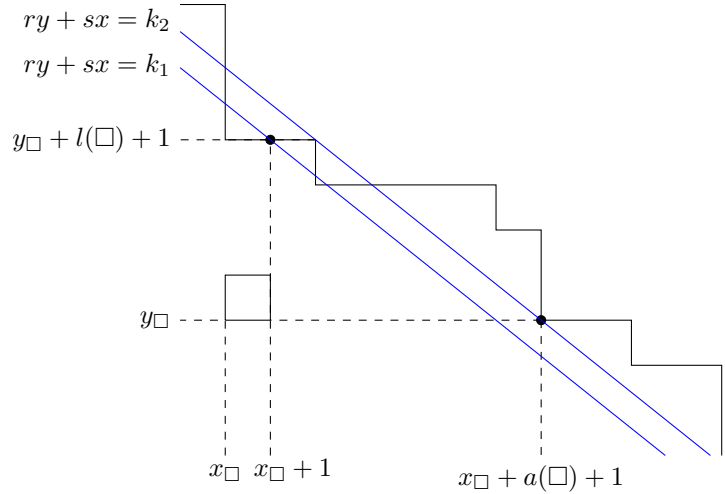
\begin{figure}[ht]
\begin{center}
\begin{tikzpicture}\begin{scope}[yscale=1,xscale=-1,rotate=90,scale=0.6]

\draw[black] (3,1)--(3,2)--(4,2)--(4,1)--(3,1);
\draw[dashed] (0,1)--(3,1)--(3,0);
\node[below] at (0,1) {$x_\square$};
\node[left] at (3,0) {$y_\square$};
\draw[dashed] (3,8)--(0,8);
\draw[dashed] (3,8)--(3,2);
\draw[fill=black] (3,8) circle[radius=0.1];
\draw[fill=black] (7,2) circle[radius=0.1];
\node[below] at (0,8) {$x_\square+a(\square)+1$};
\draw[dashed] (7,3)--(7,0);
\draw[dashed] (0,2)--(7,2);
\node[left] at (7,0) {$y_\square+l(\square)+1$};
\node[below] at (0.13,2.5) {$x_\square+1$};
\draw (0,12)--(2,12)--(2,10)--(3,10)--(3,8)--(4,8)--(5,8)--(5,7)--(6,7)--(6,4)--(6,3)--(7,3)--(7,1)--(10,1)--(10,0);
\draw[blue] (9.4,0)--(0,11.75);
\draw[blue] (8.6,0)--(0,10.75);
\node[left] at (8.65,0) {$ry+sx=k_1$};
\node[left] at (9.65,0) {$ry+sx=k_2$};
\end{scope}\end{tikzpicture}
\end{center}
\caption{The box $\square$ and the lines $ry+sx=k_1$ and $ry+sx=k_2$.}
\end{figure}

Let $k_1=ry_1+sx_1$ and $k_2=ry_2+sx_2$. Then, \begin{equation}\label{tk1k1}t=k_1-k_2,\end{equation} \begin{equation}\label{armlegbox1}s(x_\square+1)+r(y_\square+l(\square)+1)=k_1,\end{equation} and \begin{equation}\label{armlegbox2}s(x_\square+a(\square)+1)+ry_\square=k_2.\end{equation} Subtracting ~\eqref{armlegbox1} from ~\eqref{armlegbox2}, and substituting in ~\eqref{tk1k1} \begin{equation}sa(\square)-rl(\square)=r-t.\end{equation}

By definition, $\square$ contributes to $\critplusxc$ if and only if $sa(\square)-rl(\square)=r$, that is, when $t=0$, proving the first claim. Similarly, $\square$ contributes to $\middxc$ if and only if $-s<sa(\square)-rl(\square)<r$, or equivalently $0<t<r+s,$ proving the second claim.

Finally, note $sa(\square)-rl(\square)=-s$ if and only if $t=r+s$. 
\end{proof}

We define the multigraph $M_{r,s,c}(\lambda)$ accordingly.

\begin{defn}[$M_{r,s,c}$, $(r,s,c)$-tour] For a partition $\lambda$ the SE directed multigraph $M_{r,s,c}(\lambda)$ is obtained from $b(\lambda)$ by imposing the relation $\sim_{r,s,c}$ on the vertices, where $(x_1,y_1)\sim_{r,s,c} (x_2,y_2)$ if $ry_1+sx_1=ry_2+sx_2$ and $x_2-y_2\equiv x_1-y_1\pmod{c}$. Denote the equivalence class with $ry+sx=v$ and $x-y\equiv i\pmod{c}$ by $(v,[i]).$ Let $q_{r,s,c}:b(\lambda)\to M_{r,s,c}(\lambda)$ be the induced homomorphism. The $(r,s,c)$-tour of $M_{r,s,c}(\lambda)$ associated to $\lambda$ is $(q_{r,s,c}(e_i))_{i\in\mathbb{Z}}$. 
At each vertex $(v,[i])$, we count the number of east edges arriving at $(v,[i])$ in the $(r,s,c)$-tour and denote this quantity by $\Ein(v,[i])$. Similarly, we count the number of east edges departing from $(v,[i])$ in the $(r,s,c)$-tour and denote this quantity by $\Eout(v,[i])$. We define $\Sin(v,[i])$ and $\Sout(v,[i])$ analogously.
\end{defn}

Now, we explain give a useful way of drawing $M_{r,s,c}(\lambda)$ in the plane. First, we show Proposition~\ref{lattice points}, which says that when the plane is cut into strips of width $\lcm(c,r+s)$ by lines with $sx+ry$ constant, then there is a unique representative of each possible vertex of $M_{r,s,c}(\lambda)$ contained in the strip. 

\begin{prop}\label{lattice points}
If there is a lattice point $(x,y)$ satisfying both $sx+ry=v$ and $x-y\equiv i\pmod{c},$ then for any real number $m$ there is exactly one such lattice point satisfying the inequality $m\leq x-y < m+\lcm(c,r+s)$.
\end{prop}
\begin{proof}
First, note that translating a lattice point $(x,y)$ by $(r,-s)$ does not change the value of $sx+ry$. Moreover, there is no lattice point on the line $sx+ry=v$ between $(x,y)$ and $(x+r,y-s),$ since if $(x+l_1,y-l_2)$ were such a point, we would have $sl_1-rl_2=0$, so since $r$ and $s$ are coprime, $s\mid l_2$ and $r\mid l_1$. 

Secondly, note that translating by $(r,-s)$ changes the value of $x-y$ by $r+s$. So, the translations that preserve both the value of $sx+ry$ and the residue class of $[y-x]$ modulo $c$ are the translations by $(ar,-as)$ where $a(r+s)$ is divisible by $c$, that is, $a(r+s)$ is divisible by $\lcm(c,r+s)$. Exactly one of these translates lies in the region $m\leq x-y < m+\lcm(c,r+s)$.
\end{proof}

So, for a fixed integer $n$, we can draw the multigraph by taking the vertices to be lattice points in the portion of $\mathbb{R}^2$ in between the lines $x-y=n$ and $x-y=\lcm(c,r+s)+n$, with an identification along the boundary lines given by $$(x,y)\sim \left(x+\frac{r \lcm(c,r+s)}{r+s},y-\frac{s \lcm(c,r+s)}{r+s}\right).$$ We identify a lattice point $(x,y)$ with the vertex $(sx+ry,[x-y])$. Then, south edges in the multigraph from $(v,[i])$ to $(v-r,[i+1])$ are south edges between lattice points in the region described. Similarly, east edges from $(v,[i])$ to $(v+s,[i+1])$ are east edges between lattice points. Moreover, each vertex $(v,[i])$ corresponds to a unique lattice point in the region. We can view the $(r,s,c)$-tour as the \textit{cylindrical} lattice path tour obtained by collapsing the boundary of the partition onto this cylinder.

\begin{ex}
When $c=2$, $r=3$ and $s=2$ we may draw the $(r,s,c)$-multigraph of $\mu=(12,12,10,8,7,4,1,1,1)$ as in Figure~\ref{Big multigraph 1}. 
\begin{figure}
\begin{center}
\begin{tikzpicture}\begin{scope}[yscale=1,xscale=-1,rotate=90,scale=1.7]
\draw (5,3)--(8.5,6.5);
\node[below] at (5,3) {$x-y=-2$};
\draw (12.5,0.5)--(10,-2);
\node[above] at (12.5,0.5) {$x-y=-12$};
\draw[fill=white] (11.85,-0.4)--(12.15,-0.4)--(12.15,0.4)--(11.85,0.4)--(11.85,-0.4);
\node at (12,0) {$(36,[0])$};
\node at (11,0) {$(33,[1])$};
\draw (10.85,-0.4)--(11.15,-0.4)--(11.15,0.4)--(10.85,0.4)--(10.85,-0.4);
\node at (10,0) {$(30,[0])$};
\draw (9.85,-0.4)--(10.15,-0.4)--(10.15,0.4)--(9.85,0.4)--(9.85,-0.4);
\node at (9,0) {$(27,[1])$};
\draw (8.85,-0.4)--(9.15,-0.4)--(9.15,0.4)--(8.85,0.4)--(8.85,-0.4);
\node at (9,1) {$(29,[0])$};
\draw (8.85,0.6)--(9.15,0.6)--(9.15,1.4)--(8.85,1.4)--(8.85,0.6);
\node at (8,1) {$(26,[1])$};
\draw (7.85,0.6)--(8.15,0.6)--(8.15,1.4)--(7.85,1.4)--(7.85,0.6);
\node at (7,1) {$(23,[0])$};
\draw (6.85,0.6)--(7.15,0.6)--(7.15,1.4)--(6.85,1.4)--(6.85,0.6);
\node at (6,1) {$(20,[1])$};
\draw (5.85,0.6)--(6.15,0.6)--(6.15,1.4)--(5.85,1.4)--(5.85,0.6);
\node at (6,2) {$(22,[0])$};
\draw (5.85,1.6)--(6.15,1.6)--(6.15,2.4)--(5.85,2.4)--(5.85,1.6);
\node at (6,3) {$(24,[1])$};
\draw (5.85,2.6)--(6.15,2.6)--(6.15,3.4)--(5.85,3.4)--(5.85,2.6);
\draw[fill=white] (5.85,3.6)--(6.15,3.6)--(6.15,4.4)--(5.85,4.4)--(5.85,3.6);
\node at (6,4) {$(26,[0])$};
\draw[fill=white] (9.85,-2.4)--(10.15,-2.4)--(10.15,-1.6)--(9.85,-1.6)--(9.85,-2.4);
\node at (10,-2) {$(26,[0])$};
\node at (9,-2) {$(23,[1])$};
\draw (8.85,-2.4)--(9.15,-2.4)--(9.15,-1.6)--(8.85,-1.6)--(8.85,-2.4);
\node at (9,-1) {$(25,[0])$};
\draw (8.85,-1.4)--(9.15,-1.4)--(9.15,-0.6)--(8.85,-0.6)--(8.85,-1.4);
\node at (8,2) {$(28,[0])$};
\draw (7.85,1.6)--(8.15,1.6)--(8.15,2.4)--(7.85,2.4)--(7.85,1.6);
\node at (7,2) {$(25,[1])$};
\draw (6.85,1.6)--(7.15,1.6)--(7.15,2.4)--(6.85,2.4)--(6.85,1.6);
\node at (7,3) {$(27,[0])$};
\draw (6.85,2.6)--(7.15,2.6)--(7.15,3.4)--(6.85,3.4)--(6.85,2.6);
\node at (7,4) {$(29,[1])$};
\draw (6.85,3.6)--(7.15,3.6)--(7.15,4.4)--(6.85,4.4)--(6.85,3.6);
\node at (10,-1) {$(28,[1])$};
\draw (9.85,-1.4)--(10.15,-1.4)--(10.15,-0.6)--(9.85,-0.6)--(9.85,-1.4);
\node at (8,0) {$(24,[0])$};
\draw (7.85,-0.4)--(8.15,-0.4)--(8.15,0.4)--(7.85,0.4)--(7.85,-0.4);
\node at (8,2) {$(28,[0])$};
\node at (8,3) {$(30,[1])$};
\draw (7.85,2.6)--(8.15,2.6)--(8.15,3.4)--(7.85,3.4)--(7.85,2.6);
\node at (8,4) {$(32,[0])$};
\draw (7.85,3.6)--(8.15,3.6)--(8.15,4.4)--(7.85,4.4)--(7.85,3.6);
\node at (8,5) {$(34,[1])$};
\draw (7.85,4.6)--(8.15,4.6)--(8.15,5.4)--(7.85,5.4)--(7.85,4.6);
\draw[fill=white] (7.85,5.6)--(8.15,5.6)--(8.15,6.4)--(7.85,6.4)--(7.85,5.6);
\node at (8,6) {$(36,[0])$};
\draw (8.85,5.6)--(9.15,5.6)--(9.15,6.4)--(8.85,6.4)--(8.85,5.6);
\node at (9,6) {$(39,[1])$};
\draw (9.85,5.6)--(10.15,5.6)--(10.15,6.4)--(9.85,6.4)--(9.85,5.6);
\node at (10,6) {$(42,[0])$};
\draw (10.85,5.6)--(11.15,5.6)--(11.15,6.4)--(10.85,6.4)--(10.85,5.6);
\node at (11,6) {$(45,[1])$};
\draw (11.85,5.6)--(12.15,5.6)--(12.15,6.4)--(11.85,6.4)--(11.85,5.6);
\node at (12,6) {$(48,[0])$};
\draw (11.85,4.6)--(12.15,4.6)--(12.15,5.4)--(11.85,5.4)--(11.85,4.6);
\node at (12,5) {$(46,[1])$};
\draw (11.85,3.6)--(12.15,3.6)--(12.15,4.4)--(11.85,4.4)--(11.85,3.6);
\node at (12,4) {$(44,[0])$};
\draw (11.85,2.6)--(12.15,2.6)--(12.15,3.4)--(11.85,3.4)--(11.85,2.6);
\node at (12,3) {$(42,[1])$};
\draw (11.85,1.6)--(12.15,1.6)--(12.15,2.4)--(11.85,2.4)--(11.85,1.6);
\node at (12,2) {$(40,[0])$};
\draw (11.85,0.6)--(12.15,0.6)--(12.15,1.4)--(11.85,1.4)--(11.85,0.6);
\node at (12,1) {$(38,1)$};
\draw[->] (11.85,0)--(11.15,0);
\draw[->] (10.85,0)--(10.15,0);
\draw[->] (9.85,0.05)--(9.15,0.05); 
\draw[->] (9.05,0.4)--(9.05,0.6); 
\draw[->] (8.85,1.05)--(8.15,1.05); 
\draw[->] (7.85,1)--(7.15,1);
\draw[->] (6.85,1)--(6.15,1);
\draw[->] (6,1.4)--(6,1.6);
\draw[->] (6,2.4)--(6,2.6);
\draw[->] (6,3.4)--(6,3.6);
\draw[->] (9.85,-2)--(9.15,-2); 
\draw[->] (9,-1.6)--(9,-1.4);
\draw[->] (9,-0.6)--(9,-0.4);
\draw[->] (8.95,0.4)--(8.95,0.6); 
\draw[->] (8.85,0.95)--(8.15,0.95); 
\draw[->] (8.05,1.4)--(8.05,1.6); 
\draw[->] (7.85,2)--(7.15,2);
\draw[->] (7,2.4)--(7,2.6);
\draw[->] (7,3.4)--(7,3.6);
\draw[->] (6.85,4)--(6.15,4);
\draw[->] (10,-1.6)--(10,-1.4);
\draw[->] (10,-0.6)--(10,-0.4);
\draw[->] (9.85,-0.05)--(9.15,-0.05); 
\draw[->] (8.85,0)--(8.15,0);
\draw[->] (8,0.4)--(8,0.6);
\draw[->] (7.95,1.4)--(7.95,1.6); 
\draw[->] (8,2.4)--(8,2.6);
\draw[->] (8,3.4)--(8,3.6);
\draw[->] (8,4.4)--(8,4.6);
\draw[->] (8,5.4)--(8,5.6);
\node at (12.6,6) {$\vdots$};
\draw[->] (12.5,6)--(12.15,6);
\draw[->] (11.85,6)--(11.15,6);
\draw[->] (10.85,6)--(10.15,6);
\draw[->] (9.85,6)--(9.15,6);
\draw[->] (8.85,6)--(8.15,6);
\draw[->] (12,0.4)--(12,0.6);
\draw[->] (12,1.4)--(12,1.6);
\draw[->] (12,2.4)--(12,2.6);
\draw[->] (12,3.4)--(12,3.6);
\draw[->] (12,4.4)--(12,4.6);
\draw[->] (12,5.4)--(12,5.6);
\draw (12,6.4)--(12,6.5);
\node at (11.98,6.65) {$\cdots$};

\end{scope}\end{tikzpicture}
\end{center}
\caption{$M_{3,2,2}(12,12,10,8,7,4,1,1,1)$.}\label{Big multigraph 1}
\end{figure}
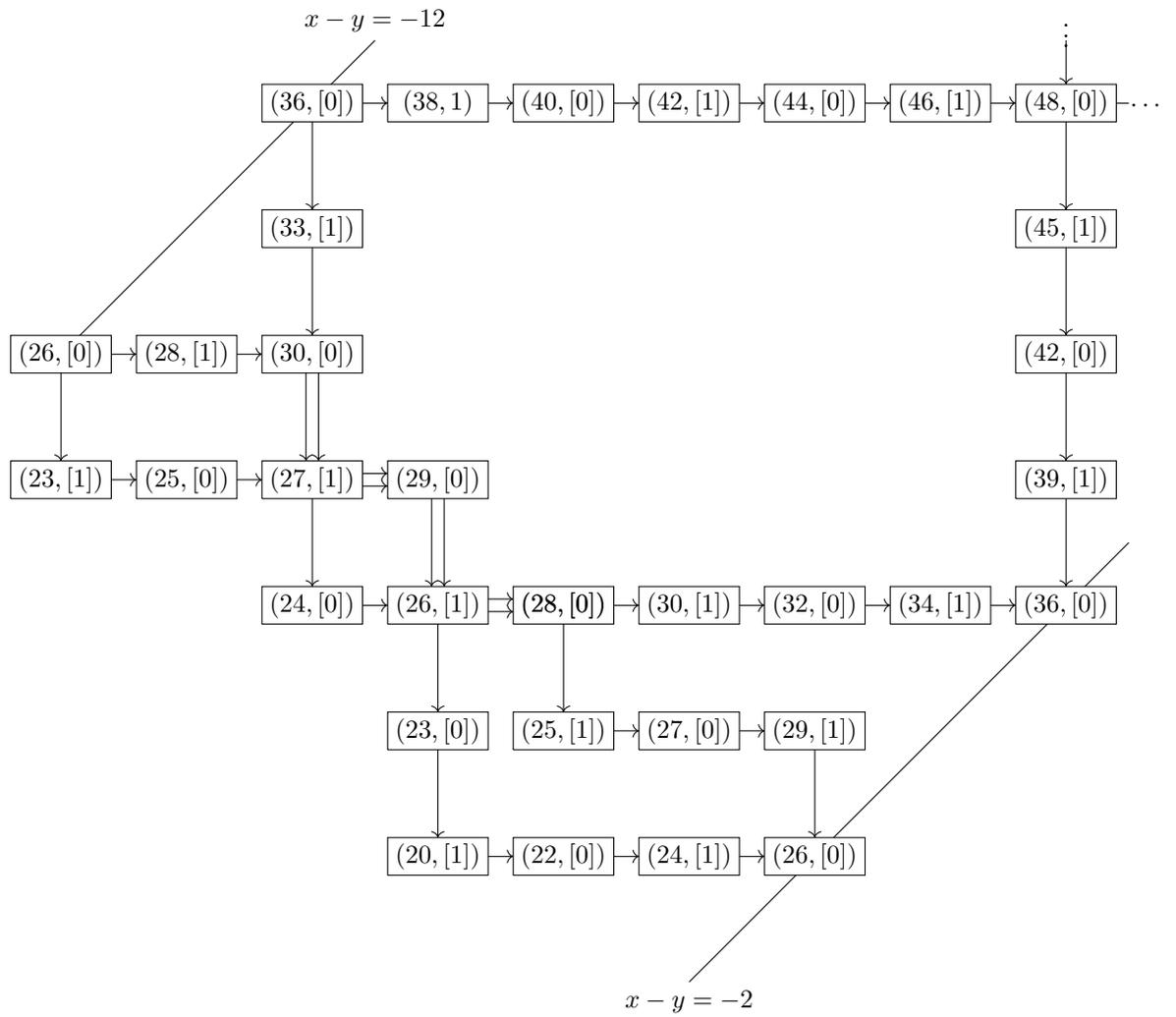
\end{ex}

\begin{rmk}
As with the boundary graph, but unlike the $c$-abacus, the direction of an edge in $M_{r,s,c}$ can be read off from its source and target. If an edge $e$ has $s(e)=(v,[i])$ and $t(e)=(w,[i+1])$ then either $d(e)=E$ and $w=v+s$ or $d(e)=S$ and $w=v-r.$
\end{rmk}

\begin{rmk}
If we act on $\mathbb{C}[x,y]$ by $T\times \mathbb{Z}/c\mathbb{Z}$ where $T=\{(t^s,t^r): t\in\mathbb{C}^*\}$ and lift to ideals as in ~\eqref{act on plane} and ~\eqref{act on ideals}, and colour boxes according to the weight of the corresponding monomial with respect to this representation, the colouring carries the same information as the multigraph.
\end{rmk}

The first property that we check is that $M_{r,s,c}(\lambda)$ determines the $c$-core of $\lambda$.
\begin{prop}\label{independence core}
If $\lambda$ and $\mu$ are partitions with $M_{r,s,c}(\lambda)=M_{r,s,c}(\mu)$ then $\lambda$ and $\mu$ have the same $c$-core.
\end{prop}
\begin{proof}
Let $v$ be large enough so that  $\left(\left\lceil\frac{v}{s}\right\rceil,0\right)$ is on the boundary of both $\lambda$ and $\mu$. Fix $m>\left\lceil\frac{v}{s}\right\rceil$ such that $c\mid m$. Then, in both the boundary tour of $\mu$ and the boundary tour of $\lambda$, every edge with index at least $m$ is an east edge. These edges account for every $E$ in an arrival word at a vertex $(w,[j])$ with $w\geq sm$. 

For each $[i],$ the number of east edges with index less than $m$, for both $\lambda$ and $\mu$, is given by $$\sum_{w<sm}\Ein(w,[i]).$$

Therefore $\lambda$, $\mu$ and $m$ satisfy the hypotheses of Corollary~\ref{keypreservecore}, and so $\lambda$ and $\mu$ have the same $c$-core.
\end{proof}
Next, we show how to read $\critplusxc(\lambda)$ and $\critminusxc(\lambda)$ off the $(r,s,c)$-tour of $M_{r,s,c}(\lambda)$. Rephrasing the first part of Proposition~\ref{motivate Mrsc} in terms of the $(r,s,c)$-tour gives
\begin{corol}\label{firstcritformulae}
Let $\lambda$ be a partition. Then

\begin{equation}\critplusxc(\lambda)=\sum_{(v,[i])\in M_{r,s,c}(\lambda)}\inv(v,[i])_a. \end{equation}
\end{corol}
A similar formula with the departure words holds for $\critminusxc(\lambda).$
\begin{corol}
\begin{equation}\critminusxc(\lambda)=\sum_{(v,[i])\in M_{r,s,c}(\lambda)}\inv(v,[i])_d. \end{equation}
\end{corol}
\begin{proof}
 By the third part of Proposition~\ref{motivate Mrsc}, $\square$ contributes to $\critminusxc(\lambda)$ if and only if the foot and hand arrive at points $(x_1,y_1)$ and $(x_2,y_2)$ respectively with \begin{equation}\label{departures rs} r+s=r(y_1-y_2)+s(x_1-x_2)\end{equation} and \begin{equation}\label{departures c}x_1-y_1\equiv x_2-y_2\pmod{c}.\end{equation} 
 The foot and hand arrive at $(x_1,y_1)$ and $(x_2,y_2)$ respectively if and only if they depart from points $(x_1-1,y_1)$ and $(x_2,y_2+1)$ respectively. The condition ~\eqref{departures c} is equivalent to \begin{equation}(x_1-1)-y_1\equiv x_2-(y_2+1)\pmod{c}.\end{equation}
 The condition~\eqref{departures rs} is equivalent to 
 \begin{equation} r+s=r(y_1-(y_2+1))+s((x_1-1)-x_2)+r+s,\end{equation}
 so subtracting $r+s$ from both sides, 
  \begin{equation} r(y_1-(y_2+1))+s((x_1-1)-x_2)=0.\end{equation}
\end{proof}

We now outline a framework for inductive proofs that the statistics in hypotheses 1-3 of Proposition~\ref{bijection properties} are determined by the multigraph $M_{r,s,c},$ using an ordering $<_{r,s,c}$ on partitions and multigraphs. The key result in this direction is Proposition~\ref{acc points are useful}.

\subsection{The order \texorpdfstring{$<_{r,s,c}$}{<rsc}}
The structure of the proofs that $M_{r,s,c}(\lambda)$ determines each property of $\lambda$ will be proven by induction on $|\lambda|$, adding a box at each step. Since the structure of $M_{r,s,c}(\lambda)$ is somewhat delicate, we have to be somewhat careful when choosing a box to add. The following ordering on partitions gives us a framework for adding boxes.

If $(x_1,y_1)$ and $(x_2,y_2)$ are two points in $\mathbb{N}^2$, say $(x_1,y_1)<_{r,s,c}(x_2,y_2)$ if either of the following hold.
\begin{itemize}
\item $sx_1+ry_1<sx_2+ry_2$;
\item $sx_1+ry_1=sx_2+ry_2,$ and $x_1-y_1\equiv x_2-y_2\pmod{c}$, and $x_1-y_1<x_2-y_2$.
\end{itemize}

The partial order $>_{r,s,c}$ on points in the plane induces a partial order $>_{r,s,c}$ on partitions as follows. Say that $\lambda'>'_{r,s,c}\lambda$ if $\lambda'$ can be obtained from $\lambda$ by adding a box with bottom left corner $(x,y)$ minimal with respect to $>_{r,s,c}$ over all possible bottom left corners of boxes that can be added to $\lambda$ to obtain a partition. Then for partitions $\mu,\lambda$ say that $\mu>_{r,s,c}\lambda$ if there is a sequence of partitions $\lambda=\lambda_0,\lambda_1,\lambda_2,\ldots,\lambda_m=\mu$ such that for each $i$, $\lambda_i<'_{r,s,c}\lambda_{i+1}.$ If $\mu>'_{r,s,c}\lambda$, say that $\mu$ is a \textit{successor} for $\lambda$ with respect to $>_{r,s,c}.$ Every partition has a successor with respect to $>_{r,s,c}$, but successors are not necessarily unique. 

\begin{ex}\label{rscexamples} Let $r=3$, $s=2$, and $c=2$. 
There are three boxes that could be added to the Young diagram of $(3,1)$ to give another partition. They have bottom left corners at $(3,0)$, $(1,1)$, and $(0,2)$, with values of $2x+3y$ of $6,5$ and $6$ respectively. So $(3,1)$ has a unique successor with respect to $<_{3,2,2}$, which is $(3,2)$.
\begin{center}
\begin{tikzpicture}\begin{scope}[yscale=1,xscale=-1,rotate=90,scale=0.4]
\draw (0,4)--(0,0);
\draw (1,3)--(1,0);
\draw (2,1)--(2,0);
\draw (0,3)--(1,3);
\draw (0,2)--(1,2);
\draw (0,1)--(2,1);
\draw (0,0)--(3,0);
\end{scope}\end{tikzpicture}
\begin{tikzpicture}\begin{scope}[yscale=1,xscale=-1,rotate=90,scale=0.4]
\draw (0,4)--(0,0);
\draw (1,3)--(1,0);
\draw (2,2)--(2,0);
\draw (0,3)--(1,3);
\draw (0,2)--(2,2);
\draw (0,1)--(2,1);
\draw (0,0)--(3,0);
\end{scope}\end{tikzpicture}
\end{center} 
For $(3,2)$, the boxes that could be added to the diagram have bottom left corners $(3,0)$, $(2,1)$ and $(0,2)$, with values of $2x+3y$ of $6, 7$ and $6$ respectively. The values of $x-y$ for $(0,3)$ and $(2,0)$ have different parity so $(2,0)\not<_{3,2,2}(0,3)$, and both $(4,2)$ and $(3,2,1)$ are successors of $(3,2)$. Note that $(4,1)\not >_{3,2,2} (3,1)$.
\end{ex}

Note that if $\mu>_{r,s,c}\lambda$, then all boxes of the Young diagram of $\lambda$ are also boxes of the Young diagram of $\mu$, but as Example~\ref{rscexamples} shows the converse is not true in general. 
\begin{defn}[Accumulation point]
For a partition $\mu$ with the property that whenever $\mu$ strictly contains $\lambda$, we also have $\mu>_{r,s,c}\lambda$, we call $\mu$ an \textit{accumulation point for $>_{r,s,c}$}.
\end{defn}

The next section describes a family of accumulation points and proves some key properties.

\subsection{The accumulation points \texorpdfstring{$\lambda_{r,s,k}$}{}}

\begin{defn}[The partition $\lambda_{r,s,k}$]\label{def:lambdarsk}
For a given natural number $k$, the partition $\lambda_{r,s,k}$ is the partition with Young diagram consisting of all boxes with top right corners on or below the line $sx+ry=k$.
\end{defn}

\begin{ex}
The Young diagram for $\lambda_{3,2,54}$ is given in Figure~\ref{3,2,54}.
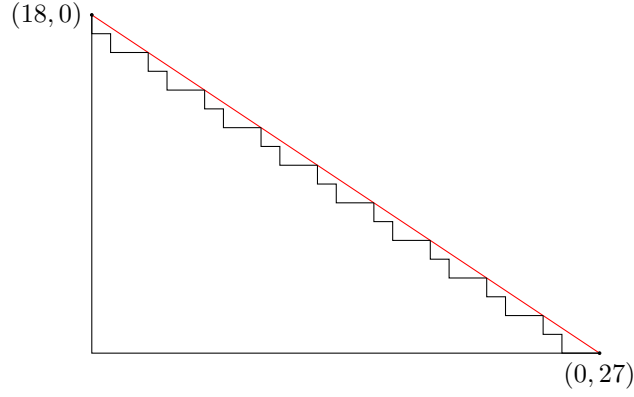
\begin{figure}[ht]
\begin{center}
\begin{tikzpicture}\begin{scope}[yscale=1,xscale=-1,rotate=90,scale=0.25]
\draw[red] (0,27)--(18,0);
\draw (18,0)--(17,0)--(17,1)--(16,1)--(16,3);
 \foreach \x in {1,...,8} {      
 \draw (2*\x, 27-3*\x)--(2*\x-1, 27-3*\x)--(2*\x-1, 27-3*\x+1)--(2*\x-2, 27-3*\x+1)--(2*\x-2, 27-3*\x+3);
 };
 \draw[fill=black] (18,0) circle[radius=0.07];
  \draw[fill=black] (0,27) circle[radius=0.07];
 \node[left] at (18,0) {$(18,0)$};
  \node[below] at (0,27) {$(0,27)$};
% \draw(0,27)--(0,30);
\draw (0,27)--(0,0)--(18,0);
\end{scope}\end{tikzpicture}
\caption{the Young diagram of $\lambda_{3,2,54}$.}\label{3,2,54}
\end{center}
\end{figure}
\end{ex}
\begin{prop}\label{accumulation point} Let $r,s,k$ be positive integers. Let $\mu$ be a partition with diagram strictly contained in the diagram of $\lambda_{r,s,k}$. Then, any successor $\mu^+$ of $\mu$ with respect to $>_{r,s,c}$ has diagram contained in the diagram of $\lambda_{r,s,k}$. In particular, $\lambda_{r,s,k}$ is an accumulation point for $\lambda_{r,s,c}.$
%Moreover, the points $(x,y)$ visited by the boundary of $\lambda_{r,s,k}$ with $sx+ry\leq k$ are all lattice points lying strictly above the line $sx+ry=k-r-s$ and below or on the line $sx+ry=k$.
\end{prop}
\begin{proof} If $(x,y)$ is the top right corner of a box in $\mu$, then since the diagram for $\mu$ is contained in the diagram of $\lambda_{r,s,k},$ $sx+ry\leq k$. So, the bottom left corner of the same box is at $(x-1,y-1)$ with $s(x-1)+r(y-1)\leq k-r-s.$
%All boxes of $\mu$ have top right corner lying on or below the line $sx+ry=k$. This implies that the bottom left corners of all boxes in $\mu$ satisfy $s(y-1)+r(x-1)\leq k-r-s$. 
Since the containment of $\mu$ in $\lambda$ is strict, there is at least one box in the diagram of $\lambda$, not contained in the diagram of $\mu$, with bottom left corner $(x-1,y-1)$ satisfying $s(x-1)+r(y-1)\leq k-r-s$. Moreover, since translating a box with top right corner $(x,y)$ left or down decreases $s(x-1)+r(y-1)$, there is a box $\square_1$ with bottom left corner $(x-1,y-1)$ that can be added to $\mu$ to give a valid partition diagram that satisfies $s(x-1)+r(y-1)\leq k-r-s$. Now, if $\mu^+$ is not contained in $\lambda$, then $\mu^+$ contains some box $\square_2$ with top right corner $(z,w)$ such that $sz+rw>k$, so the bottom left corner $(z-1,w-1)$ satisfies $s(z-1)+r(w-1)>k-r-s.$ This is a contradiction, as $(z-1,w-1)>_{r,s,c}(x-1,y-1)$, and $\square_1$ can be added to $\mu$.
\end{proof}

The accumulation points $\lambda_{r,s,k}$ will be extremely useful for two reasons. Firstly, as we check in Proposition~\ref{independence rsk}, $M_{r,s,c}(\lambda_{r,s,k})$ admits a unique $(r,s,c)$-tour whenever $rsc\mid k$, so that $M_{r,s,c}$ must determine any partition statistic in these cases, as it determines the partition itself. Secondly, as we check in Proposition~\ref{rskeventualsuccessor}, if we take successor with respect to $<_{r,s,c}$ iteratively on a given partition, we will eventually hit an accumulation point. This allows us to use the $\lambda_{r,s,k}$ as a base case for iterative proofs that statistics are independent of the choice of $(r,s,c)$-tour, and reduces the problem of understanding how a statistic interacts with $M_{r,s,c}$ to understanding how it behaves when we take successor. 

The $\lambda_{r,s,k}$ are not necessarily the only accumulation points. However, they suffice for our purposes.

\begin{ex} The partition $(1,1)$ is an accumulation point when $r=s=1$ and $c=2$, but is not a $\lambda_{1,1,k}.$ Indeed, the only successor of the empty partition is $(1)$, and the only successor of $(1)$ is $(1,1)$ since the bottom left corners of the boxes addable to $(1)$ are $(1,0)$ and $(0,1)$ with $1-0\equiv 0-1 \pmod{2}.$
\end{ex}

\begin{prop}\label{rskeventualsuccessor} If the diagram of a partition $\mu$ is contained in the diagram of $\lambda_{r,s,k}$ for some $k$, then for any sequence $$\mu=\mu_0<'_{r,s,c}\mu_1<'_{r,s,c}\cdots<'_{r,s,c}\mu_m$$ where $m=|\lambda_{r,s,k}|-|\mu|$, we must have $\mu_m=\lambda_{r,s,k}$.
\end{prop}
\begin{proof} Applying Proposition~\ref{accumulation point} to $\mu_0,\mu_1,\ldots,\mu_m,$ the diagram of $\mu_m$ must be contained in the diagram of $\lambda_{r,s,k}$, and $|\mu_m|=|\mu_0|+m=|\lambda_{r,s,k}|$, so $\mu_m=\lambda_{r,s,k}.$
\end{proof}

We now work towards proving that, in the case $rsc\mid k$, if $M_{r,s,c}(\lambda)=M_{r,s,c}(\lambda_{r,s,k}),$ then $\lambda=\lambda_{r,s,k}.$ First, we collect some restrictions on the arrival words that arise in the $(r,s,c)$-tour corresponding to $\lambda_{r,s,k}$. The condition that $rsc\mid k$ does not damage the capacity of the $\lambda_{r,s,k}$ to act as base cases, as to contain the diagram of a partition we just need $k$ to be large enough.

\begin{prop}\label{Multigraph rsk} Let $rsc\mid k$ and let $k_1=\frac{k}{rs}$. The vertices $(v,[i])$ in the multigraph of $\lambda_{r,s,k}$ all satisfy $v>k-r-s$. Moreover, we have the following constraints on the arrival words at a vertex $(v,[i]).$
%and departure words.
\begin{itemize}
%\item If $k-s<v< k$ all letters in the departure word from any vertex $(v,[i])$ are $S$s. 
%\item If $k-r-s<v \leq k-s$, all letters in the departure word from any vertex $(v,[i])$ are $E$s.
\item If $k-r-s<v \leq k-r$, then all letters in the arrival word are $S$s.
\item If $k-r<v<k$, all letters in the arrival word are $E$s. 
\item If $v=k$ there the arrival word at $(k,[0])$ has first letter $S$ and all other letters $E$. For $[i]\not=[0]$, all letters in the arrival word at $(k,[i])$ are $E$s.

\end{itemize}
\end{prop}
\begin{proof} 
If a box $\square$ has top right corner $(x_1,y_1)$ with $ry_1+sx_1\leq k-r-s$, then the $2\times 2$ box with centre $(x_1,y_1)$ contains $\square$, along with three other boxes with top right corners $(x_1+1,y_1)$, $(x_1,y_1+1)$ and $(x_1+1,y_1+1)$. 
\begin{center}
\begin{tikzpicture}\begin{scope}[yscale=1,xscale=-1,rotate=90,scale=0.9]
\draw (0,0)--(2,0)--(2,2)--(0,2)--(0,0);
\draw(1,0)--(1,2);
\draw(0,1)--(2,1);
\node[below] at (0,1) {$x_1$};
\node[below] at (0,2) {$x_1+1$};
\node[left] at (1,0) {$y_1$};
\node[left] at (2,0) {$y_1+1$};
\draw[red] (0,2.333)--(2.667,0);
\node[left] at (2.8,0) {$sx+ry=k-r-s$};
\draw[red] (1.167,3)--(2.917,1.5);
\node[right] at (1.167,3) {$sx+ry=k$};
\end{scope}\end{tikzpicture}
\end{center}
These points satisfy $ry_1+s(x_1+1)\leq k-r<k$, $r(y_1+1)+sx_1\leq k-s<k$, and $r(y_1+1)+s(x_1+1)\leq k$. Since $\lambda_{r,s,k}$ contains \textit{all} boxes with top right corners on or below the line $sx+ry=k$, the entire $2\times 2$ box with centre $(x,y)$ is contained in $\lambda_{r,s,k}$ so the boundary never visits $(x,y)$.

Suppose $k-r-s<v \leq k-r$. Any east letter in the arrival word at a vertex $(v,[i])$ is also an east letter in the departure word of some vertex $(v-s,[i-1])$, but $v-s\leq k-r-s$, so there is no such vertex. 

Suppose now that $k-r < v \leq k$. Any south letter in the arrival word at a vertex $(v,[i])$ arriving at a point $(x_1,y_1)$ with $ry_1+sx_1=v$ is also a south letter in the departure word of some vertex $(v+r,[i-1])$. We have that $v+r>k$, so the south edge cannot be the right edge of a box in the Young diagram of $\lambda_{r,s,k}$ and must be along the $y$ axis. Therefore, $x_1=0$ and $v=ry_1$ is divisible by $r$. However, by assumption $k$ is divisible by $r$ and therefore $k-r$ and $k$ are consecutive multiples of $r$. So, this is only possible if $v=k$. Since the value of $ry$ decreases as the boundary progresses south down the $y$-axis, there is only one such edge, namely, the edge departing from $\left(0,\frac{k}{r}+1\right)$ and arriving at $\left(0,\frac{k}{r}\right)$. 
\end{proof}

We are now in a position to check our base case. We will show that, if $rsc\mid k$ and $M_{r,s,c}(\mu)=M_{r,s,c}(\lambda_{r,s,k})$ then $\mu=\lambda_{r,s,k}.$ So, the accumulation point $\lambda_{r,s,k}$ act as a base case for a claim that any statistic is independent of the choice of $(r,s,c)$-tour.

\begin{prop}\label{independence rsk}
For fixed integers $r,s,c,k$ with $k=rsk_1$ and $c\mid k_1$, there is a unique $(r,s,c)$-tour of $M_{r,s,c}(\lambda_{r,s,k})$.
\end{prop}

\begin{proof} 
Suppose we pick a different $(r,s,c)$-tour of $M_{r,s,c}(\lambda_{r,s,k})$ corresponding to a partition $\mu$. First, we will show that the partition boundary of $\mu$ must leave the $y$-axis earlier than the boundary of $\lambda_{r,s,k}$. 
Let $(v,[i])$ be the vertex with $v$ maximal such that the arrival word at $(v,[i])$ changes. Such a vertex certainly exists because any partition boundary differs in finitely many edges from the boundary of the empty partition. Let $(v,[i])_a^\lambda$ and $(v,[i])_a^\mu$ be the arrival words at $(v,[i])$ in the tour corresponding to $\lambda_{r,s,k}$ and $\mu$ respectively. Then, $(v,[i])_a^\mu$ must be a permutation of $(v,[i])_a^\lambda$, so since $(v,[i])_a^\lambda\not=(v,[i])_a^\mu,$ $(v,[i])_a^\lambda$ must contain both $E$s and $S$s. Proposition~\ref{Multigraph rsk} then tells us that either
\begin{itemize}
    \item $v>k$, in which case any letter in the arrival word at $(v,[i])$ must correspond to an edge on a co-ordinate axes. Since the value of $v$ decreases as the boundary steps south along the $y$ axis, and increases as it steps east along the $x$-axis, we must have $(v,[i])_a^{\lambda_{r,s,k}}=SE$.
    \item $(v,[i])=(k,[0])$, in which case Proposition~\ref{Multigraph rsk} implies $(v,[i])_a^{\lambda_{r,s,k}}$ is an $S$ followed by a string of $E$s, where the $S$ corresponds to an edge on the $y$-axis. 
\end{itemize}
In either case, $(v,[i])_a^\mu$ must begin with an $E$. So, the boundary of $\mu$ must step east off the $y$-axis before it hits the lattice point on the $y$-axis corresponding to $(v,[i])$ - otherwise $(v,[i])_a^\mu$ would have first letter $S$. So, the boundary of $\mu$ does step east off the axis earlier than the boundary of $\lambda_{r,s,k}$. In particular, the boundary of $\mu$ never visits the point $(0,k_1s).$

Now consider the arrival words $(k,[0])^{\lambda_{r,s,k}}_a$ and $(k,[0])^{\mu}_{a}.$ Let $Z$ be the set of points $(x,y)$ in the plane in the equivalence class $(k,[0])$ with respect to $\sim_{r,s,c}$, 
\begin{equation} Z=\left\{(x,y): x,y\in\mathbb{Z}_{\geq 0}, sx+ry=k\text{ and }x-y\equiv 0\pmod{c}\right\}.\end{equation}

The length of the arrival words $(k,[0])^{\lambda_{r,s,k}}_a$ and $(k,[0])^{\mu}_{a}$ count the number of times the boundaries of $\lambda_{r,s,k}$ and $\mu$ respectively visit points in $Z$. 
Both arrival words have the same length (they are permutations of each other) so the boundaries of $\lambda_{r,s,k}$ and $\mu$ must visit the same number of lattice points in $Z$. By the definition of $\lambda_{r,s,k}$, the boundary of $\lambda_{r,s,k}$ visits \textit{all} of the points in $Z$, so the boundary of $\mu$ must also visit all $|Z|$ of these points. But the boundary of $\mu$ does not visit the point $(0,k_1s),$ a contradiction.
\end{proof}

Next, we check that there is a sensible pull back of the ordering $>_{r,s,c}$ to $(r,s,c)$-multigraphs, so that taking successor can be understood to mean something at both the level of the partition and at the level of the multigraph. We abuse notation and write $>_{r,s,c}$ for the ordering on multigraphs and partitions.

\begin{prop}\label{westarrivalnorthdeparture}
Given an $(r,s,c)$-multigraph $M$, let $V_S$ be the set of vertices $(w,[i])$ with at least one south edge arriving at $(w,[i])$. Let $(v,[i])\in V_S$ such that $v$ is minimal. Then there is an edge from $(v,[i])$ to $(v+s,[i+1])$.
\end{prop}
\begin{proof}
 At least one edge arrives at $(v,[i])$ so at least one edge departs from $(v,[i])$. Any south edge departing from $(v,[i])$ would arrive at $(v-r,[i+1]),$ so $(v-r,[i+1])$ would be in $V_S$, contradicting the minimality of $v$. Therefore at least one east edge departs from $(v,[i])$, and arrives at $(v+s,[i+1]).$
\end{proof}

\begin{defn}[Multigraph successors]
Given an $(r,s,c)$-multigraph $M$, let $(v,[i])\in V_S$ as in the previous proposition. Then we say $M^+$ is a \textit{successor} of $M$ if $M^+$ can be obtained from $M$ by deleting one south edge from $(v+r,[i-1])$ to $(v,[i])$ and one east edge from $(v,[i])$ to $(v+s,[i+1])$, and adding one east edge from $(v+r,[i-1])$ to $(v+r+s,[i])$ and one south edge from $(v+r+s,[i])$ to $(v+s,[i+1])$. Sometimes we emphasize the vertex $(v,[i])$ and say $M^+$ is a successor of $M$ that \textit{changes from $(v,[i])$}. 
\end{defn}

At the level of multigraphs, we will only need the notion of successors, but for completeness we also explicitly define $<_{r,s,c}$ at the level of multigraphs.

\begin{defn}[Ordering on multigraphs] Given $(r,s,c)$-multigraphs $M=M_{r,s,c}(\lambda)$ and $M'=M_{r,s,c}(\lambda')$ we say $M<_{r,s,c}M'$ if there is a sequence of $(r,s,c)$-multigraphs $M=M_1,\ldots,M_n=M'$ such that $M_{i+1}$ is a successor of $M_i^+$ for each $i$.
\end{defn}

\begin{corol}\label{noWoutNin(l,[i])}
If $\lambda$ is a partition with $M_{r,s,c}(\lambda)=M$ and $M^+$ is a successor of $M$ changing from $(v,[i])$, then in $M$, $\Ein(v,[i])=\Sout(v,[i])=0.$
\end{corol}
\begin{proof}
 Identical to the proof of Proposition~\ref{westarrivalnorthdeparture}.
\end{proof}
The next proposition shows that this definition of successors at the level of multigraphs aligns with our definition at the level of partitions.

\begin{prop}\label{successors pass}
Let $\lambda$ be a partition with $M_{r,s,c}(\lambda)=M.$ If $M^+$ is a successor of $M$ that changes at $(v,[i])$, then there is a unique partition $\lambda^+$ such that $\lambda^+>'_{r,s,c}\lambda$ and $M^+=M_{r,s,c}(\lambda^+).$
\end{prop}
\begin{proof}
Let $\lambda'$ be any successor of $\lambda$. Then, the Young diagram of $\lambda'$ consists of all boxes in the Young diagram of $\lambda$ and one additional box $\square$. Let the bottom left corner of $\square$ have co-ordinate $(x_1,y_1)$, where $x_1-y_1\equiv i \pmod{c}$ and $ry_1+sx_1=l.$ Then by definition of a successor, if we take minima over the points $(x,y)$ in $b(\lambda)$,
\begin{equation}
    l=\min (sx+ry)
\end{equation}and\begin{equation}\label{earliest}
    x_1-y_1=\min_{\substack{sx+ry=l \\
    [x-y]=[i]}}(x-y).
\end{equation}
In particular, $l$ and $[i]$ are sufficent to determine $x_1-y_1$. Let $s_1$ and $s_2$ be the edges in $b(\lambda)$ arriving at and departing from $(x_1,y_1)$ respectively, and let $s_1'$ and $s_2'$ be the edges in $b(\lambda')$ arriving at and departing from $(x_1+1,y_1+1)$ respectively, as shown in Figure~\ref{super important}.
\begin{figure}[ht]
\begin{center}
\begin{tikzpicture}\begin{scope}[yscale=1,xscale=-1,rotate=90, scale=1.2]
\draw (1,0)--(0,0)--(0,1);
\draw[dashed] (1,0)--(1,1)--(0,1);
\draw[dashed,->] (1,0)--(1,0.55);
\draw[->] (1,0)--(0.45,0);
\node[above left] at (0.5,0) {$s_1$};
\node[below] at (0,0.5) {$s_2$};
\node[above] at (1,0.5) {$s_1'$};
\node[right] at (0.25,1) {$s_2'$};
\draw[->] (0,0)--(0,0.55);
\draw[dashed,->] (1,1)--(0.45,1);
\draw (0,1)--(0,4)--(-3,4);
\draw (1,0)--(5,0)--(5,-2);
\draw[red] (-3,2)--(3,-2);
\draw[red] (-3,2.667)--(4,-2);
\draw[red] (-3,3)--(4.5,-2);
\draw[red] (-3,3.667)--(5.5,-2); 
\node[left] at (3,-2) {$sx+ry=l$};
\node[left] at (4,-2) {$sx+ry=l+r$};
\node[left] at (4.5,-2) {$sx+ry=l+s$};
\node[left] at (5.333,-2) {$sx+ry=l+r+s$};
\draw[fill=black] (0,0) circle[radius=0.05];
\node[below left] at (0,0) {$(x_1,y_1)$};
\end{scope}\end{tikzpicture}
\end{center}
\caption{a partition and its successor differ by replacing $s_1$ and $s_2$ with $s_1'$ and $s_2'$.}\label{super important}
\end{figure}
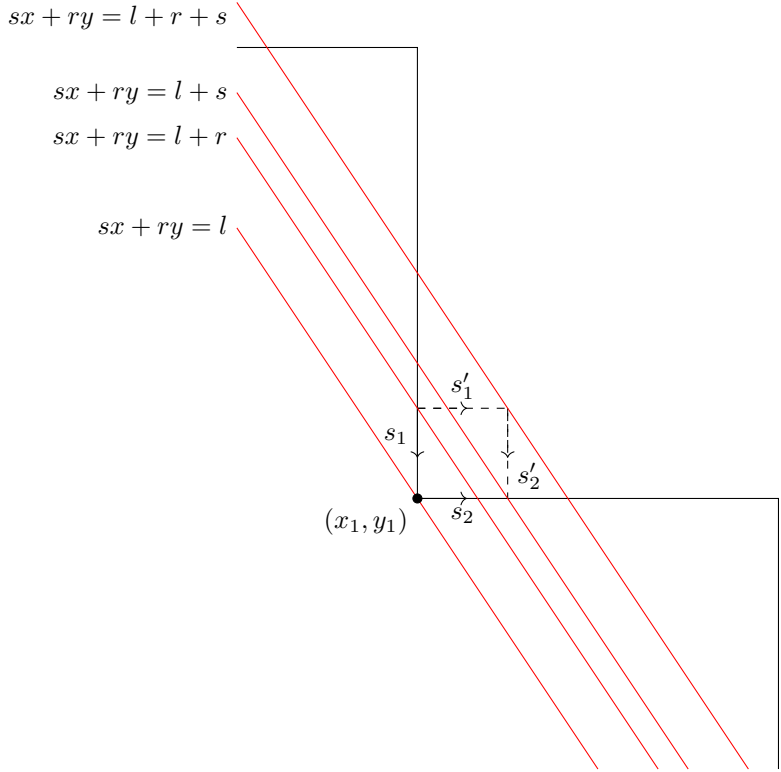
Then, the multigraph of $\lambda'$ differs from the multigraph of $\lambda$ only in that one edge from $(l+r,[i-1])$ to $(l,[i])$, and one edge from $(l,[i])$ to $(l+s,[i+1])$ corresponding to $s_1$ and $s_2$ respectively, are deleted, and one edge from $(l+r,[i-1])$ to $(l+r+s,[i])$, and one edge from $(l+r+s,[i])$ to $(l+s,[i+1])$, corresponding to $s_1'$ and $s_2'$ respectively are added. That is, $M_{r,s,c}(\lambda')$ is the successor of $M$ changing from $(l,[i])$. 

For uniqueness, given that $M^+$ changes from $M$ at $(l,[i]),$ any successor of $\lambda$ with multigraph $M^+$ must be $\lambda'$ by ~\eqref{earliest} because the value of $x-y$ increases by 1 at every consecutive point visited in the boundary. 

For existence, if $M^+$ changes from $M$ at $(v,[j])$ then $v$ is minimal such that there is a south edge into $(v,[j])$ and an east edge out of $(v,[j])$. So, $v=\min_{(x,y)\in b(\lambda)}(sx+ry)$ and there is at least one point $(x,y)$ on the boundary such that $[x-y]=[j].$ So, letting $(x_2,y_2)$ minimise $x-y$ over all such points, and adding a box with bottom left corner $(x_2,y_2)$ gives a successor $\lambda^+$ of $\lambda$ with multigraph $M^+$.  \end{proof}

We are now in a position to prove our key structural proposition.

\begin{prop}\label{acc points are useful}
Let $f:\Par\to\mathbb{R}.$ Suppose there is a function $g:\{M_{r,s,c}(\lambda)\mid\lambda\in\Par\}^2\to \mathbb{R}$ such that, if $\lambda$ is a partition, and $\lambda^+$ is a successor of $\lambda,$ where $\lambda$ and $\lambda^+$ have $(r,s,c)$-multigraphs $M$ and $M^+$ respectively, \begin{equation} f(\lambda^+)-f(\lambda)=g\left(M^+,M\right).\end{equation}

Then, for any partitions $\mu_1$ and $\mu_2$ with $M_{r,s,c}(\mu_1)=M_{r,s,c}(\mu_2)$, $f(\mu_1)=f(\mu_2).$
\end{prop}
\begin{proof} Let $M=M_{r,s,c}(\mu_1)=M_{r,s,c}(\mu_2)$. There is a sequence of multigraphs  $M=M_0,M_1,\ldots$ where $M_j$ is a successor of $M_{j-1}$ for each $j$. Set $\lambda_0=\mu_1$ and $\nu_0=\mu_2$. Then, by Proposition~\ref{successors pass} there are sequences of partitions $\lambda_0,\lambda_1,\ldots$ and $\nu_0,\nu_1,\ldots$ such that $M_j=M_{r,s,c}(\lambda_j)=M_{r,s,c}(\nu_j)$, \begin{equation}\lambda_0<'_{r,s,c}\lambda_1<'_{r,s,c}\ldots,\end{equation} and
\begin{equation}\nu_0<'_{r,s,c}\nu_1<'_{r,s,c}\ldots.\end{equation} 
Let $k$ be divisible by $rsc$ and large enough so that all boxes in the Young diagrams of $\mu_1$ or $\mu_2$ lie below the line $sx+ry=k$. By Proposition~\ref{rskeventualsuccessor}, there is some $m$ such that $M_m=M_{r,s,c}(\lambda_{r,s,k})$. By Proposition~\ref{independence rsk}, $\lambda_m=\nu_m=\lambda_{r,s,k}$. 
Then, 
\begin{align} f(\mu_1)&=f(\lambda_{r,s,k})-\sum_{i=1}^m\left(f(\lambda_{i})-f(\lambda_{i-1})\right)\\
&=f(\lambda_{r,s,k})-\sum_{i=1}^m g(M_i,M_{i-1})\\
&=f(\lambda_{r,s,k})-\sum_{i=1}^m\left(f(\nu_{i})-f(\nu_{i-1})\right)\\
&=f(\mu_2).
\end{align} \end{proof}

Armed with Proposition~\ref{acc points are useful}, checking that $M_{r,s,c}$ determines the area of a partition is particularly straightforward.

\begin{corol}\label{areacorol}
If $\lambda$ and $\mu$ are partitions with $M_{r,s,c}(\lambda)=M_{r,s,c}(\mu)$ then $|\lambda|=|\mu|.$
\end{corol}

\begin{proof} Apply Proposition~\ref{acc points are useful} with $g(M^+,M)=1.$\end{proof}

Having outlined the structure of the proofs that $M_{r,s,c}$ determines partition statistics, we defer the checks that $M_{r,s,c}$ determines $\middxc$ and $\critplusxc+\critminusxc$ to Section 6. We now turn our attention to defining $I_{r,s,c}$.

\section{The involution \texorpdfstring{$I_{r,s,c}$}{Irsc}}
\stepcounter{essaypart}
In this section we construct the bijection $I_{r,s,c}$ and check that it is well defined. In order to do so, we first need to understand how to recover a partition from a family of arrival words.

\subsection{Recovering a partition from the arrival words}
Thus far we have constructed $M_{r,s,c}(\lambda)$ and an $(r,s,c)$-tour from $b(\lambda)$. We will define $I_{r,s,c}$ as an involution that preserves $M_{r,s,c}$ but changes the $(r,s,c)$-tour, in fact by changing the order in which some of the letters appear in the arrival words. In order to check the result is well defined, we need to understand how to recover a boundary sequence from a family of arrival words, and indeed have a criterion for when it is possible to do so if the family of arrival words does not a priori arise from a partition.

If $v$ is minimal such that all boxes in the partition have top right corner on or below the line $sx+ry=v$, then we have that for all $w>v$, any arrival at a vertex $(w,[i])$ must be on a co-ordinate axis. So, 
\begin{equation}\label{rscboringvertex}(w,[i])_a=\begin{cases} 
      SE & \text{if }r\mid w, s\mid w, \frac{w}{s}\equiv \frac{-w}{r}\equiv i\pmod{c} \\
      E & \text{if }s\mid w, c\mid\left(\frac{w}{s}-i\right)\text{ and either }r\nmid w\text{ or }c\nmid(\frac{-w}{r}-i)\\ 
      S & \text{if }r\mid w, c\mid\left(\frac{-w}{r}-i\right)\text{ and either }s\nmid w\text{ or }c\nmid(\frac{w}{s}-i) \\
      \emptyset &\text{otherwise.} \\
   \end{cases}\end{equation}
Moreover, $v$ is uniquely specified as the largest vertex where the arrival word at $(v,[i])$ does \textit{not} satisfy \eqref{rscboringvertex} for some $i\in\{1,2,\ldots,c\}$.

So, we can identify $v$ and fill in the co-ordinate axes above or to the right of the line $sx+ry=v$ as part of the partition boundary. We may then fill in the remainder working backwards from the arrival words - we outline the method below by example.

\begin{ex}\label{fillinaxes} Suppose we have $r=3$, $s=2$, $c=2$, and the set of arrival words specified below
\begin{center}
\begin{tabular}{cccccccc}
(20,[1])&S&(22,[0])&E&(23,[0])&S&(23,[1])&S\\
(24,[0])&S&(24,[1])&E&(25,[0])&E&(25,[1])&S\\
(26,[0])&SE&(26,[1])&SSE&(27,[0])&E&(27,[1])&SSE\\
(28,[0])&EE&(28,[1])&E&(29,[0])&SS&(29,[1])&E\\
(30,[0])&SE&(30,[1])&E&&&&\\
\end{tabular}
\end{center}
Empty for all other vertices $(w,[j])$ with $w<30$.
Then for $w>30$, 

\begin{equation*}(w,[j])_a=\begin{cases} 
      SE & \text{if }6\mid w, \frac{w}{2}\equiv \frac{-w}{3}\equiv j\pmod{2} \\
      E & \text{if }2\mid w, 2\mid\left(\frac{w}{2}-j\right)\text{ and either }3\nmid w\text{ or }2\nmid(\frac{-w}{3}-j)\\ 
      S& \text{if }3\mid w, 2\mid\left(\frac{-w}{3}-j\right)\text{ and either }3\nmid w\text{ or }2\nmid(\frac{w}{2}-j) \\
      \emptyset &\text{otherwise.} \\
   \end{cases}\end{equation*}

Looking at the vertex $(30,[0])$, with $w=30$ and $j=0$ we have $\frac{w}{2}\not\equiv j\pmod{2}$, so 30 is maximal such that there is vertex $(30,[i])$ that does not satisfy~\eqref{rscboringvertex} for some $i$. So, $v=30$ and we draw a ray along the positive $x$-axis beginning at $(15,0),$ and a ray along the positive $y$-axis beginning at $(0,10).$ It then remains to fill in the boundary between the points $\left(\left\lceil\frac{v}{s}\right\rceil,0\right),$ and  $\left(0,\left\lceil\frac{v}{r}\right\rceil\right).$ To do this, we look first at the arrival word at $\left(s\left\lceil\frac{v}{s}\right\rceil,\left[\left\lceil\frac{v}{s}\right\rceil\right]\right)$, $(30,[1])$ in our example, corresponding to the point on the $x$-axis at which the ray begins. The last letter of this word tells us what kind of edge we should add to the boundary to arrive at $\left(\left\lceil\frac{v}{s}\right\rceil,0\right)$, in this case an $E$, so we add an edge from $(14,0)$ to $(15,0)$. and delete the last $E$ from $(30,[1])_a$. The same logic allows the rest of the boundary to be filled out edge by edge, as in Figure~\ref{fillinaxesfigure}.
\begin{figure}[!htbp]
\begin{center}
\begin{tikzpicture}\begin{scope}[yscale=0.9,xscale=-0.9,rotate=90,scale=0.7]
\draw (0,17)--(0,15);
\node[below] at (0,15) {$(15,0)$};
\draw (10,0)--(15,0);
\draw[gray] (0,14)--(0,15);
\draw[gray] (0,14) circle[radius=0.05];
\node[rotate=45, below left] at (0,13.75) {$(28,[0])_a=E$\sout{$E$}};
\draw[gray] (0,13)--(0,14);
\draw[gray] (0,13) circle[radius=0.05];
\node[rotate=45,below left] at (0,12.75) {$(26,[1])_a=SS$\sout{$E$}};
\draw[gray] (0,12)--(0,13);
\draw[gray] (0,12) circle[radius=0.05];
\node[rotate=45,below left] at (0,11.75) {$(24,[0])_a=$\sout{$S$}};
\draw[gray] (0,12)--(1,12);
\draw[gray] (1,12) circle[radius=0.05];

\node[rotate=45, right] at (1,12) {$(27,[1])_a= SE$\sout{$S$}};
%\node[below] at (1,11) ;
\draw[gray] (1,12)--(2,12);
\draw[gray] (2,12) circle[radius=0.05];

\node[rotate=45, right] at (2,12) {$(30,[0])_a=SE$\sout{$S$}};
\draw[gray] (2,11)--(2,12);
\draw[gray] (2,11) circle[radius=0.05];

\node[rotate=45, below left] at (2,10.75) {$(28,[1])_a=$\sout{E}};
\draw[gray] (2,11)--(2,10);
\draw[gray] (2,10) circle[radius=0.05];

\node[rotate=45, below left] at (2,9.75) {$(26,[0]): E$\sout{S}};
\draw[gray] (3,10)--(2,10);
\draw[gray] (3,10) circle[radius=0.05];

\node[rotate=45, right] at (3,10) {$(29,[1])_a=$\sout{$E$}};
\draw[gray] (3,10)--(3,9);
\draw[gray] (3,9) circle[radius=0.05];

\node[rotate=45, below left] at (3,8.75) {$(27,[0])_a=E$\sout{$E$}};
\draw[gray] (3,8)--(3,9);
\draw[gray] (3,8) circle[radius=0.05];

\node[rotate=45, below left] at (3,7.75) {$(25,[1])_a=$\sout{$S$}};
\draw[gray] (3,8)--(4,8);
\draw[gray] (4,8) circle[radius=0.05];

\node[rotate=45, right] at (4,8) {$(28,[0])_a=E$\sout{$E$}};
\draw[gray] (4,8)--(4,7);
\draw[gray] (4,7) circle[radius=0.05];

\node[rotate=45, below left] at (4,6.75) {$(26,[1])_a=S$\sout{$SE$}};
\draw[gray] (5,7)--(4,7);
\draw[gray] (5,7) circle[radius=0.05];

\node[rotate=45, above right] at (5,7) {$(29,[0])_a=E$\sout{$E$}};
\draw[gray] (5,6)--(5,7);
\draw[gray] (5,6) circle[radius=0.05];

\node[rotate=45, below left] at (5,5.75) {$(27,[1])_a=S$\sout{$SE$}};
\draw[gray] (5,6)--(5,5);
\draw[gray] (5,5) circle[radius=0.05];

\node[rotate=45, below left] at (5,4.75) {$(25,[0])_a=$\sout{$E$}};
\draw[gray] (5,4)--(5,5);
\draw[gray] (5,4) circle[radius=0.05];

\node[rotate=45, below left] at (5,3.75) {$(23,[1])_a=$\sout{$S$}};
\draw[gray] (5,4)--(6,4);
\draw[gray] (6,4) circle[radius=0.05];

\node[rotate=45, right] at (6,4) {$(26,[0])_a=$\sout{$SE$}};
\draw[gray] (6,4)--(6,3);
\draw[gray] (6,3) circle[radius=0.05];

\node[rotate=45, below left] at (6,2.75) {$(24,[1])_a=$\sout{$E$}};
\draw[gray] (6,2)--(6,3);
\draw[gray] (6,2) circle[radius=0.05];

\node[rotate=45, below left] at (6,1.75) {$(22,[0])_a=$\sout{$E$}};
\draw[gray] (6,2)--(6,1);
\draw[gray] (6,1) circle[radius=0.05];

\node[rotate=45, below left] at (6,0.75) {$(20,[1])_a=$\sout{$S$}};
\draw[gray] (7,1)--(6,1);
\draw[gray] (7,1) circle[radius=0.05];
\coordinate (G) at (7.2,0.8);
\node[rotate=45, right] at (7,1) {$(23,[0])_a=$\sout{$S$}};
\draw[gray] (7,1)--(8,1);
\draw[gray] (8,1) circle[radius=0.05];

\node[rotate=45, right] at (8,1) {$(26,[1])_a=$\sout{$SSE$}};
\draw[gray] (9,1)--(8,1);
\draw[gray] (9,1) circle[radius=0.05];

\node[rotate=45, right] at (9,1) {$(29,[0])_a=$\sout{$EE$}};

\node[rotate=45, below left] at (9,-0.25) {$(27,[1])_a=$\sout{$SSE$}};
\draw[gray] (9,1)--(9,0);
\draw[gray] (9,0) circle[radius=0.05];
\draw[gray] (9,0)--(10,0);

\node[left] at (10,0) {$(0,10)$};
%\draw[red] (-2,18)--(12,-3);
\end{scope}\end{tikzpicture}
\end{center}
\caption{the arrival words in Example~\ref{fillinaxes} give the partition $(12,12,10,8,7,4,1,1,1)$.}\label{fillinaxesfigure}
\end{figure}
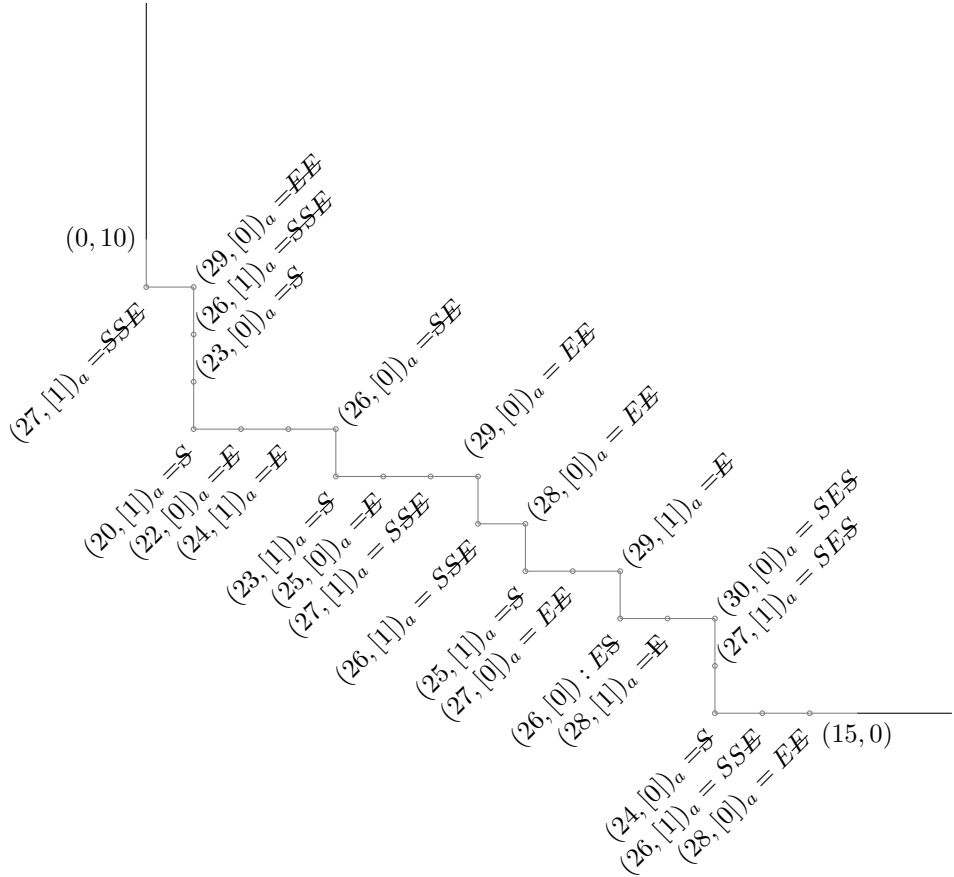
\end{ex}

\subsection{The first arrival tree}
Next we lay out a criterion for a family of arrival words to arise from a partition. We already know that any family of arrival words arising from a partition must satisfy ~\eqref{rscboringvertex} for $w>v$ large enough. We now give a criterion on the arrival words at the remaining vertices with $w\leq v$ to arise from a partition. 

\begin{defn}[First arrival graph]
Let $\lambda$ be a partition and let $M_{r,s,c}(\lambda)=M$. Let $V$ and $E$ be the vertex set and edge set of $M$ respectively. Let the $(r,s,c)$-tour of $M$ corresponding to $\lambda$ have arrival word $(v,[i])_a$ at each vertex $(v,[i])\in V$. Suppose there is another family of arrival words  \begin{equation} S= \{(v,[i])'_a: (v,[i])\in V\},
    \end{equation} such that for each $(v,[i])$, $(v,[i])'_a$ is a permutation of $(v,[i])_a$.
    Denote the first letter of the arrival word $(v,[i])'_a$ by $(v,[i])'_1$, and let the first arrival edge $e_1(v,[i])$ with respect to $S$ be any edge $e$ with $t(e)=(v,[i])$ and $d(e)=(v,[i])'_1.$ Let $T_S$ be the subgraph of $M$ with vertex set $V$ and directed edge set \begin{equation}E(T_S)=\{e_1(v,[i])'_a: (v,[i])\in M\}.\end{equation} In this case we call $T_S$ the \textit{first arrival graph with respect to $S$}. 
\end{defn}

\begin{defn}[The graphs $M^{\leq k}$ and $T^{\leq k}$]
For an integer $k$, and an $(r,s,c)$-multigraph $M,$ let $M^{\leq k}$ be the induced subgraph of $M$ with vertex set $$V(M^{\leq k})=\{(v,[i]): (v,[i])\in V(M)\text{ and }v\leq k \}.$$ For a family of arrival words $S$, let $T_S^{\leq k}$ be the induced subgraph of $T_S$ with vertex set $V(M^{\leq k}).$
\end{defn}

We require some preparation before proving Proposition~\ref{image works}, as we make use of \cite[Thm 5]{vAE}. The proof is not hard, but could possibly be disruptive to the flow of this paper, so the interested reader is referred to \cite{vAE} for a full proof. The notion we will need is that of a $T$-graph.
\begin{defn}[$T$-graph] A $T$-graph is a finite directed multigraph such that at each vertex, the number of edges arriving is the same as the number of edges departing. 
\end{defn}

Theorem 5a of~\cite{vAE} says that, given a complete circuit of a $T$-graph, starting and ending at a vertex $v,$ the set of edges given by the last departures from any given vertex give a spanning tree of the $T$-graph rooted at $v$. Reversing the direction of all edges, equivalantly, the first arrival graph arising from a complete circuit of a $T$-graph is a spanning tree. Conversely, \cite[Thm 5b]{vAE} says that any spanning tree rooted at $v$ gives rise to a complete circuit with last departures (or equivalently, first arrivals) agreeing with the edges of the spanning tree.

In order to apply these theorems to our situation, we need to separate $M$ into a $T$-graph and a well understood complement, which is how we prove Proposition~\ref{image works}.

\begin{prop}\label{image works}
Let $\lambda$ be a partition and let $M_{r,s,c}(\lambda)=M$. Suppose there is a family of arrival words  $S= \{(v,[i])'_a: (v,[i])\in V\}$ assigned to $M$. Let $T_S$ be the first arrival graph with respect to $S$. Then there is a partition $\mu$ with an $(r,s,c)$-tour having arrival words $S$ if and only if both of the following hold.
\begin{enumerate}
    \item There exists some $v$ such that for all $w>v$, and all $j$, $(w,[j])'_a$ satisfies~\eqref{rscboringvertex}.
    \item $T_S$ is a spanning tree of $M$.
\end{enumerate}
\end{prop}
\begin{proof} The first condition has already been shown to be necessary, so we prove that assuming the first condition holds, the second condition is equivalent to the existence of $\mu$.
Fix $v$ such that for all $w\geq v$,~\eqref{rscboringvertex} holds for both $(w,[j])_a$ and $(w,[j])_a'.$ Let $k\geq v$ be such that $k=rsk_1$ for some integer $k_1$ with $c\mid k_1$. Let $M^{\leq k}$ and $M^{>k}$ be the induced subgraphs of $M$ with vertex sets given by \begin{equation}\label{Mleq k}
    V(M^{\leq k})=\{(v,[i])\in V(M) : v\leq k\},
\end{equation}\begin{equation}
    V(M^{> k})=\{(k,[0])\}\cup\{(v,[i])\in V(M) : v > k\},
\end{equation}
and let $T_S^{\leq k}$ and $T_S^{>k}$ be the induced subgraphs of $T_S$ with vertex sets $V(M^{\leq k})$ and $V(M^{> k})$. Then $T_S^{>k}$ is a spanning tree of $M^{>k}$. 

\begin{figure}[ht]
\begin{center}
    \begin{tikzpicture}[scale=1.8]
%    \draw (8.7,5.85)--(9.3,5.85)--(9.3,6.15)--(8.7,6.15)--(8.7,5.85);
\node at (0,0) {$(k,[0])$};
%\draw (9.7,5.85)--(10.3,5.85)--(10.3,6.15)--(9.7,6.15)--(9.7,5.85);
%\node at (2,0) {$(k+r,[-1])$};
%\draw (10.7,5.85)--(11.3,5.85)--(11.3,6.15)--(10.7,6.15)--(10.7,5.85);
%\node at (4.5,0) {$(k+2r,[-2])$};
%\draw (11.7,5.85)--(12.3,5.85)--(12.3,6.15)--(11.7,6.15)--(11.7,5.85);
%\node at (7,0) {$(k+3r,[-3])$};
\draw[->](0.7,4)--(1.1,4);
\draw[->](3.05,4)--(3.8,4);
%\draw[->](5.4,4)--(5.8,4);
%\draw[->](8.5,0)--(8,0);
%\node at (6,3.97) {$\cdots$};
\draw[->] (4.25,4)--(4.9,4);
\node at (6.1,0) {$(k+(cr-1)s,[0])$};
\node at (2.1,4) {$(k+(cr+1)s,[1])$};
\node at (4,3.97) {$\cdots$}; 
%{$(k+(cr+2)s,[2])$};
\node at (6,4) {$(k+(2cr-1)s,[-1])$};

\node at (0,1) {$(k+r,[-1])$};
\node at (0,2) {$(k+2r,[-2])$};
\draw[->] (0,0.8)--(0,0.2);
\draw[->] (0,1.8)--(0,1.2);
\draw[->] (0,2.8)--(0,2.2);
\node at (0,3) {$\vdots$};
\draw[->] (0,3.8)--(0,3.2);
\node at (0,4) {$(k+crs,[0])$};

\node at (2.1,0) {$(k+s,[1])$};
\node at (3.6,0) {$(k+2s,[2])$};
\draw[->](0.4,0)--(1.5,0);
\draw[->](2.7,0)--(3,0);
\draw[->](4.2,0)--(4.55,0);
%\draw[->](8.5,0)--(8,0);
\node at (4.7,-0.03) {$\cdots$};
\draw[->] (4.85,0)--(5.2,0);
\draw[->] (0,4.8)--(0,4.2);
\node at (0,5) {$\vdots$};
\end{tikzpicture}
\end{center}
\caption{$T_S^{>k}$ in the case $(c,r+s)=1$.}
\end{figure}
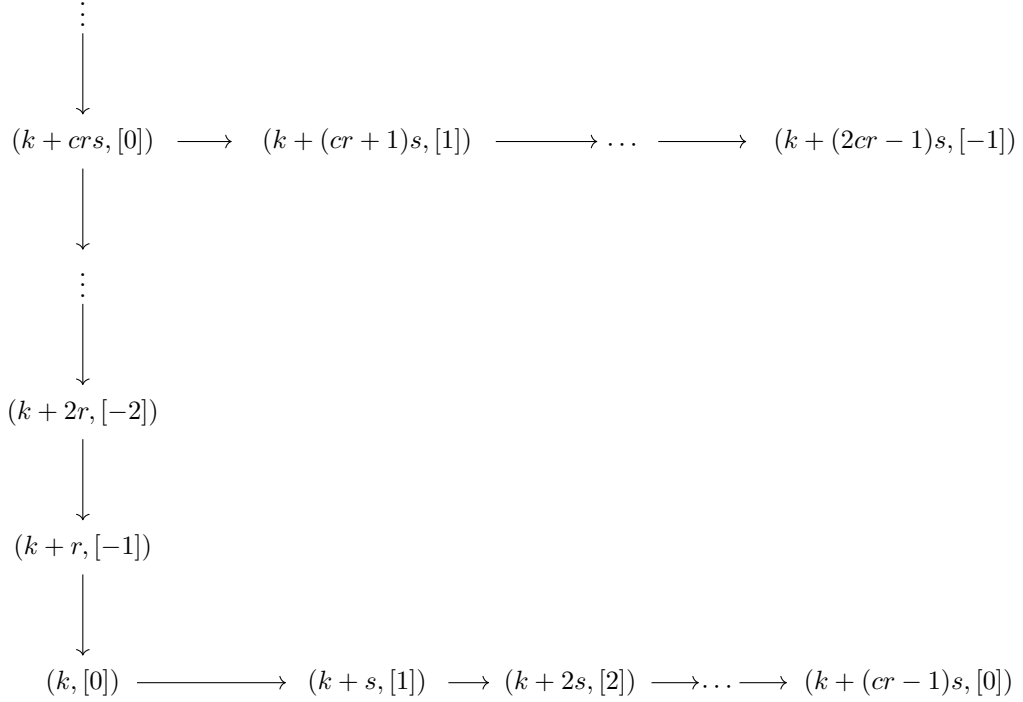

Let $x=(k_1r,0)$ and $y=(0,k_1s)$ on the boundary. Then, the edges in $M^{>k}$ correspond to the rays along the axes starting at $x$ and $y$. The $(r,s,c)$-tour corresponding to $\lambda$ restricted to $M^{\leq k}$ is a complete circuit starting and finishing at $(k,[0])$, and each edge corresponds to an edge in the boundary of $\lambda$ that occurs after the south edge arriving at $y$ and occurs before the east edge departing from $x$. So, $M^{\leq k}$ contains $|x|$ south edges and $|y|$ east edges. Therefore, there is a partition $\mu$ with arrival words $S$ if and only if there is a complete circuit of $M^{\leq k}$ such that the arrival words agree with $S$.

%So, in order to prove the second condition is necessary, it suffices to show that $T_S^{\geq k}$ is a spanning tree of $M^{\leq k}$ rooted at $(k,[0]).$  

Assume that a complete circuit of $M^{\leq k}$ with arrival words as given in $S$ exists. The $(r,s,c)$-tour of $M$ corresponding to $\lambda$ consists of a circuit of $M^{>k}$ and $M^{\leq k},$ so the in-degree of any vertex $v$ of $M^{\leq k}$ is equal to the out-degree of $v$ in $M^{\leq k}$ and $M^{\leq k}$ is connected. In particular, $M^{\leq k}$ is a $T$-graph. So, ~\cite[Thm 5a]{vAE} implies that $T_S^{\leq k}$ is a tree rooted at $(k,[0]).$ Therefore, $T_S$ is a spanning tree of $M$.

Now assume that $T_S$ is a spanning tree of $M$. Then, $T_S^{\leq k}$ is a tree rooted at $(k,[0])$ and ~\cite[Thm 5b]{vAE} implies that there is a complete circuit of $M^{\leq k}$ with arrival words agreeing with $S$.\end{proof}

From now on, let $\lambda<_{r,s,c}\lambda_{r,s,k}$ where $k=rsk_1$ and $c\mid k_1$. Let $M_{r,s,c}(\lambda)=M$, let $S$ be the family of arrival words corresponding to $\lambda$. We will now refer to $T_S$ as the first arrival tree.

When we have a drawing of $M_{r,s,c}(\lambda)$ on the cylinder defined in Proposition~\ref{lattice points}, and a directed path $p$ from $(v,[i])$ to $(w,[j])$, we define the \textit{winding number of $p$} to be the number of times strictly after leaving $(v,[i])$ and before or on arriving at $(w,[j])$ that $p$ arrives at a vertex on the upper boundary strip. We will be particularly interested in the case where $(v,[i])=(k,[-k_1s])$ and $p$ is the unique path in the first arrival tree from $(k,[-k_1s])$ to $(w,[j]).$

\begin{defn}[Switch, eastern, southern]
Given a partition $\lambda$ with $\lambda<_{r,s,c} \lambda_{r,s,k}$, and $(r,s,c)$-multigraph $M$, let  $T$ denote the first arrival tree of $M$ corresponding to $\lambda$. Let $(v,[i])\in V(M)$ have $v\leq k$ and $(v,[i])\not=(k,[0]).$ Then $(v,[i])$ is a \textit{switch} if $(v+r,[i-1])$ and $(v-s,[i-1])$\footnote{these are the two equivalence classes that, if they are vertices of $M$, could form a tail of an edge to $(v,[i])$} are both vertices of $M$, and the distances in $T$ from the vertex $(k,[0])$ to $(v+r,[i-1])$, $(v-s,[i-1])$ are equal.
%This distance must be 1 less than the distance in $T$ from $(k,[0])$ to $(v,[i])$.
Now drop the condition that $v\leq k$ and $(v,[i])\not=(k,[0]).$
If $(v,[i])$ is not a switch and the first letter in the arrival word is $E$, say $(v,[i])$ is \textit{eastern}, and let $\Ea$ be the set of all eastern vertices $(v,[i])$ with $v\leq k$. 
If $(v,[i])$ is not a switch and the first letter in the arrival word is $S$, say $(v,[i])$ is \textit{southern}, and let $\So$ be the set of all southern vertices $(v,[i])$ with $v\leq k$. 
\end{defn}

\begin{ex} For $\mu=(12,12,10,8,7,4,1,1,1)$, the first arrival tree of $M_{3,2,2}(\mu)$ is as in Figure~\ref{Big tree 1}.

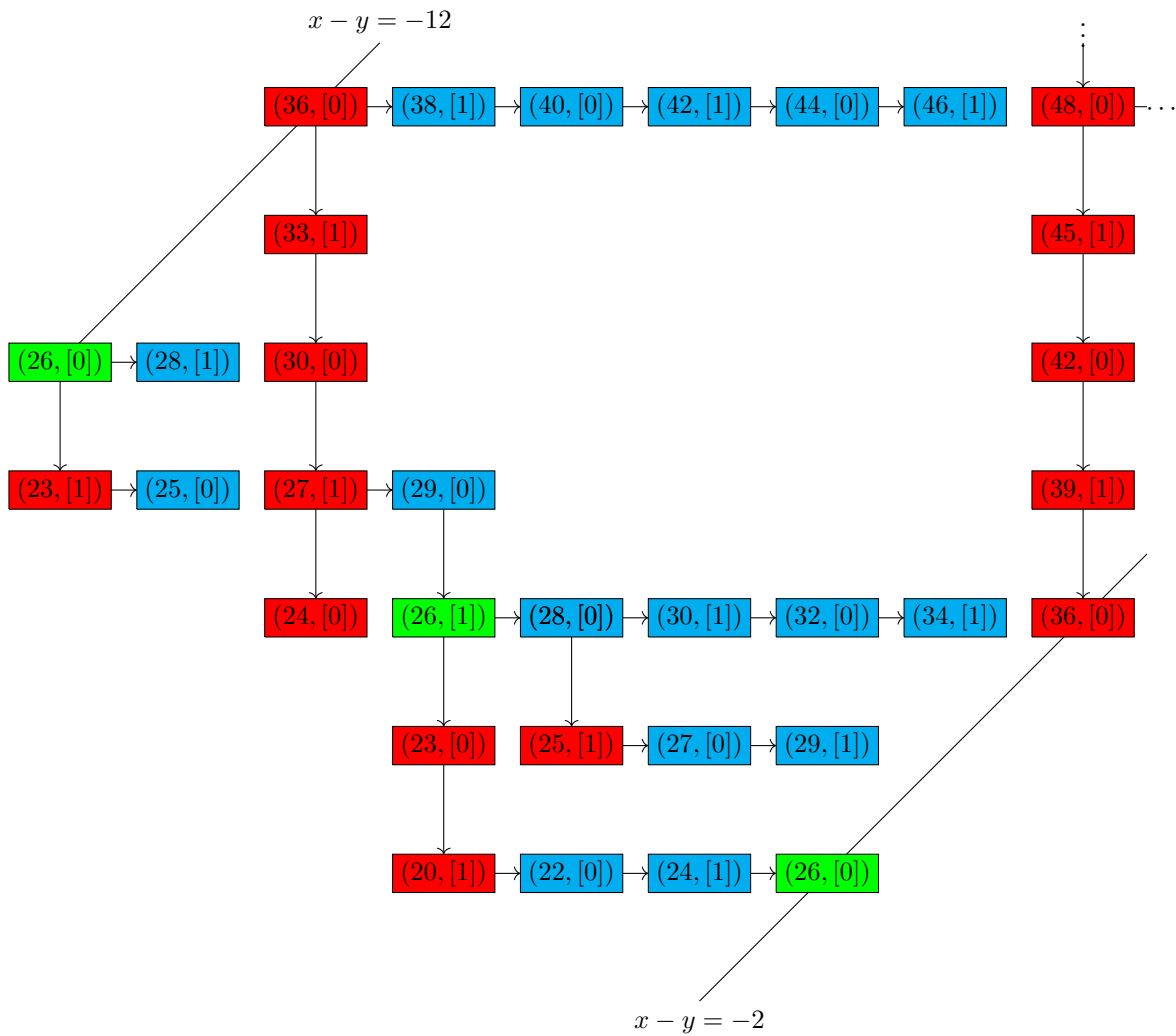
\begin{figure}[!htbp]
\begin{center}
\begin{tikzpicture}\begin{scope}[yscale=1,xscale=-1,rotate=90,scale=1.7]
\draw (5,3)--(8.5,6.5);
\node[below] at (5,3) {$x-y=-2$};
\draw (12.5,0.5)--(10,-2);
\node[above] at (12.5,0.5) {$x-y=-12$};
\draw[fill=red] (11.85,-0.4)--(12.15,-0.4)--(12.15,0.4)--(11.85,0.4)--(11.85,-0.4);
\node at (12,0) {$(36,[0])$};
\draw[fill=red] (10.85,-0.4)--(11.15,-0.4)--(11.15,0.4)--(10.85,0.4)--(10.85,-0.4);
\node at (11,0) {$(33,[1])$};
\draw[fill=red] (9.85,-0.4)--(10.15,-0.4)--(10.15,0.4)--(9.85,0.4)--(9.85,-0.4);
\node at (10,0) {$(30,[0])$};
\draw[fill=red] (8.85,-0.4)--(9.15,-0.4)--(9.15,0.4)--(8.85,0.4)--(8.85,-0.4);
\node at (9,0) {$(27,[1])$};
\draw[fill=cyan] (8.85,0.6)--(9.15,0.6)--(9.15,1.4)--(8.85,1.4)--(8.85,0.6);
\node at (9,1) {$(29,[0])$};
\draw[fill=green] (7.85,0.6)--(8.15,0.6)--(8.15,1.4)--(7.85,1.4)--(7.85,0.6);
\node at (8,1) {$(26,[1])$};
\draw[fill=red] (6.85,0.6)--(7.15,0.6)--(7.15,1.4)--(6.85,1.4)--(6.85,0.6);
\node at (7,1) {$(23,[0])$};
\draw[fill=red] (5.85,0.6)--(6.15,0.6)--(6.15,1.4)--(5.85,1.4)--(5.85,0.6);
\node at (6,1) {$(20,[1])$};
\draw[fill=cyan] (5.85,1.6)--(6.15,1.6)--(6.15,2.4)--(5.85,2.4)--(5.85,1.6);
\node at (6,2) {$(22,[0])$};
\draw[fill=cyan] (5.85,2.6)--(6.15,2.6)--(6.15,3.4)--(5.85,3.4)--(5.85,2.6);
\node at (6,3) {$(24,[1])$};
\draw[fill=green] (5.85,3.6)--(6.15,3.6)--(6.15,4.4)--(5.85,4.4)--(5.85,3.6);
\node at (6,4) {$(26,[0])$};
\draw[fill=green] (9.85,-2.4)--(10.15,-2.4)--(10.15,-1.6)--(9.85,-1.6)--(9.85,-2.4);
\node at (10,-2) {$(26,[0])$};
\draw[fill=red] (8.85,-2.4)--(9.15,-2.4)--(9.15,-1.6)--(8.85,-1.6)--(8.85,-2.4);
\node at (9,-2) {$(23,[1])$};
\draw[fill=cyan] (8.85,-1.4)--(9.15,-1.4)--(9.15,-0.6)--(8.85,-0.6)--(8.85,-1.4);
\node at (9,-1) {$(25,[0])$};
\draw[fill=cyan] (7.85,1.6)--(8.15,1.6)--(8.15,2.4)--(7.85,2.4)--(7.85,1.6);
\node at (8,2) {$(28,[0])$};
\draw[fill=red] (6.85,1.6)--(7.15,1.6)--(7.15,2.4)--(6.85,2.4)--(6.85,1.6);
\node at (7,2) {$(25,[1])$};
\draw[fill=cyan] (6.85,2.6)--(7.15,2.6)--(7.15,3.4)--(6.85,3.4)--(6.85,2.6);
\node at (7,3) {$(27,[0])$};
\draw[fill=cyan] (6.85,3.6)--(7.15,3.6)--(7.15,4.4)--(6.85,4.4)--(6.85,3.6);
\node at (7,4) {$(29,[1])$};
\draw[fill=cyan] (9.85,-1.4)--(10.15,-1.4)--(10.15,-0.6)--(9.85,-0.6)--(9.85,-1.4);
\node at (10,-1) {$(28,[1])$};
\draw[fill=red] (7.85,-0.4)--(8.15,-0.4)--(8.15,0.4)--(7.85,0.4)--(7.85,-0.4);
\node at (8,0) {$(24,[0])$};
\node at (8,2) {$(28,[0])$};
\draw[fill=cyan] (7.85,2.6)--(8.15,2.6)--(8.15,3.4)--(7.85,3.4)--(7.85,2.6);
\node at (8,3) {$(30,[1])$};
\draw[fill=cyan] (7.85,3.6)--(8.15,3.6)--(8.15,4.4)--(7.85,4.4)--(7.85,3.6);
\node at (8,4) {$(32,[0])$};
\draw[fill=cyan] (7.85,4.6)--(8.15,4.6)--(8.15,5.4)--(7.85,5.4)--(7.85,4.6);
\node at (8,5) {$(34,[1])$};
\draw[fill=red] (7.85,5.6)--(8.15,5.6)--(8.15,6.4)--(7.85,6.4)--(7.85,5.6);
\node at (8,6) {$(36,[0])$};
\draw[fill=red] (8.85,5.6)--(9.15,5.6)--(9.15,6.4)--(8.85,6.4)--(8.85,5.6);
\node at (9,6) {$(39,[1])$};
\draw[fill=red] (9.85,5.6)--(10.15,5.6)--(10.15,6.4)--(9.85,6.4)--(9.85,5.6);
\node at (10,6) {$(42,[0])$};
\draw[fill=red] (10.85,5.6)--(11.15,5.6)--(11.15,6.4)--(10.85,6.4)--(10.85,5.6);
\node at (11,6) {$(45,[1])$};
\draw[fill=red] (11.85,5.6)--(12.15,5.6)--(12.15,6.4)--(11.85,6.4)--(11.85,5.6);
\node at (12,6) {$(48,[0])$};
\draw[fill=cyan] (11.85,4.6)--(12.15,4.6)--(12.15,5.4)--(11.85,5.4)--(11.85,4.6);
\node at (12,5) {$(46,[1])$};
\draw[fill=cyan] (11.85,3.6)--(12.15,3.6)--(12.15,4.4)--(11.85,4.4)--(11.85,3.6);
\node at (12,4) {$(44,[0])$};
\draw[fill=cyan] (11.85,2.6)--(12.15,2.6)--(12.15,3.4)--(11.85,3.4)--(11.85,2.6);
\node at (12,3) {$(42,[1])$};
\draw[fill=cyan] (11.85,1.6)--(12.15,1.6)--(12.15,2.4)--(11.85,2.4)--(11.85,1.6);
\node at (12,2) {$(40,[0])$};
\draw[fill=cyan] (11.85,0.6)--(12.15,0.6)--(12.15,1.4)--(11.85,1.4)--(11.85,0.6);
\node at (12,1) {$(38,[1])$};
\draw[->] (11.85,0)--(11.15,0);
\draw[->] (10.85,0)--(10.15,0);
\draw[->] (9.85,0)--(9.15,0); 
\draw[->] (9,0.4)--(9,0.6); 
\draw[->] (8.85,1)--(8.15,1); 
\draw[->] (7.85,1)--(7.15,1);
\draw[->] (6.85,1)--(6.15,1);
\draw[->] (6,1.4)--(6,1.6);
\draw[->] (6,2.4)--(6,2.6);
\draw[->] (6,3.4)--(6,3.6);
\draw[->] (9.85,-2)--(9.15,-2); 
\draw[->] (9,-1.6)--(9,-1.4);
 \draw[->] (8,1.4)--(8,1.6); 
\draw[->] (7.85,2)--(7.15,2);
\draw[->] (7,2.4)--(7,2.6);
\draw[->] (7,3.4)--(7,3.6);
\draw[->] (10,-1.6)--(10,-1.4);
\draw[->] (8.85,0)--(8.15,0);
\draw[->] (8,2.4)--(8,2.6);
\draw[->] (8,3.4)--(8,3.6);
\draw[->] (8,4.4)--(8,4.6);
\node at (12.62,6) {$\vdots$};
\draw[->] (12.5,6)--(12.15,6);
\draw[->] (11.85,6)--(11.15,6);
\draw[->] (10.85,6)--(10.15,6);
\draw[->] (9.85,6)--(9.15,6);
\draw[->] (8.85,6)--(8.15,6);
\draw[->] (12,0.4)--(12,0.6);
\draw[->] (12,1.4)--(12,1.6);
\draw[->] (12,2.4)--(12,2.6);
\draw[->] (12,3.4)--(12,3.6);
\draw[->] (12,4.4)--(12,4.6);

\draw (12,6.4)--(12,6.5);
\node at (11.98,6.63) {$\cdots$};
\end{scope}\end{tikzpicture}
\end{center}

\caption{The first arrival tree in $M_{3,2,2}(12,12,10,8,7,4,1,1,1)$ with the eastern vertices coloured blue, the southern vertices coloured red, and the switches coloured green.}\label{Big tree 1}
\end{figure}
The paths in the first arrival tree from $(36,[0])$ to $(26,[0]),$ $(23,[1]),$ $(28,[1])$ and $(25,[0])$ have winding number $1$, whilst the other vertices $(v,[i])$ for which there is a path in the first arrival tree from $(36,[0])$ to $(v,[i])$ have winding number $0$. The switches are coloured green, the southern vertices red, and the eastern vertices blue (compare with Figure~\ref{Big multigraph 1} to verify the colouring). 
\end{ex}

\subsection{Definition of \texorpdfstring{$I_{r,s,c}$}{Irsc}}
Given a partition $\lambda$ with $\lambda\leq_{r,s,c} \lambda_{r,s,k}$, with multigraph $M$ and first arrival tree $T$, we define the partition $I_{r,s,c}(\lambda)$ as follows. The multigraph of $I_{r,s,c}(\lambda)$ is also given by $M$.

Now, we obtain the $(r,s,c)$-tour of $I_{r,s,c}(\lambda)$ by, at each switch, reversing the arrival word, and at each vertex that is not a switch, fixing the first letter of the arrival word and reversing the rest of the arrival word. 

To see that $I_{r,s,c}(\lambda)$ is well defined, we need to check that taking the first arrival at each vertex of $M$ gives a spanning tree. We do this by checking that in $T$, the move of deleting an east (respectively south) edge arriving at a switch $(v,[i])$ and adding a new south (respectively east) edge arriving at $(v,[i])$ gives another spanning tree $T'$. There are still edges arriving at every vertex we had edges arriving at before, but now the edge arriving at $(v,[i])$ might be departing from a different vertex. So, it suffices to check that $(v,[i])$ is still connected to each of $(v-s,[i-1])$ and $(v+r,[i-1])$, and that we have not introduced a cycle by adding the new edge. For the former, it suffices to check $(v-s,[i-1])$ and $(v+r,[i-1])$ are still connected to each other. In $T$, $(v+r,[i-1])$ and $(v-s,[i-1])$ are both connected to $(k,[0])$ by paths. Moreover, the distance in $T$ to $(k,[0])$ strictly decreases with each step along the path we take, so $(v,[i])$ is not a vertex on either of these paths. So, both of these paths exist $T'$, and $(v+r,[i-1])$ and $(v-s,[i-1])$ are connected to one another. To see that the new edge does not introduce a cycle, observe that if we had introduced a cycle, we would now have two distinct paths from $(k,[0])$ to $(v,[i])$. Since the only edge into $(v,[i])$ is from its new neighbour, we must have had two distinct paths from $(k,[0])$ to the new neighbour in $T$ originally. But then we had a cycle in $T$ originally, a contradiction. 

Hence, we may permute the letters in any arrival word at any vertex and the result will still correspond to a partition as long as we do not change the first letter in the arrival word at a vertex that is not a switch. Since we defined $I_{r,s,c}(\lambda)$ to fix the first letter in the arrival word at any vertex that is not a switch, $I_{r,s,c}(\lambda)$ is well defined. Moreover, we can recover $\lambda$ from $I_{r,s,c}(\lambda)$ by doing the same operation again, as each operation is self-inverse and preserves switches. 

Since $I_{r,s,c}$ does not change $M_{r,s,c}$, we can apply Proposition~\ref{independence core} and Corollary ~\ref{areacorol} respectively to obtain
\begin{equation}
    |I_{r,s,c}(\lambda)|=|\lambda|,
\end{equation}
and 
\begin{equation}
    \core_c(I_{r,s,c}(\lambda))=\core_c(\lambda).
\end{equation}

Moreover, the map sending $\lambda$ to $I_{r,s,c}(\lambda)$ is an involution - it is immediate from the definition that a vertex is a switch after this reassignment if and only if it were a switch before the reassignment. 

\section{Further statistics determined by \texorpdfstring{$M_{r,s,c}$}{the multigraph}}
\stepcounter{essaypart}
This section checks that $I_{r,s,c}$ satisfies hypotheses 2--4 in Proposition~\ref{bijection properties}.

First, we use the method introduced in Section 4 to prove that $M_{r,s,c}$ determines $\middxc$ and $\critplusxc+\critminusxc.$ Then we check that $I_{r,s,c}$ exchanges $\critplusxc$ and $\critminusxc$, concluding the proof of Theorem~\ref{main theorem}.  In the language of Proposition~\ref{acc points are useful}, Propositions~\ref{middifference} and ~\ref{crittotdiff} calculate $g(M^+,M)$ for $f(\lambda)=\middxc(\lambda)$ and $f(\lambda)=\critplusxc(\lambda)+\critminusxc(\lambda)$ respectively.

\subsection{\texorpdfstring{$M_{r,s,c}$}{The multigraph} determines \texorpdfstring{$\middxc$}{mid}}

\begin{notn}\label{arrivalwordsuccessors}
Let $\lambda$, $M$, $M^+$, $(x,y)$, $(l,[i])$, $s_1$, $s_2$ $s_1'$ and $s_2'$ be as in the proof of Proposition~\ref{successors pass}. For any edge $e$ in the boundary of $\lambda$, write 
$\Ein^{\to e}(w,[j])$ for the number of $E$s in the arrival word at a vertex $(w,[j])$ that correspond to east edges in the boundary of $\lambda$ that occur before $e$. Define $\Sin^{\to e}(w,[j])$ analogously for the number of $S$s. Write $\Ein^{e\to}(w,[j])$ for the number of $E$s in the arrival word at a vertex $(w,[j])$ that correspond to east edges in the boundary of $\lambda$ that occur after $e$. Define $\Sin^{e\to}(w,[j])$ analogously for the number of $S$s. Analogously define $\Sout^{\to e}(w,[j])$, $\Eout^{\to e}(w,[j])$, $\Sout^{\to e}(w,[j])$, and $\Eout^{\to e}(w,[j])$ for the departure words. We will use this notation with $e=s_1$ or $e=s_2$. Finally, write $\Einplus(w,[j])$ for the number of $E$s in the arrival word at $(w,[j])$ in $M^+$ and define analogously $\Sinplus, \Eoutplus$ and $\Soutplus$.

We work in the ring $R$ of functions $V(M)\to \mathbb{Z}.$ Practically, the only consequence of this is that we write $fg(v,[i])$ for the pointwise product $f(v,[i])g(v,[i])$ and $(f+g)(v,[i])$ for $f(v,[i])+g(v,[i]).$ There should be no confusion between composition and product of functions as functions from $V(M)$ to $\mathbb{Z}$ are not composable. 
\end{notn}

\begin{ex} Let $\lambda=(4,1),$ $r=3$, $s=1$ and $c=2.$ Then $\min_{(x,y)\in b(\lambda)}(3y+x)=4$ achieved at $(4,0)$ and $(1,1).$ Since $4-0\equiv 1-1 \pmod{2}$ and $1-1<4-0$, we have $(1,1)<_{3,1,2}(4,0)$ so $(4,2)$ is the only $(3,1,2)$-successor of $\lambda$. So, $(l,[i])=(4,[0]).$
\begin{figure}
\begin{center}
    \begin{tikzpicture}\begin{scope}[yscale=1,xscale=-1,rotate=90,scale=0.5]
    \draw (0,5)--(0,0);
    \draw (1,4)--(1,0);
    \draw (2,1)-- (2,0);
    \draw (0,0)--(4,0);
    \draw (0,1)--(2,1);
    \draw (0,2)--(1,2);
    \draw (0,3)--(1,3);
    \draw (0,4)--(1,4);
    \draw[dashed] (2,1)--(2,2)--(1,2);
\draw[dashed,->] (2,1)--(2,1.55);
\draw[->] (2,1)--(1.45,1);
\node[left] at (1.5,1) {$s_1$};
\node[below] at (1,1.5) {$s_2$};
\node[above] at (2,1.5) {$s_1'$};
\node[right] at (1.5,2) {$s_2'$};
\draw[->] (1,1)--(1,1.55);
\draw[dashed,->] (2,2)--(1.45,2);
    \end{scope}\end{tikzpicture}
\end{center}
\caption{the partition $\lambda$ with $(l,[i])=(4,[0])$ and $s_1,s_2,s_1'$ and $s_2'$ labelled}
\end{figure}
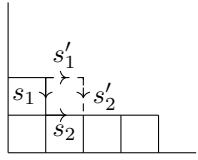
As an example of the use of Notation ~\ref{arrivalwordsuccessors}, $(\Ein^{\to s_1}\Sout^{s_1 \to}+\Eoutplus)(7,[1])=1\times1 +1=2.$
\end{ex}

\begin{prop}\label{middifference} Let $\lambda$ be a partition with $M_{r,s,c}(\lambda)=M$. Let $M^+$ be a successor of $M$ that changes from $(l,[i])$. If $\lambda^+>'_{r,s,c}\lambda$ and $M^+=M_{r,s,c}(\lambda^+),$
\begin{equation}\label{middifferenceformula}\middxc(\lambda^+)-\middxc(\lambda)=\sum_{w=l+1}^{l+s+r-1}\left(\Eout-\Ein\right)(w,[i]),\end{equation}
where $x=\frac{r}{s}$.
\end{prop}
\begin{proof}
By Proposition~\ref{successors pass}, the Young diagram of $\lambda^+$ is obtained from that of $\lambda$ by adding a box with bottom corner $(x_1,y_1)$ where $ry_1+sx_1=l$ and $[x_1-y_1]=[i]$.

Proposition~\ref{motivate Mrsc} gives a formula for $\middxc(\lambda)$: it is the number of pairs of edges $e_1,e_2$ in the boundary sequence such that $e_1$ is an east edge arriving at a point $(x,y)$ satisfying $sx+ry=v$ and $[x-y]=[j]$ for some $j$, and $e_2$ is a south edge occurring after $e_1$ arriving at a point $(x',y')$ satisfying $ry'+sx'=w$ and $[x'-y']=[j]$, where $w$ and $v$ satisfy $-s-r<w-v<0$.
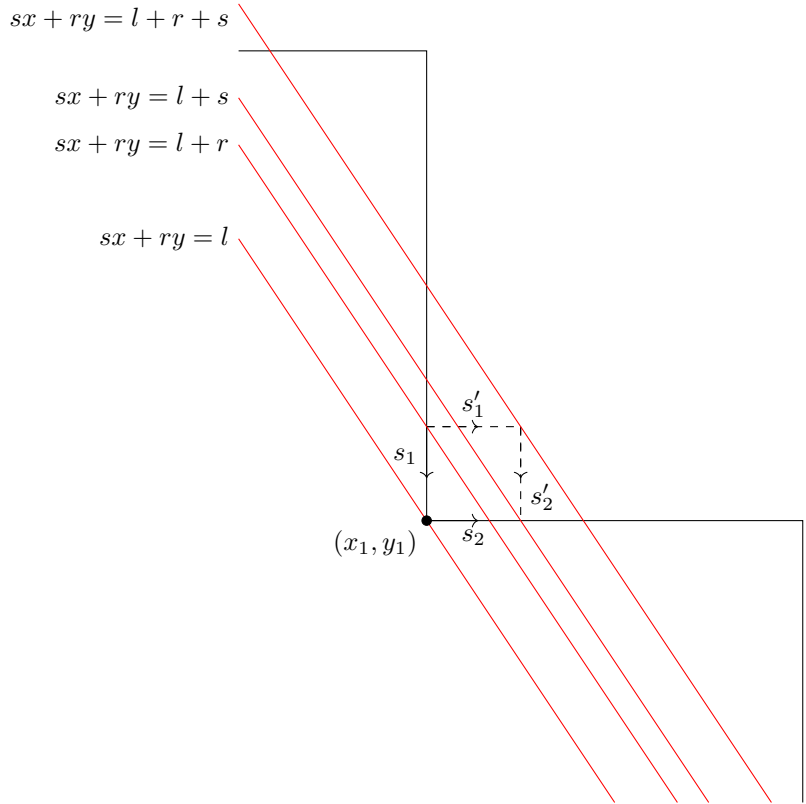
\begin{figure}[ht]
\begin{center}
\begin{tikzpicture}\begin{scope}[yscale=1,xscale=-1,rotate=90,scale=1.25]
\draw (1,0)--(0,0)--(0,1);
\draw[dashed] (1,0)--(1,1)--(0,1);
\draw[dashed,->] (1,0)--(1,0.55);
\draw[->] (1,0)--(0.45,0);
\node[above left] at (0.5,0) {$s_1$};
\node[below] at (0,0.5) {$s_2$};
\node[above] at (1,0.5) {$s_1'$};
\node[right] at (0.25,1) {$s_2'$};
\draw[->] (0,0)--(0,0.55);
\draw[dashed,->] (1,1)--(0.45,1);
\draw (0,1)--(0,4)--(-3,4);
\draw (1,0)--(5,0)--(5,-2);
\draw[red] (-3,2)--(3,-2);
\draw[red] (-3,2.667)--(4,-2);
\draw[red] (-3,3)--(4.5,-2);
\draw[red] (-3,3.667)--(5.5,-2); 
\node[left] at (3,-2) {$sx+ry=l$};
\node[left] at (4,-2) {$sx+ry=l+r$};
\node[left] at (4.5,-2) {$sx+ry=l+s$};
\node[left] at (5.333,-2) {$sx+ry=l+r+s$};
\draw[fill=black] (0,0) circle[radius=0.05];
\node[below left] at (0,0) {$(x_1,y_1)$};
\end{scope}\end{tikzpicture}
\end{center}
\caption{the diagrams of $\lambda$ and $\lambda^+$.}
\label{one figure to rule them all}
\end{figure}

We account for the change in the number of such pairs when changing $s_1$ to $s_1'$ and $s_2$ to $s_2'$ below: the only changes to $\middxc$ will be when $e_1\in \{s_2,s_1'\}$ or $e_2\in \{s_1,s_2'\}.$

By adding $s_1'$ we gain the number of south edges after $s_1$, arriving at points $(x,y)$ on lines $sx+ry=w$, such that $-s-r<w-(l+r+s)<0$ and  $[x-y]=[i]$. By deleting $s_1$ we lose the number of east edges occurring before $s_1$ arriving at points $(x,y)$ on lines $sx+ry=v$ such that $-s-r<l-v<0$ and $[x-y]=[i]$ So, the contribution to $\middxc(\lambda^+)-\middxc(\lambda)$ from switching $s_1$ to $s_1'$ is $S_1$ where

\begin{equation}\label{S1} S_1=\sum_{w=l+1}^{l+r+s-1}\left(\Sin^{s_1\to}-\Ein^{\to s_1}\right)(w,[i]).\end{equation}

By adding $s_2'$ we gain the number of east edges before $s_2$, arriving at points $(x,y)$ on lines $sx+ry=v$, such that $-s-r<l+s-v<0$, and $x-y\equiv i+1\pmod{c}$. By deleting $s_2$ we lose the number of south edges occurring after $s_2$ arriving at points $(x,y)$ on lines $sx+ry=w$ such that $-s-r<w-(l+s)<0$ and $[x-y]=[i+1].$
So, the contribution to $\middxc(\lambda^+)-\middxc(\lambda)$ from switching $s_2$ to $s_2'$ is $S_2$ where \begin{equation}\label{S2}S_2=\sum_{v=l+s+1}^{l+r+2s-1}\Ein^{\to s_2}(v,[i+1]) -\sum_{v=l-r+1}^{l+s-1}\Sin^{s_2\to}(v,[i+1]).\end{equation}

So, \begin{equation}\label{midbreakdown}\middxc(\lambda^+)-\middxc(\lambda)=S_1+S_2.\end{equation} Now, note that an east edge into $(v,[i+1])$ is also an east edge out of $(v-s,[i])$, and a south edge into $(w,[i+1])$ is also a south edge out of $(w+r,[i])$. Applying this reasoning to \eqref{S2}, 

\begin{equation}\label{S2v2}S_2=\sum_{w=l+1}^{l+r+s-1}\left(\Eout^{\to s_2} -\Sout^{s_2\to}\right)(w,[i]).\end{equation} Substituting~\eqref{S1} and~\eqref{S2v2} into~\eqref{midbreakdown}, 
\begin{equation}\label{diffmid}\middxc(\lambda^+)-\middxc(\lambda)=\sum_{w=l+1}^{l+s+r-1}\left(\Sin^{s_1\to}-\Ein^{\to s_1}+\Eout^{\to s_2}-\Sout^{s_2\to}\right)(w,[i]).\end{equation}
Since $s_2$ is an east edge occuring immediately after $s_1$, $\Sout^{s_2\to}=\Sout^{s_1\to}$, and since $s_1$ is a south edge immediately preceding $s_2$, $\Eout^{\to s_2}=\Eout^{\to s_1}.$ So,
\begin{equation}\label{diffmid2}\middxc(\lambda^+)-\middxc(\lambda)=\sum_{w=l+1}^{l+s+r-1}\left(\Sin^{s_1\to}-\Ein^{\to s_1}+\Eout^{\to s_1}-\Sout^{s_1\to}\right)(w,[i]).\end{equation}

Now, note that at any vertex $(v,[j])$ except $(l,[i])$, we have that \begin{equation}\left(\Sin^{s_1\to}+\Ein^{s_1\to}\right)(v,[j])=\left(\Sout^{s_1\to}+\Eout^{s_1\to}\right)(v,[j]),\end{equation} because after $s_1$ we depart every vertex after we arrive at it, the left hand side counting arrivals at the vertex after $s_1$ and the right side counting departures.
Rearranging gives 
\begin{equation}\label{afters1}\left(\Sin^{s_1\to}-\Sout^{s_1\to}\right)(v,[j])=\left(\Eout^{s_1\to}-\Ein^{s_1\to}\right)(v,[j]).\end{equation} 

Substituting \eqref{afters1} into \eqref{diffmid2}, 
\begin{equation}\middxc(\lambda^+)-\middxc(\lambda)=\sum_{w=l+1}^{l+s+r-1}\left(\Eout^{s_1\to}-\Ein^{s_1\to}-\Ein^{\to s_1}+\Eout^{\to s_1}\right)(w,[i]).\end{equation}
Since $s_1$ is a south edge, $\Ein^{s_1\to}+\Ein^{\to s_1}=\Ein$ and $\Eout^{s_1\to}+\Eout^{\to s_1}=\Eout,$ so
\begin{equation}
\middxc(\lambda^+)-\middxc(\lambda)=\sum_{w=l+1}^{l+s+r-1}\left(\Eout-\Ein\right)(w,[i]).
\end{equation}\end{proof}

\begin{corol}\label{midcorol}
If $\lambda$ and $\mu$ are partitions with $M_{r,s,c}(\lambda)=M_{r,s,c}(\mu),$ then $\middxc(\lambda)=\middxc(\mu).$
\end{corol}
\begin{proof}
 Apply Proposition~\ref{acc points are useful} with \begin{equation}g(M^+,M)=\sum_{w=l+1}^{l+s+r-1}\left(\Eout-\Ein\right)(w,[i]), \end{equation} where $M^+$ is the successor of $M$ that changes from $(l,[i]).$
\end{proof}

\subsection{\texorpdfstring{$M_{r,s,c}$}{The Multigraph} determines \texorpdfstring{$\critplusxc+\critminusxc$}{the sum crit+ + crit-}}

\begin{prop}\label{crittotdiff} Let $\lambda$ be a partition with $M_{r,s,c}(\lambda)=M$ and let $M^+$ be a successor of $M$ that changes from $(l,[i])$. Then, if $\lambda^+$ is the successor of $\lambda$ with multigraph $M^+,$ \begin{equation}
    \critplusxc(\lambda^+)+\critminusxc(\lambda^+)-\critplusxc(\lambda)-\critminusxc(\lambda)
\end{equation}
is equal to \begin{equation}\Sin(l,[i])-1+(\Sin-\Sout)(l+r+s,[i]),\end{equation}
where $x=\frac{r}{s}.$
\end{prop}
\begin{proof} 

First, we compute $\critplusxc(\lambda^+)-\critplusxc(\lambda)$. Corollary~\ref{firstcritformulae} implies that 
\begin{equation}\label{differenceplus}
    \critplusxc(\lambda^+)-\critplusxc(\lambda)=\sum_{v\in M_{r,s,c}(\lambda^+)}\inv(v_a)-\sum_{v\in M_{r,s,c}(\lambda)}\inv(v_a)
\end{equation}

We keep the notation of the previous proposition and reference Figure~\ref{one figure to rule them all} throughout.
The only nonzero terms in the difference ~\eqref{differenceplus}
come from $v\in\{(l,[i]), (l+r+s,[i]), (l+s,[i+1])\}$. We work case-by-case through these vertices.
\begin{itemize}
\item We delete the first arrival at $(l,[i])$, corresponding to deleting $s_1.$ All arrivals at $(l,[i])$ are $S$s by Corollary~\ref{noWoutNin(l,[i])}, so this does not affect $\inv(l,[i])_a.$
\item We add an $E$ to the arrival word at $(l+r+s,[i])$, corresponding to adding $s_1'.$ 
\begin{center}
    \begin{tikzpicture}[scale=1.1]
    \draw (0,0.5)--(0,0)--(3.95,0)--(3.95,0.5);
    \draw (4.05,0.5)--(4.05,0)--(7,0)--(7,0.5);
    \node[below] at (1.975,0) {before $s_1$};
    \node[below] at (5.502,0) {after $s_1$};
    \draw (-0.25,-0.5)--(-0.25,-1)--(3.7,-1)--(3.7,-0.5);
    \draw (3.8,-0.5)--(3.8,-1)--(4.2,-1)--(4.2,-0.5);
    \node at (4,-0.75) {$E$};
    \draw (4.3,-0.5)--(4.3,-1)--(7.25,-1)--(7.25,-0.5);
    \node[below] at (1.725,-1) {before $s_1'$};
    \node[below] at (4,-1) {$s_1'$};
    \node[below] at (5.752,-1) {after $s_1'$};
\end{tikzpicture}
\end{center}

This $E$ is the first letter in an inversion with second letter any $S$ occuring after $s_1'$, so $(l+r+s,[i])$ contributes $\Sin^{s_1\to}(l+r+s,[i])$ to  ~\eqref{differenceplus}.
\item We replace the first $E$ in the arrival word at $(l+s,[i+1])$ (corresponding to $s_2$) with an $S$ (corresponding to $s_2'$). Therefore, we lose all inversions with the replaced $E$ edge as their first letter passing from $\lambda$ to $\lambda^+$. There are $\Sin^{s_2\to}(l+s,[i+1])$ such inversions.
\begin{center}
    \begin{tikzpicture}[scale=1.1]
    \draw (-0.25,0.5)--(-0.25,0)--(3.7,0)--(3.7,0.5);
    \draw (3.8,0.5)--(3.8,0)--(4.2,0)--(4.2,0.5);
    \node at (4,0.25) {$E$};
    \draw (4.3,0.5)--(4.3,0)--(7.25,0)--(7.25,0.5);
    \node at (1.8,0.25) {$SSSSSSSS\ldots SSSS$};
    \node[below] at (1.725,0) {before $s_2$};
    \node[below] at (4,0) {$s_2$};
    \node[below] at (5.752,0) {after $s_2$};
    \draw (-0.25,-0.5)--(-0.25,-1)--(3.7,-1)--(3.7,-0.5);
    \draw (3.8,-0.5)--(3.8,-1)--(4.2,-1)--(4.2,-0.5);
    \node at (4,-0.75) {$S$};
    \node at (1.8,-0.75) {$SSSSSSSS\ldots SSSS$};
    \draw (4.3,-0.5)--(4.3,-1)--(7.25,-1)--(7.25,-0.5);
    \node[below] at (1.725,-1) {before $s_2'$};
    \node[below] at (4,-1) {$s_2'$};
    \node[below] at (5.752,-1) {after $s_2'$};
\end{tikzpicture}
\end{center}

We gain no inversions from the new $S$ edge, because $s_2$ was the first east departure from $(l,[i])$ in the tour corresponding to $\lambda$. So, $(l+s,[i+1])$ contributes $-\Sin^{s_2\to}(l+s,[i+1])$ to ~\eqref{differenceplus}.
\end{itemize}

So, \begin{equation}\label{star2}\critplusxc(\lambda^+)-\critplusxc(\lambda)=\Sin^{s_1\to}(l+r+s,[i])-\Sin^{s_2\to}(l+s,[i+1]).\end{equation}

A south arrival before (respectively after) $s_2$ at $(l+s,[i+1])$ is a south departure before (respectively after) $s_2$ from $(l+r+s,[i]).$ Combining this logic with \eqref{star2},
 \begin{equation}\label{star3} \critplusxc(\lambda^+)-\critplusxc(\lambda) =
\left(\Sin^{s_1\to}-\Sout^{s_2\to}\right)(l+r+s,[i]).\end{equation}

We now analyse  \begin{equation}\label{differenceminus}\critminusxc(\lambda^+)-\critminusxc(\lambda)=\sum_{v\in M_{r,s,c}(\lambda^+)}\inv(v_d)-\sum_{v\in M_{r,s,c}(\lambda)}\inv(v_d).\end{equation} The departure words at every vertex except for $(l,[i])$, $(l+r+s,[i])$, and $(l+r,[i-1])$ are unchanged so the only nonzero terms in~\eqref{differenceminus} come from $v\in\{(l,[i]),(l+r+s,[i]),(l+r,[i-1])\}.$ An analogous argument to the above shows that the contribution of $(l+r,[i-1])$ to $~\eqref{differenceminus}$ is $\left(\Sout^{s_1\to}-\Eout^{\to s_1}\right)(l+r,[i-1])$, the contribution of $(l+r+s,[i])$ is $\Eout^{\to s_2}(l+r+s,[i])$, and $(l,[i])$ does not contribute. So,

\begin{equation}\label{dagger2}\critminusxc(\lambda^+)-\critminusxc(\lambda)=\left(\Sout^{s_1\to}-\Eout^{\to s_1}\right)(l+r,[i-1])+\Eout^{\to s_2}(l+r+s,[i]).\end{equation}
An east departure from $(l+r,[i-1])$ is an east arrival at $(l+r+s,[i]),$ so \begin{equation}\label{dagger3} \critminusxc(\lambda^+)-\critminusxc(\lambda)=
\Sin^{s_1\to}(l,[i])-\left(\Ein^{\to s_1}-\Eout^{\to s_2}\right)(l+r+s,[i]).\\
\end{equation}
Now, since $s_1$ is the first edge to arrive at $(l,[i]),$ \begin{equation}\label{s1firstwest} \Sin^{s_1\to}(l,[i])=\Sin(l,[i])-1.\end{equation}
Since $s_1$ does not arrive at $(l+r+s,[i])$, we leave $(l+r+s,[i])$ before $s_1$ the same number of times as we arrive before $s_1$. So,
\begin{equation}\label{arrivalsbeforearedeparturesafter}(\Ein^{\to s_1}+\Sin^{\to s_1})(l+r+s,[i])=(\Eout^{\to s_2}+\Sout^{\to s_2})(l+r+s,[i])).\end{equation}
Rearranging, \begin{equation}\label{reexpressNin}\Ein^{\to s_1}(l+r+s,[i])=(\Eout^{\to s_2}+\Sout^{\to s_2}-\Sin^{\to s_1})(l+r+s,[i]).\end{equation}
Substituting \eqref{reexpressNin} and \eqref{s1firstwest} into \eqref{dagger3},
\begin{equation}\label{dagger4}\critminusxc(\lambda^+)-\critminusxc(\lambda)=\Sin(l,[i])-1+(\Sin^{\to s_1}-\Sout^{\to s_2})(l+r+s,[i]).\end{equation}
Since $(l+r+s,[i])$ is not an endpoint of $s_1$ or $s_2$, \begin{equation}\label{westin} (\Sin^{\to s_1}+\Sin^{ s_1\to})(l+r+s,[i])=\Sin(l+r+s,[i])\end{equation} and \begin{equation}\label{westout}(\Sout^{\to s_2}+\Sout^{ s_2\to})(l+r+s,[i])=\Sout(l+r+s,[i]).\end{equation} Adding  \eqref{dagger4} and \eqref{star3}, and then applying \eqref{westout} and \eqref{westin} completes the proof.
\end{proof}
\begin{corol}\label{crittotcorol}
If $\lambda$ and $\mu$ are partitions such that $M_{r,s,c}(\mu)=M_{r,s,c}(\lambda)$ then \begin{equation}\critplusxc(\lambda)+\critminusxc(\lambda)=\critplusxc(\mu)+\critminusxc(\mu).\end{equation}
\end{corol}
\begin{proof}
Apply Proposition~\ref{acc points are useful} with \begin{equation}g(M^+,M)=\Sin(l,[i])-1+(\Sin-\Sout)(l+r+s,[i]).\end{equation} where the calculations $\Sin$ and $\Sout$ are done with respect to the multigraph $M$, and $M^+$ is the successor of $M$ changing from $(l,[i]).$
\end{proof}

So, we know that $M_{r,s,c}(\lambda)$ determines the $c$-core of $\lambda$, $|\lambda|,$ $\middxc(\lambda)$ and $\critplusxc(\lambda)+\critminusxc(\lambda),$ and that any bijection preserving $M_{r,s,c}$ therefore satisfies hypotheses 1-3 of Proposition~\ref{bijection properties}. It will be useful in our final remaining check, that $I_{r,s,c}$ satisfies the fourth criterion in Proposition~\ref{bijection properties}, to have a formula for $\critplusxc(\lambda)+\critminusxc(\lambda)$ in terms of $M_{r,s,c}(\lambda).$ This is what Proposition~\ref{ctot formula} computes.

\begin{prop}\label{ctot formula} Let $\lambda$ be a partition. If $k=rsk_1$ where $c\mid k_1$ and $\lambda <_{r,s,c} \lambda_{r,s,k},$ then

\begin{equation}\label{crittot}(\critplusxc+\critminusxc)(\lambda)= \sum_{\substack{(v,[j]) \\ v\leq k}} \Ein\Sin(v,[j])-\left\lfloor\frac{k_1(s+r)}{\lcm(c,s+r)}\right\rfloor.\end{equation}
\end{prop}

\begin{proof} First, we prove that~\eqref{crittot} holds when  $\lambda=\lambda_{r,s,k}$. 

We will show that for all boxes $\square\in\lambda_{r,s,k}$, $-s<sa(\square)-rl(\square)<r$, and hence that the left hand side of ~\eqref{crittot} is zero at $\lambda_{r,s,k}$. We will then check that the right hand side of \eqref{crittot} is zero at $\lambda_{r,s,k}.$

The $i$th part of $\lambda_{r,s,k}$ corresponds to a row with top right corner $(x_i,i)$ where $x_i$ is maximal such that $sx_i+ri\leq k$. So, \begin{equation}x_i=\left\lfloor\frac{k-ri}{s}\right\rfloor=\left\lfloor\frac{k_1rs-ri}{s}\right\rfloor=k_1r-\left\lceil\frac{ri}{s}\right\rceil.\end{equation} Similarly, the number of parts of $\lambda_{r,s,k}$ of size at least $j$ corresponds to a column with top right corner $(j,y_j)$ where $y_j$ is maximal such that $sj+ry_j\leq k$, so \begin{equation}y_j=\left\lfloor\frac{k-sj}{r}\right\rfloor=\left\lfloor\frac{k_1rs-sj}{r}\right\rfloor=k_1s-\left\lceil\frac{sj}{r}\right\rceil.\end{equation}

Now, let $\square\in\lambda$ be a box with top right corner $(i,j)$. Then, the arm of $\square$ is given by $x_i-j$ and the leg of $\square$ is given by $y_j-i$. So, \begin{align}sa(\square)-rl(\square)&=s(x_i-j)-r(y_j-i)\\
&=k_1rs-s\left\lceil\frac{ri}{s}\right\rceil-sj-k_1rs+r\left\lceil\frac{sj}{r}\right\rceil+ri\\
&=\left(r\left\lceil\frac{sj}{r}\right\rceil-sj\right)-\left(s\left\lceil\frac{ri}{s}\right\rceil+ri\right). 
\end{align}

Now, consider the two bracketed quantities separately, setting $x=\left(r\left\lceil\frac{sj}{r}\right\rceil-sj\right)$ and $y=-\left(s\left\lceil\frac{ri}{s}\right\rceil+ri\right)$. For the first bracket we have that \begin{equation}r\left(\frac{sj}{r}\right)\leq r\left\lceil \frac{sj}{r}\right\rceil< r\left(\frac{sj}{r}+1\right),\end{equation} so \begin{equation}0\leq r\left\lceil\frac{sj}{r}\right\rceil-sj< r.\end{equation}

Similarly for the second bracket, \begin{equation}-s< ri-s\left\lceil\frac{ri}{s}\right\rceil\leq 0.\end{equation}

So, $sa(\square)-rl(\square)$ can be written as $x+y$ for $x\in[0,r)$ and $y\in(-s,0]$ and therefore $-s<sa(\square)-rl(\square)<r.$

Therefore, \begin{equation}\critplusxc(\lambda_{r,s,k})+\critminusxc(\lambda_{r,s,k})=0.\end{equation}

Next we evaluate the right hand side of \eqref{crittot} at $\lambda_{r,s,k}$. Proposition \ref{Multigraph rsk} tells us that for all vertices $(v,[i])$ such that $0\leq v<k$, the arrival word at $(v,[i])$ in $M_{r,s,c}(\lambda_{r,s,k})$ does not contain both an $E$ and a $S$.  So, for all such $(v,[i])$ we have $\Ein\Sin(v,[i])=0$. So, the right hand side of \eqref{crittot} simplifies to
\begin{equation}\label{crittotrsk1}\sum_{i=0}^{c-1}\Ein\Sin(k,[i])-\left\lfloor\frac{k_1(s+r)}{\lcm(c,s+r)}\right\rfloor\end{equation}

Proposition~\ref{Multigraph rsk} also tells us that $\Sin(k,[i])=0$ unless $[i]=[0]$, and that $\Sin(k,[0])=1$, so we can rewrite \eqref{crittotrsk1} as

\begin{equation}\label{crittotrsk2}\Ein(k,[0])-\left\lfloor\frac{k_1(s+r)}{\lcm(c,s+r)}\right\rfloor.\end{equation}

So, it suffices to show that $\Ein(k,[0])=\left\lfloor\frac{k_1(s+r)}{\lcm(c,s+r)}\right\rfloor.$ The east edges in the boundary of $\lambda_{r,s,k}$ arriving at vertices $(k,[i])$ for some $i$ correspond to points $(x,y)$ with $x>0$ and $y\geq 0$ such that $sx+ry=k$. These points have coordinates $\{(r,s(k_1-1)), (2r,s(k_1-2)),\ldots, ((k_1-1)r,s), (k_1r,0)\}$. Now, $\Ein(k,[0])$ counts the number of these points that also lie on a line $x-y=i$ for $[i]=[0]$. The set of  values of $x-y$ for this set of points is $\{r+s-k_1s, 2(r+s)-k_1s,\ldots, k_1(r+s)-k_1s\}$. Letting $l(r+s)=\lcm(c,r+s)$, the values of $x-y$ that give us the same congruence class as $0$ when taken modulo $c$ are of the form $ml(r+s)-k_1s$ for some integer $m$. The number of values of this form in the given set is indeed $\left\lfloor\frac{k_1(s+r)}{\lcm(c,s+r)}\right\rfloor.$

Now suppose $\lambda<_{r,s,c}\lambda_{r,s,k}$ is maximal with respect to $>_{r,s,c}$ such that the proposition is false. In particular, the proposition holds for any successor $\lambda^+>'_{r,s,c}\lambda$. Let $M^+$ be a successor of $M$ that changes from $(l,[i]),$ and let $\lambda^+$ be the successor of $\lambda$ with multigraph $M^+$. Then, $(\critplusxc+\critminusxc)(\lambda^+)-(\critplusxc+\critminusxc)(\lambda)$ can be written as $\Delta_1$, where
\begin{equation}\label{Delta}\Delta_1=\Sin(l,[i])-1+(\Sin-\Sout)(l+r+s,[i]).\end{equation}

By assumption, \begin{equation}\label{crittotlambdaplus}(\critplusxc+\critminusxc)(\lambda^+)= \sum_{\substack{(v,[j]) \\ v\leq k}}\Einplus\Sinplus(v,[j])-\left\lfloor\frac{k_1(s+r)}{\lcm(c,s+r)}\right\rfloor\end{equation}
So, combining \eqref{Delta} and \eqref{crittotlambdaplus}, \begin{equation}(\critplusxc+\critminusxc)(\lambda)= \sum_{\substack{(v,[j]) \\ v\leq k}}\Einplus\Sinplus(v,[j])-\left\lfloor\frac{k_1(s+r)}{\lcm(c,s+r)}\right\rfloor-\Delta_1.\end{equation}

First, we note that a vertex $(v,[j])$ contributes the same to the sums $$\sum_{\substack{(v,[j]) \\ v\leq k}}\Ein\Sin(v,[j])$$ taken over the multigraphs $M$ or $M^+$ unless $(v,[j])\in\{(l,[i]),(l+r+s,[i],(l+s,[i+1])\}$. In fact, since there are no east edges into $(l,[i])$ in $M$ or $M^+$, we only need consider terms with $(v,[j])\in\{(l+r+s,[i],(l+s,[i+1])\}$. 
So, \begin{align*}(\critplusxc+\critminusxc)(\lambda)&= \sum_{\substack{(v,[j]) \\ v\leq k}}\Ein\Sin(v,[j])-\left\lfloor\frac{k_1(s+r)}{\lcm(c,s+r)}\right\rfloor -\Delta_1+\Delta_2 \stepcounter{equation}\tag{\theequation}\label{differencefromM}\end{align*}
where

\begin{equation}\Delta_2=(\Einplus\Sinplus-\Ein\Sin)(l+r+s,[i])+(\Einplus\Sinplus-\Ein\Sin)(l+s,[i+1]).\stepcounter{equation}\tag{\theequation}\label{Delta2}\end{equation}

Because $M^+$ changes from $M$ at $(l,[i])$, $\Einplus(l+s,[i+1])=\Ein(l+s,[i+1])-1$, $\Sinplus(l+s,[i+1])=\Sin(l+s,[i+1])+1$, $\Einplus(l+r+s,[i])=\Ein(l+r+s,[i])+1$ and $\Sinplus(l+r+s,[i])=\Sin(l+r+s,[i])$ so  \eqref{Delta2} simplifies to
\begin{equation}\label{Delta22}\Delta_2=\Ein(l+s,[i+1])-\Sin(l+s,[i+1])+\Sin(l+r+s,[i])-1.
\end{equation}

A south arrival at $(l+s,[i+1])$ is the same as a south departure from $(l+r+s,[i]),$ and an east arrival at $(l+s,[i+1])$ is the same as an east departure from $(l,[i]),$ so
\begin{equation}\label{Delta23}\Delta_2=\Eout(l,[i])-\left(\Sout-\Sin\right)(l+r+s,[i])-1.
\end{equation}%\begin{equation}\label{westinout}\Sin(l+s,[i+1])=\Sout(l+r+s,[i+1])\end{equation} and \begin{equation}\label{northinout}\Ein(l+s,[i+1])=\Eout(l,[i]).\end{equation} 
By Corollary~\ref{noWoutNin(l,[i])}, all edges leaving $(l,[i])$ are east edges and all edges arriving are south edges. The same number of edges arrive and leave, so $\Eout(l,[i])=\Sin(l,[i]).$ So,
 \begin{equation}\label{EqualDeltas}\Delta_2=(\Sin-\Sout)(l+r+s,[i]) +\Sin(l,[i])-1=\Delta_1.\end{equation}
Substituting \eqref{EqualDeltas} into \eqref{differencefromM} completes the proof.\end{proof}

It remains to check that $\critplusxc(I_{r,s,c}(\lambda))=\critminusxc(\lambda)$ and $\critminusxc(I_{r,s,c}(\lambda))=\critplusxc(\lambda)$.

First, we make some  make some straightforward but important observations about $M_{r,s,c}(\lambda)$ and winding numbers in Proposition~\ref{winding numbers}. Then, we apply these to the first arrival tree to prove some formulae about distances between consecutive vertices in the $(r,s,c)$-tour with respect to the first arrival tree, depending on whether the vertex is eastern, southern, or a switch in Proposition~\ref{di-di-1}. Finally, we apply these to proving $\critplusxc(I_{r,s,c}(\lambda))=\critminusxc(\lambda)$ and $\critminusxc(I_{r,s,c}(\lambda))=\critplusxc(\lambda)$ in Proposition~\ref{critexchange}.

\begin{prop}\label{winding numbers}
Let $(v,[i])$ and $(w,[j])$ be two vertices of $M_{r,s,c}(\lambda)$, and let $p_1$ and $p_2$ be directed paths between $(v,[i])$ and $(w,[j])$. Suppose $p_1$ is given by the sequence of vertices $(v,[i])=(v_0,[i_0]),\ldots, (v_{|p_1|,[i_0+|p_1|]})=(w,[j])$. Then,
\begin{enumerate}
\item $|p_1|-|p_2|$ is divisible by $\lcm(c,r+s)$.
\item Let $(v,[i])$ be $m$ lattice steps below the upper boundary of the cylinder, and let $|p_1|=q\lcm(c,r+s)+u$ where $-m < u\leq \lcm(c,r+s)-m$. The winding number of $p_1$ is $q$. 
\end{enumerate}
\end{prop}
\begin{proof} The first point follows from Proposition~\ref{lattice points}: $p_1$ and $p_2$ are lattice paths from points $(x_1,y_1)$ and $(x_1+ar,y_1-as)$ respectively to points $(x_2,y_2)$ and $(x_2+br,y_2-bs)$ respectively, where $\lcm(c,r+s)$ divides $a(r+s)$ and $b(r+s)$. We have that $|p_1|=x_2-x_1+y_1-y_2$ and $|p_2|=x_2+br-x_1-ar+y_1-as-y_2+bs$, so $|p_1|-|p_2|=(r+s)(a-b)$, which is divisible by $\lcm(c,r+s)$.

The second point follows because as we trace out a directed path, the value of $x-y$ moves cyclically through the residue classes modulo $\lcm(c,r+s)$, incrementing by 1 with each step.
\end{proof}

\begin{prop}\label{di-di-1} Let $(k,[0])=v_0,v_1\ldots,v_{(r+s)k_1}=(k,[0])$ be the vertices visited, in order, by the $(r,s,c)$-tour, corresponding to the section of the boundary of $\lambda$ between $(0,k_1s)$ and $(k_1r,0)$. Let $d_i$ denote the distance in the first arrival tree $T$ from $(k,[0])$ to $v_i$. 
\begin{enumerate}
    \item If $v_i$ is a switch, or if there is a copy of the edge $(v_{i-1},v_i)$ in $T$, then $d_i-d_{i-1}=1$.
    \item If $v_i$ is an eastern vertex and there is no copy of $(v_{i-1},v_i)$ in $E(T)$, then $d_i-d_{i-1}=1+\lcm(c,r+s)$.
    \item If $v_i$ is a southern vertex and there is no copy of $(v_{i-1},v_i)$ in $E(T)$, then $d_i-d_{i-1}=1-\lcm(c,r+s)$.
\end{enumerate}
\end{prop}
\begin{proof}
Write $p_i$ for the path in $T$ from $(k,[0])$ to $v_i$, so that $|p_i|=d_i.$ The first point follows immediately from the definition of a switch and the definition of $T$. 

In general, the winding number of a vertex $v$ is the same as the winding number of the last vertex on the upper boundary strip that $T$ before $v$. So, drawing $T$ on the cylinder and then forgetting the identification of the two boundary lines, the connected components form sets of vertices of equal winding number. 

Moreover, if $\wind(p_i)=\wind(p_{i-1})$, then by the second part of Proposition~\ref{winding numbers}, $||p_i|-|p_{i-1}||<\lcm(c,r+s).$ Since there is a path of length 1 (not necessarily in $T$) connecting $v_{i-1}$ and $v_i$, then by the first part of Proposition~\ref{winding numbers}, $|p_i|-|p_{i-1}|\equiv 1\pmod{\lcm(c,r+s)}$. Therefore, $d_i-d_{i-1}=1,$ so $v_i$ is a switch or there is a copy of $(v_{i-1},v_{i})$ in $E(T)$.

For 2 and 3, we first prove that as we scan southwest along the upper boundary strip, the winding numbers of the paths from $(k,[0])$ to the vertices on the strip weakly increase. We proceed by induction.

\begin{center}
\begin{tikzpicture}
\draw (0,0)--(-4,-1);
\draw (0,-2)--(-4,-3);

\node[above left] at (-3.5,-0.875) {$B$};
\node[above] at (-3.25,-0.8125) {$A$};
\draw[fill=black] (-3.5,-0.875) circle[radius=0.07];
\draw[fill=black] (-3.25,-0.8125) circle[radius=0.07];

\draw[blue] (-2.75,-0.675)--(-1.125,-1.275)--(-0.25,-2.0625);
\draw[red] (-1.75,-0.425)--(-0.5,-2.125);
\draw[fill=black] (-1.125,-1.275) circle[radius=0.07];
\node[above right] at (-1.125,-1.275) {$C$};

\node[below left] at (-0.5,-2.125) {$B$};
\node[below right] at (-0.25,-2.0625) {$A$};
\draw[fill=black] (-0.5,-2.125) circle[radius=0.07];
\draw[fill=black] (-0.25,-2.0625) circle[radius=0.07];
\node[above] at (-2.75,-0.675) {$A'$};
\node[above] at (-1.75,-0.425) {$B'$};
\draw[fill=black] (-2.75,-0.675) circle[radius=0.07];
\draw[fill=black] (-1.75,-0.425) circle[radius=0.07];
\end{tikzpicture}
\end{center}

Suppose $A$ and $B$ are vertices on the upper boundary strip and $B$ is southwest of $A$, and let $p_A$ and $p_B$ be the paths in $T$ from $(k,[0])$ to $A$ and $B$ respectively. We will show $\wind(p_B)\geq \wind(p_A)$. If $A=(k,[0])$ then we are done, so suppose not. There is a copy of both $A$ and $B$ on the lower boundary strip, with $B$ still southwest of $A$. Moreover, $p_A$ and $p_B$ run from points $A'$ and $B'$ respectively on the upper boundary strip to $A$ and $B$, where we possibly have $A'=B'.$ However, $A'$ cannot be strictly southwest of $B'$, as otherwise $p_A$ and $p_B$ would have to cross at a vertex $C$, introducing a cycle from $(k,[0])$ following $p_A$ to $C$ and then following $p_B$ back to $(k,[0])$. Let $p_B'$ and $p_A'$ be $p_B$ and $p_A$ shortened to finish at $B'$ and $A'$ respectively. Then, by strong induction, $\wind(p_B')\geq \wind(p_A')$. Adding 1 to both sides, $\wind(p_B)\geq \wind(p_A).$

Now, in the case that $v_i$ is eastern, and there is no copy of $(v_{i-1},v_i)$ in $E(T)$, $(v_{i-1},v_i)$ must be a south edge, and $v_{i-1}$ and $v_i$ lie in different connected components. Since $\wind(p_i)\not=\wind(p_{i-1}),$ $\wind(p_i)>\wind(p_{i-1}).$ Let $D$ and $E$ be the last vertices on the upper boundary strip on $p_i$ and $p_{i-1}$ respectively.

\begin{center}
\begin{tikzpicture}[scale=0.5]
\draw (0,0)--(-5,-5);

\node[above] at (-1,-1) {$D$};
\node[above] at (-3,-3) {$E$};
\draw[fill=black] (-1,-1) circle[radius=0.07];
\draw[fill=black] (-3,-3) circle[radius=0.07];
\draw[fill=black] (1,-4) circle[radius=0.07];
\draw[fill=black] (1,-5) circle[radius=0.07];
\draw[->,blue, densely dotted] (1,-4)--(1,-4.5);
\draw[blue, densely dotted] (1,-4.5)--(1,-5);
\draw (-1,-1)--(0,-1)--(1,-1)--(1,-2)--(1,-3)--(1,-4);
\draw (-3,-3)--(-2,-3)--(-2,-4)--(-2,-5)--(-1,-5)--(1,-5);
\node[right] at (1,-4) {$v_{i-1}$};
\node[right] at (1,-5) {$v_i$};
\end{tikzpicture}

\end{center}

Since all paths in $T$ have vertices at lattice points and do not intersect with each other, there can be no path that starts at a vertex on the upper boundary strip between $E$ and $D$ that crosses all the way to the lower boundary strip. Hence, the copy of $E$ on the lower boundary strip either lies in the same connected component as $D$ or in a component northeast of $D.$ So, $\wind(p_i)=\wind(p_{i-1})+1.$ Let $q_i$ be the path obtained by extending $p_{i-1}$ by the south edge $(v_{i-1},v_i)$. Then $|q_i|=d_{i-1}+1.$ The second part of Proposition~\ref{winding numbers} tells us that $|p_i|$ and $|q_i|$ agree modulo $\lcm(c,r+s)$ and therefore $d_i=d_{i-1}+1+\lcm(c,r+s)$.

An analogous argument proves the third formula.
\end{proof}

\begin{prop}\label{critexchange} Let $\lambda$ be a partition. Then
\begin{equation}\critplusxc(I_{r,s,c}(\lambda))=\critminusxc(\lambda)\end{equation} and \begin{equation}\critminusxc(I_{r,s,c}(\lambda))=\critplusxc(\lambda).\end{equation}
\end{prop}
\begin{proof}
 We will check that $\critplusxc(\lambda)=(\critplusxc+\critminusxc)(M_{r,s,c}(\lambda))-\critplusxc(I_{r,s,c}(\lambda)).$

Recall that $\critplusxc$ counts the total number of inversions in the arrival word at vertices in $M_{r,s,c}(\lambda).$ Suppose the arrival word at vertex $(v,[i])$ has $a$ south edges and $b$ east edges. 
If $v$ is a switch, then $I$ reverses the arrival word at $(v,[i])$, so the pairs of $S,E$ edges that contribute to $\critplusxc(I(\lambda))$ are exactly those that do not contribute to $\critplusxc(\lambda)$, so the contributions over $I(\lambda)$ and $\lambda$ at $(v,[i])$ sum to $ab$.

Note that if $(v,[i])\in \Ea$  then $I(\lambda)$ has inversions in the arrival word at $(v,[i])$ using the first $E$ and any $S$ in the arrival word, and then any other pair of south and east edges contribute to $I(\lambda)$ if and only if they do not contribute to $\lambda$, so the two contributions sum to $ab+a.$ 
Similarly, if $(v,[i])\in \So$ then the contributions sum to $ab-b$.   Hence, we have that the total $\critplusxc(\lambda)+\critplusxc(I_{r,s,c}(\lambda))$ can be written as $S_1+S_2+S_3$ where $$S_1 = \sum_{(v,[i])\text{is a switch}}\Ein(v,[i])\Sin(v,[i])$$
$$S_2=\sum_{(v,[i])\in \Ea}\Ein(v,[i])\Sin(v,[i])+\Sin(v,[i])$$
$$S_3=\sum_{(v,[i])\in \So}\Ein(v,[i])\Sin(v,[i])-\Ein(v,[i]).$$

Now, note first that no vertex $(v,[i])$ with $v>k$ contributes to any of these sums. Indeed, no such vertex is a switch, and the arrival word at any such $(v,[i])$ has length $0,1$ or $2$, containing at most one $S$ and at most one $E$. If the arrival word is empty there is nothing to prove. If the arrival word is $E$ then the vertex is eastern, and $\Ein(v,[i])\Sin(v,[i])+\Sin(v,[i])=0.$ If the arrival word is $S$ then the vertex is southern and $\Ein(v,[i])\Sin(v,[i])-\Ein(v,[i])=0.$ The only other possible arrival word is $SE$, in which case the vertex is southern and $\Ein(v,[i])\Sin(v,[i])-\Ein(v,[i])=1-1=0.$
So, we may restrict our sum to vertices $(v,[i])$ with $v\leq k$. 

Proposition~\ref{crittotdiff} proves~\eqref{crittot}, $$(\critplusxc+\critminusxc)(M_{r,s,c}(\lambda))=\sum_{v=0}^k\sum_{i=0}^{c-1}\Ein(v,[i])\Sin(v,[i])-\left\lfloor\frac{k_1(s+r)}{\lcm(c,r+s)}\right\rfloor.$$

We wish to show that~\eqref{crittot} is equal to $S_1+S_2+S_3$, and therefore it suffices to check that

\begin{equation}\label{toprove}\sum_{(v,[i])\in\So}\Ein(v,[i])-\sum_{(v,[i])\in\Ea}\Sin(v,[i])=\left\lfloor\frac{k_1(s+r)}{\lcm(c,s+r)}\right\rfloor.\end{equation}

Note that the east edges entering southern vertices and the south edges entering eastern vertices are exactly the edges in $M_{r,s,c}(\lambda)$ arriving at non-switch vertices that are \textit{not} a copy of an edge in the first arrival tree $T$. Hence, if we let $n_0$ denote the number of edges $e$ entering vertices $(v,[i])$ with $v\leq k$ such that either 
\begin{itemize}
    \item $(v,[i])$ is a switch, or
\item $(v,[i])$ is not a switch and there is a copy of $e$ in the first arrival tree $T$,
\end{itemize} then \begin{equation}\label{usetogglability}n_0+\sum_{(v,[i])\in\So}\Ein(v,[i])+\sum_{(v,[i])\in\Ea}\Sin(v,[i])=k_1(r+s).\end{equation}

Now, let $(k,[-k_1s])=v_0,v_1\ldots,v_{(r+s)k_1}= (k,[k_1r])$ be the vertices visited, in order, possibly with repetition, by the $(r,s,c)$-tour. Let $d_i$ denote the distance in the first arrival tree from $(k,[-k_1s])$ to $v_i$. Now $c\mid k_1$ by assumption, and thus $c\mid(r+s)k_1$, so we have that \begin{equation}\label{telescoper}0=d_{(r+s)k_1}=\sum_{i=1}^{(r+s)k_1} d_i-d_{i-1}.\end{equation} Substituting the formulae for $d_i-d_{i-1}$ proven in Proposition~\ref{di-di-1} into ~\eqref{telescoper} and writing $l$ for $\lcm(c,r+s)$, 
\begin{equation}\label{stratifybyvertextype}n_0+(1+l)\sum_{(v,[i])\in\Ea}\Sin(v,[i])+(1-l)\sum_{(v,[i])\in\So}\Ein(v,[i])=0.\end{equation}

Subtracting~\eqref{stratifybyvertextype} from~\eqref{usetogglability} gives  \begin{equation}k_1(r+s)=\lcm(c,r+s)\left(\sum_{(v,[i])\in\So} \Ein(v,[i])-\sum_{(v,[i])\in\Ea} \Sin(v,[i])\right).\end{equation} Now, since $k$ is divisible by $rsc$, $k_1=\frac{k}{rs}$ is divisible by $c$, so $\lcm(c,r+s)$ divides $k_1(r+s)$. Therefore, $$\left\lfloor\frac{k_1(r+s)}{\lcm(c,r+s)}\right\rfloor=\frac{k_1(r+s)}{\lcm(c,r+s)}=\sum_{(v,[i])\in\So} \Ein(v,[i])-\sum_{(v,[i])\in\Ea} \Sin(v,[i])),$$ which is~\eqref{toprove}, which completes the proof.
\end{proof}

\subsection{Extended Example}
Let $c=2$ and $n=7$, and $\mu=(2,1).$ Then \begin{equation}\Par^2_{\mu}(7)=\{(6,1),(4,3),(4,1,1,1),(2,2,2,1),(2,1,1,1,1,1)\}.\end{equation}
\begin{figure}[ht]
\begin{center}
    \begin{tikzpicture}\begin{scope}[yscale=1,xscale=-1,rotate=90,scale=0.4]
    \draw[white, fill=yellow] (0,5)--(1,5)--(1,4)--(0,4)--(0,5);
    \draw[white, fill=yellow] (0,3)--(1,3)--(1,2)--(0,2)--(0,3);
    \draw (0,6)--(0,0);
    \draw (1,6)--(1,0);
    \draw (2,1)--(2,0);
    \draw (0,6)--(1,6);
    \draw (0,5)--(1,5);
    \draw (0,4)--(1,4);
    \draw (0,3)--(1,3);
    \draw (0,2)--(1,2);
    \draw (0,1)--(2,1);
    \draw (0,0)--(2,0);
    \end{scope}\end{tikzpicture}
      \begin{tikzpicture}\begin{scope}[yscale=1,xscale=-1,rotate=90,scale=0.4]
    \draw[white, fill=yellow] (0,2)--(1,2)--(1,1)--(0,1)--(0,2);
    \draw[white, fill=yellow] (1,2)--(2,2)--(2,1)--(1,1)--(1,2);
    \draw (0,4)--(0,0);
    \draw (1,4)--(1,0);
    \draw (2,3)--(2,0);
    \draw (0,4)--(1,4);
    \draw (0,3)--(2,3);
    \draw (0,2)--(2,2);
    \draw (0,1)--(2,1);
    \draw (0,0)--(2,0);
    \end{scope}\end{tikzpicture}
     \begin{tikzpicture}\begin{scope}[yscale=1,xscale=-1,rotate=90,scale=0.4]
    \draw[white, fill=yellow] (0,3)--(1,3)--(1,2)--(0,2)--(0,3);
    \draw[white, fill=yellow] (2,1)--(3,1)--(3,0)--(2,0)--(2,1);
    \draw (0,4)--(0,0);
    \draw (1,4)--(1,0);
    \draw (2,1)--(2,0);
    \draw (3,1)--(3,0);
    \draw (4,1)--(4,0);
    \draw (0,4)--(1,4);
    \draw (0,3)--(1,3);
    \draw (0,2)--(1,2);
    \draw (0,1)--(4,1);
    \draw (0,0)--(4,0);\end{scope}\end{tikzpicture}
     \begin{tikzpicture}\begin{scope}[yscale=1,xscale=-1,rotate=90,scale=0.4]
    \draw[white, fill=yellow] (1,2)--(2,2)--(2,1)--(1,1)--(1,2);
    \draw[white, fill=yellow] (1,1)--(2,1)--(2,0)--(1,0)--(1,1);
    \draw (0,2)--(0,0);
    \draw (1,2)--(1,0);
    \draw (2,2)--(2,0);
    \draw (3,2)--(3,0);
   \draw (4,1)--(4,0);
    \draw (0,2)--(3,2);
    \draw (0,1)--(4,1);
    \draw (0,0)--(4,0);\end{scope}\end{tikzpicture}
     \begin{tikzpicture}\begin{scope}[yscale=1,xscale=-1,rotate=90,scale=0.4]
         \draw[white, fill=yellow] (5,0)--(5,1)--(4,1)--(4,0)--(5,0);
    \draw[white, fill=yellow] (3,0)--(3,1)--(2,1)--(2,0)--(3,0);
    \draw (0,2)--(0,0);
    \draw (1,2)--(1,0);
    \draw (2,1)--(2,0);
    \draw (3,1)--(3,0);
    \draw (4,1)--(4,0);
    \draw (5,1)--(5,0);
    \draw (6,1)--(6,0);
    \draw (0,2)--(1,2);
    \draw (0,1)--(6,1);
    \draw (0,0)--(6,0);\end{scope}\end{tikzpicture}    
\end{center}
\caption{the partitions in $\Par^2_{\mu}(7)$ with boxes of even hook length coloured yellow.}\label{extex}
\end{figure}
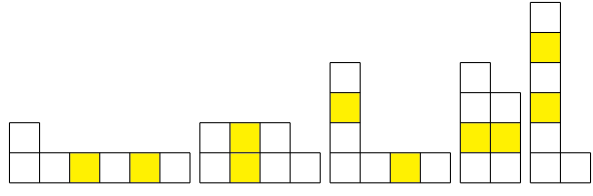

For the shaded cells, the set of values of $\frac{a(\square)}{l(\square)+1}$ is $\left\{3,1,0,\frac{1}{3}\right\},$ and the set of values of $\frac{a(\square)+1}{l(\square)}$ is $\left\{\infty,3,1,\frac{1}{3}\right\}$. So, the critical rationals are $\left\{0,\frac{1}{3},1,3,\infty\right\}.$

In this example, we will verify that $$\sum_{\lambda\in\Par^2_{\mu}(7)}t^{h_{4,2}^+(\lambda)}=\sum_{\Par^2_{\mu}(7)}t^{\lambda^{2*}_{\square}}.$$
Recall
\begin{equation} h_{4,2}^+(\lambda)=\left|\left\{\square \in \lambda: 2\mid  h(\square) \text{ and } \frac{a(\square)}{l(\square)+1}\leq 4<\frac{a(\square)+1}{l(\square)}\right\}\right|.\end{equation}

From our computation of the critical rationals, given that $2\mid h(\square)$ for some box in a partition $\lambda\in\Par^2_{\mu}(7),$ $4<\frac{a(\square)+1}{l(\square)}$ if and only if $3< \frac{a(\square)+1}{l(\square)}$, and $\frac{a(\square)}{l(\square)+1}\leq 4$ if and only if $\frac{a(\square)}{l(\square)+1}\leq 3.$ So, $h_{4,2}^+(\lambda)=h_{3,2}^+(\lambda)$. 
Now we use $I_{3,1,2}:\Par^2_{\mu}(7)\to\Par^2_{\mu}(7)$. Because $\operatorname{mid}_{3,2}(\lambda)=\operatorname{mid}_{3,2}(I_{3,1,2}(\lambda))$ and $\operatorname{crit}^{\pm}_{3,2}(\lambda)=\operatorname{crit}^{\mp}_{3,2}(\lambda)$, $I_{3,1,2}$ is a bijection exchanging $h_{3,2}^+$ and $h_{3,2}^-,$ so
$$\sum_{\lambda\in\Par^2_{\mu}(7)}t^{h_{3,2}^+(\lambda)}=\sum_{\lambda\in\Par^2_{\mu}(7)}t^{h_{3,2}^-(\lambda)}.$$

We now explicitly compute $I_{3,1,2}(\lambda)$ for $\lambda=(6,1).$

The diagram of $(6,1)$ lies below the line $3y+x=9.$ So, we choose the smallest value $k\geq 9$ such that $3\times2\times1\mid k,$ $k=12.$ Then, $k_1=\frac{12}{3}=4.$ 

The $(r,s,c)$-tour of $M_{3,1,2}((6,1))$ is defined by the following family of arrival words.
\begin{center}
    
\begin{tabular}{cccccc}
(4,[0])&S&(5,[1])&E&(6,[0])&SSE\\
(7,[1])&EEE&(8,[0])&E&(9,[1])&SEE
\end{tabular}
\end{center}
and for $w>9,$ \begin{equation}(w,[i])_a=\begin{cases} 
      SE & 3\mid w, w\equiv \frac{-w}{3}\equiv i\pmod{2} \\
      E &  2\mid\left(w-i\right)\text{ and either }3\nmid w\text{ or }2\nmid(\frac{-w}{3}-i)\\ 
      S & 3\mid w, 2\mid\left(\frac{-w}{3}-i\right), 2\nmid(w-i) \\
      \text{empty} &\text{otherwise} \\
   \end{cases}.\end{equation}
   The multigraph is given in Figure~\ref{multex1} with the edges in the first arrival tree in bold.
   \begin{figure}[ht]
\begin{center}
\begin{tikzpicture}\begin{scope}[yscale=0.9,xscale=-0.9,rotate=90,scale=2.25]
\draw (1.5,-2.5)--(5,1);
\draw (0.5,0.5)--(3.5,3.5);
\draw[fill=red] (3.85,0.3)--(4.15,0.3)--(4.15,-0.3)--(3.85,-0.3)--(3.85,0.3);
\draw[fill=red] (2.85,0.3)--(3.15,0.3)--(3.15,-0.3)--(2.85,-0.3)--(2.85,0.3);
\draw[fill=red] (1.85,0.3)--(2.15,0.3)--(2.15,-0.3)--(1.85,-0.3)--(1.85,0.3);
\draw[fill=cyan] (1.85,1.3)--(2.15,1.3)--(2.15,0.7)--(1.85,0.7)--(1.85,1.3);
\draw[fill=red] (0.85,1.3)--(1.15,1.3)--(1.15,0.7)--(0.85,0.7)--(0.85,1.3);
\draw[fill=red] (1.85,-1.7)--(2.15,-1.7)--(2.15,-2.3)--(1.85,-2.3)--(1.85,-1.7);
\draw[fill=cyan] (1.85,-0.7)--(2.15,-0.7)--(2.15,-1.3)--(1.85,-1.3)--(1.85,-0.7);
\draw[fill=cyan] (1.85,2.3)--(2.15,2.3)--(2.15,1.7)--(1.85,1.7)--(1.85,2.3);
\draw[fill=cyan] (2.85,-0.7)--(3.15,-0.7)--(3.15,-1.3)--(2.85,-1.3)--(2.85,-0.7);
\draw[fill=cyan] (2.85,0.7)--(3.15,0.7)--(3.15,1.3)--(2.85,1.3)--(2.85,0.7);
\draw[fill=cyan] (2.85,1.7)--(3.15,1.7)--(3.15,2.3)--(2.85,2.3)--(2.85,1.7);
\draw[fill=red] (2.85,2.7)--(3.15,2.7)--(3.15,3.3)--(2.85,3.3)--(2.85,2.7);
\draw[very thick,->] (3.85,0)--(3.15,0);
\draw[very thick,->] (2.85,0.05)--(2.15,0.05);
\draw[->] (2.85,-0.05)--(2.15,-0.05);
\draw[very thick,->] (1.85,1)--(1.15,1);
\draw[very thick,->] (2.05,0.3)--(2.05,0.7);
\draw[->] (1.95,0.3)--(1.95,0.7);
\draw[very thick,->] (1.95,1.3)--(1.95,1.7);
\draw[->] (2.05,1.3)--(2.05,1.7);
\draw[very thick,->] (2,-1.7)--(2,-1.3);
\draw[->] (2,-0.7)--(2,-0.3);
\draw[->] (3.05,-0.7)--(3.05,-0.3);
\draw[->] (2.95,-0.7)--(2.95,-0.3);
\draw[very thick,->] (3,0.3)--(3,0.7);
\draw[very thick,->] (3,1.3)--(3,1.7);
\draw[very thick,->] (3.85,3)--(3.15,3);
\node[above] at (3.8,3) {$\vdots$};
\draw[->] (3,2.3)--(3,2.7);
%\draw[->] (1.85,1)--(1.15,1);
\node at (4,0) {(12,[0])};
\node at (3,0) {(9,[1])};
\node at (2,0) {(6,[0])};
\node at (2,1) {(7,[1])};
\node at (1,1) {(4,[0])};
\node at (2,-2) {(4,[0])};
\node at (2,-1) {(5,[1])};
\node at (2,2) {(8,[0])};
\node at (3,-1) {(8,[0])};
\node at (3,1) {(10,[0])};
\node at (3,2) {(11,[1])};
\node at (3,3) {(12,[0])};
\end{scope}\end{tikzpicture}
\end{center}

\caption{$M_{3,1,2}((6,1))$ with the edges of the first arrival tree in bold.}\label{multex1}
\end{figure}

After applying $I_{3,1,2}$ the arrival words are
\begin{center}
\begin{tabular}{cccccc}
(4,[0])&S&(5,[1])&E&(6,[0])&SSE\\
(7,[1])&EEE&(8,[0])&EE&(9,[1])&SEE
\end{tabular}
\end{center}
with all arrival words at $(w,[i])$ with $w>9$ unchanged. These arrival words correspond to the partition $(4,3).$ So, $h_{3,2}^+((6,1))=h_{3,2}^-((4,3)).$

From our computation of the critical rationals, given that $2\mid h(\square)$ for some box in a partition $\lambda\in\Par^2_{\mu}(7),$ $3\leq\frac{a(\square)+1}{l(\square)}$ if and only if $1< \frac{a(\square)+1}{l(\square)}$, and $\frac{a(\square)}{l(\square)+1} < 3$ if and only if $\frac{a(\square)}{l(\square)+1}\leq 1.$ So, $h_{3,2}^-(\lambda)=h_{1,2}^+(\lambda)$ for all $\lambda\in\Par^2_{\mu}(7)$. Now, $I_{1,1,2}$ exchanges $h_{1,2}^+$ and $h_{1,2}^-$, and $I_{1,1,2}((4,3))=(2,2,2,1)$, so $h_{4,2}^+((6,1))=h_{1,2}^+((2,2,2,1)).$ Using the same logic again $h_{1,2}^+(\lambda)=h_{\frac{1}{3},2}^-(\lambda)$ for each $\lambda\in\Par^2_{\mu}(7).$ Using $I_{1,3,2}$, $I_{1,3,2}((2,2,2,1))=(2,1,1,1,1,1),$ so $h_{4,2}^+((6,1))=h_{\frac{1}{3},2}^{-}((2,1,1,1,1,1))$. Finally, for any partition $\lambda\in\Par^2_{\mu}(7),$ $\frac{1}{3}\leq\frac{a(\square)+1}{l(\square)}$ if and only if $0<\frac{a(\square)+1}{l(\square)}$, and $\frac{a(\square)}{l(\square)+1}< \frac{1}{3}$ if and only if $a(\square)=0,$ if and only if $\frac{a(\square)}{l(\square)+1}\leq 0,$ so $h^-_{\frac{1}{3},2}(\lambda)=h^+_{0,2}(\lambda).$

Therefore, since $$I_{1,3,2}\circ I_{1,1,2}\circ I_{3,1,2}((6,1))=(2,1,1,1,1,1),$$ we have that $h^+_{4,2}(6,1)=h^+_{0,2}(2,1,1,1,1,1).$
For the other partitions in $\Par^2_{\mu}(7),$
$$I_{1,3,2}\circ I_{1,1,2}\circ I_{3,1,2}(4,3)=I_{1,3,2}\circ I_{1,1,2}((6,1))=I_{1,3,2}(2,1,1,1,1,1)=(2,2,2,1),$$
$$I_{1,3,2}\circ I_{1,1,2}\circ I_{3,1,2}(4,1,1,1)=I_{1,3,2}\circ I_{1,1,2}((4,1,1,1))=I_{1,3,2}(4,1,1,1)=(4,1,1,1).$$
$$I_{1,3,2}\circ I_{1,1,2}\circ I_{3,1,2}(2,2,2,1)=I_{1,3,2}\circ I_{1,1,2}((2,2,2,1))=I_{1,3,2}(4,3)=(4,3).$$
$$I_{1,3,2}\circ I_{1,1,2}\circ I_{3,1,2}(2,1,1,1,1,1)=I_{1,3,2}\circ I_{1,1,2}((2,1,1,1,1,1))=I_{1,3,2}(6,1)=(6,1).$$

Hence we can verify the equidistribution of $h^+_{x,2}$ with $h^-_{x,2}$ over $\Par_{\mu}^2(7)$ for each $x\in\mathbb{R}_{>0},$ thus verifying Theorem 3.3 in this case.

\section{Further Work}
We note here that Problem 8.9 in \cite{Walsh} may be amenable to similar techniques.

%%%%%%%%%%%%%%%%%%%%%%%%%%%%%%%%%%%%%%%%%%%%%%%%%
% The following optional unnumbered section is where you put personal acknowledgements,
% research grant support, and similar things.  Do not put them on the front page.
\subsection*{Acknowledgements}
The author is indebted to her PhD supervisor Paul Johnson for suggesting the problem and for many helpful conversations. The author is funded by an EPSRC studentship\footnote{UKRI grant number EP/R513313/1.}.

%BIBLIOGRAPHY
% You do not have to use the same format for your references, but 
%    include everything in this file.
% If you use BibTeX to create a bibliography, copy the .bbl file into here.
% We recommend you use \doi{...} and \arxiv{...} like the examples below,
% as they give a short display form with an active link to the full url.

\end{document}